\documentclass[11pt]{amsart}
\usepackage[latin1]{inputenc}
\usepackage{amsmath}
\usepackage{amsfonts}
\usepackage{amssymb}
\usepackage{graphicx}
\usepackage{fourier}
\usepackage{hyperref}
\usepackage{enumerate}
\usepackage{esint}
\usepackage{bm}
\usepackage{ulem}
\usepackage{xcolor} 
\usepackage{verbatim}        
\usepackage{bbm}

\newtheorem{theorem}{Theorem}[section]
\newtheorem{lemma}[theorem]{Lemma}

\theoremstyle{definition}
\newtheorem{definition}[theorem]{Definition}

\newtheorem{corollary}[theorem]{Corollary}
\theoremstyle{remark}
\newtheorem{remark}[theorem]{Remark}
\newtheorem*{remark*}{Remark}

\newcommand{\inner}[1]{\left\langle#1\right\rangle}
\newcommand{\norm}[1]{\left\lVert#1\right\rVert}
\newcommand{\abs}[1]{\left\lvert#1\right\rvert}
\newcommand{\pa}[1]{\left( #1 \right)}
\newcommand{\rpa}[1]{\left[ #1 \right]}
\newcommand{\br}[1]{\left\lbrace #1\right\rbrace}
\newcommand{\R}{\mathbb{R}}

\newcommand{\N}{\mathbb{N}}
\newcommand{\I}{\mathcal{I}}
\newcommand{\PP}{\mathbf{P}}
\newcommand{\XX}{\mathbf{X}}
\newcommand{\YY}{\mathbf{Y}}
\newcommand{\W}{\mathbf{W}}
\numberwithin{equation}{section}

\def\1{\mathbbm{1}}

\definecolor{bostonuniversityred}{rgb}{0.8, 0.0, 0.0}
\newcommand{\amit}{\textcolor{black}} 
 
\definecolor{byzantium}{rgb}{0.44, 0.16, 0.39}

\newcommand{\D}{{\bf D}}
\newcommand{\X}{{\bf X}}
\newcommand{\Y}{{\bf Y}}
\newcommand{\II}{{\bf I}}
\newcommand{\C}{{\bf C}}
\newcommand{\B}{{\bf B}}
\newcommand{\A}{{\bf A}}
\newcommand{\K}{{\bf K}}

\newcommand{\tr}{\operatorname{tr}}
\renewcommand{\span}{\operatorname{span}}

\renewcommand{\Re}{\operatorname{Re}}

\newcommand{\hess}{\operatorname{Hess}}

\title[Generalised Fisher Information Approach to Defective Fokker-Planck Equations]{Generalised Fisher Information in Defective Fokker-Planck Equations}
\author{Anton Arnold$^1$, Amit Einav$^2$ \& Tobias W\"ohrer$^1$}
\date{}
\address{$^1$Vienna University of Technology, Institute of Analysis and Scientific Computing, Wiedner Hauptstr. 8--10, A-1040 Wien, Austria}
\email{\{anton.arnold, tobias.woehrer\}@tuwien.ac.at}

\address{$^2$Durham University, School of Mathematical Sciences, Lower Mountjoy Stockton Road, DH1 3LE, Durham, United Kingdom}
\email{amit.einav@durham.ac.uk}
\thanks{The authors gratefully acknowledge the support of the Hausdorff Research Institute for Mathematics
	(Bonn), through the Junior Trimester Program on Kinetic Theory. The first author was partially supported by the FWF-funded SFB \#F65. 
The second author was supported by the Austrian Science Fund (FWF) grant M 2104-N32.
The first and the third authors were partially supported by the FWF-doctoral school ``Dissipation
and dispersion in non-linear partial differential equations''. The third author conducted part of the research at FAU Erlangen, Department of Data Science and was funded in part by
the Austrian Science Fund (FWF) 10.55776/J4681.\\
The authors would like to thank the reviewers of this work for their helpful and constructive comments.}

\begin{document}

\maketitle

\begin{abstract}
	The goal of this work is to introduce and investigate a \textit{generalised Fisher Information} in the setting of linear Fokker-Planck equations. This functional, which depends on two functions instead of one, exhibits the same decay behaviour as the standard Fisher information, and allows us to investigate different parts of the Fokker-Planck solution via an appropriate decomposition. Focusing almost exclusively on Fokker-Planck equations with constant drift and diffusion matrices, we will use a modification of the well established Bakry-Emery method with this newly defined functional to provide an alternative proof to the sharp long time behaviour of relative entropies of solutions to such equations when the diffusion matrix is positive definite and the drift matrix is defective. This novel approach is different to previous techniques and relies on minimal spectral information on the Fokker-Planck operator, unlike the one presented in \cite{AEW18} where powerful tools from spectral theory were needed.
\end{abstract}
	
	\tableofcontents

\section{Introduction}\label{sec:intro}

\subsection{Background}\label{subsec:background}
The Fokker-Planck equation, in its various forms, has been a focal point of attention since its formalisation at the early 20th century. In its heart, the Fokker-Planck equation is influenced by two processes: \textit{Diffusion}, which is governed by a (possibly spatial variable dependent) \textit{positive semi-definite} matrix, and \textit{Drift}, which is generated by a variable dependent vector.\\
In recent years interest in degeneracies in the Fokker-Planck equation has increased. In particular, in the question of how such conditions impact the existence of an equilibrium and the rate of convergence to it.

In this work we will focus our attention on Fokker-Planck equations with linear drift of the form
\begin{equation}\label{eq:fokkerplanck}
\partial_t f(t,x)=Lf(t,x):=\text{div}\pa{\D\nabla f(t,x)+\C xf(t,x)}, \quad\quad t>0, x\in\R^d,
\end{equation}
with appropriate initial conditions $f_0(x)$, where $\D$ and $\C$, the \textit{diffusion} and \textit{drift} matrices respectively, are constant and real-valued.\\
Equations of the form \eqref{eq:fokkerplanck}, particularly in the case of degeneracy in the diffusion, have been extensively investigated in the works of Arnold and Erb (\cite{AE}) and Ottobre, Pavliotis and Pravda-Starov (\cite{OPPS12}). In particular, it was shown in \cite{AE} that in order to be able to perform long time behaviour analysis, and indeed to even have a state of equilibrium, additional conditions on the diffusion and drift matrices must hold. Namely:
\begin{enumerate}[(A)]
\item\label{item:cond_semipositive} $\D$ is positive semi-definite with
$$1\leq \operatorname{rank}\pa{\D} \leq d.$$
\item\label{item:cond_positive_stability} All eigenvalues of $\C$ have positive real part (this is sometimes known as \textit{positive stable}).
\item\label{item:cond_no_invariant_subspace_to_kernel} There exists no $\C^T$-invariant subspace of $\operatorname{ker}\pa{\D}$.
\end{enumerate}
With these conditions at hand, the authors have shown that the equation has a (normalised) unique equilibrium, and continued to investigate the convergence to it with the help of so-called (relative) entropy functionals:
\begin{definition}\label{def-entropy}
	We say that a function $\psi$ is a \textit{generating function for an admissible relative entropy} if, $\psi\in C\pa{[0,\infty)}\cap C^4\pa{\R_+}$, $\psi(1)=\psi^\prime(1)=0$, $\psi^{\prime\prime} > 0$ on $\R_+$ and 
	\begin{equation}\label{eq:psicondition}
		\pa{\psi^{\prime\prime\prime}}^2 \leq \frac{1}{2}\psi^{\prime\prime}\psi^{\prime\prime\prime\prime} \quad \text{on } \R_+,
	\end{equation} 
	where $\R_+:=\pa{0,\infty}$.
	For such a $\psi$, and a non-negative function $f$, we define the \textit{admissible relative entropy} $e_\psi\pa{\cdot|f_\infty}$ as the functional
	\begin{equation}\label{eq:defentropy}
		e_\psi\pa{f|f_\infty}:= \int_{\R^d}\psi\pa{\frac{f(x)}{f_\infty(x)}}f_\infty(x)dx,
	\end{equation}
	where $f_\infty(x)$ is a positive function that represents the equilibrium we expect to converge to. It is conventional to assume that $f$ has the same mass as $f_\infty$, i.e.\ $\int_{\R^d} f(x) dx = \int_{\R^d} f_\infty(x) dx$. Due to the linearity of \eqref{eq:fokkerplanck} we shall assume that $f_\infty$ has mass $1$.
\end{definition}
It is worth to mention that condition \eqref{eq:psicondition} is equivalent to the concavity of $\frac{1}{\psi''}$ on $\R_+$. 
For more information on such functionals we direct the reader to \cite{AMTU01,BaEmD85}.

To be able to state our main theorem, we need to remind ourselves the notions of defectiveness of an eigenvalue:
\begin{definition}
An eigenvalue is \textit{defective} if its geometric multiplicity is strictly less than its algebraic multiplicity. We will call this difference the \textit{defect} of the eigenvalue.
\end{definition}

 Arnold and Erb showed that the rate of convergence to equilibrium, measured with \amit{relative entropies as defined above}, depends on the minimum of the real parts of the eigenvalues of the matrix $\C$ \textit{and} the defects associated to these eigenvalues. Denoting by
 \begin{equation}\label{eq:def_of_mu}
 	\mu:=\min\br{\Re\pa{\lambda}\;|\;\lambda \text{ is an eigenvalue of }\C}
 \end{equation}
and 
\begin{equation}\label{eq:def_of_R_mu}
	R_\mu:=\br{\lambda \; | \; \lambda\text{ is an eigenvalue of }\C\text{ and }\Re\pa{\lambda}=\mu}
\end{equation}
 the following theorem was shown in \cite{AE}:
\begin{theorem}\label{thm:anton_erb}
	Consider the Fokker-Planck equation \eqref{eq:fokkerplanck} with diffusion and drift matrices $\D$ and $\C$ which satisfy Conditions \eqref{item:cond_semipositive}-\eqref{item:cond_no_invariant_subspace_to_kernel} and denote by $f_\infty(x)$ the unit mass equilibrium of the equation\footnote{i.e. $Lf_\infty(x)=0$ and $$\int_{\R^d}f_\infty(x)dx=1.$$}.
	Then for any admissible relative entropy $e_\psi$ and a solution $f(t)$ to \eqref{eq:fokkerplanck} with a non-negative unit mass initial datum $f_0$ and a finite initial entropy we have that:
	\begin{enumerate}[(i)]
		\item If the eigenvalues from the set $R_\mu$, defined in \eqref{eq:def_of_R_mu}, are all non-defective
		then there exists a fixed geometric constant $c$, that does not depend on $f_0$, such that
		\begin{equation} \label{eq:entropydecayAEnodef}
			e_\psi(f(t)|f_\infty) \leq c e_\psi(f_0|f_\infty)e^{-2\mu t}, \quad t\geq 0.
		\end{equation}
			\item  If one of the eigenvalues from the set $R_\mu$, defined in \eqref{eq:def_of_R_mu}, is defective then for any $\epsilon>0$ there exists a fixed geometric constant $c_{\epsilon}$, that does not depend on $f_0$, such that
			\begin{equation} \label{eq:entropydecayAEdef}
				e_\psi(f(t)|f_\infty) \leq c_{\epsilon} e_\psi(f_0|f_\infty)e^{-2(\mu-\epsilon) t},\quad t\geq 0.
			\end{equation}
		\end{enumerate}
\end{theorem}
It is now known, for instance from \cite{AEW18,MoH19}, that \amit{for many types of relative entropies} the correction to \eqref{eq:entropydecayAEnodef} in the defective case is attained by a multiplication by a polynomial of order that equals the largest defect of the eigen\-values in $R_\mu$. This result will be reproved in this work, though with completely new tools and techniques. 
\subsection{Main Results}\label{subsec:main_results}
Our main focus in this paper will be to reprove the sharp long time behaviour of solutions to \eqref{eq:fokkerplanck} by using the (relative) entropies that are generated by the family of functions
\begin{equation}\label{eq:pentropy}
\psi_p(y):=\frac{y^{p}-p(y-1)-1}{p(p-1)} \geq 0, \quad 1<p\leq 2,
\end{equation}
i.e. the functionals 
$$e_p\pa{f|f_\infty}:=\int_{\R^d}\psi_p\pa{\frac{f(x)}{f_\infty(x)}}f_\infty(x)dx = \frac{1}{p(p-1)}\pa{\norm{f}_{L^p\pa{\R^d,f_\infty^{1-p}}}^p-1},$$
where $f_\infty(x)$ is the normalised equilibrium state of the equation, and $\norm{\cdot}_{L^p\pa{\R^d,\omega}}$ is the ($\omega-$)weighted $L^p$ norm. We will call these entropies \textit{$p$-entropies}. While the authors have achieved this goal in previous work, \cite{AEW18}, the technique we use to find the rate of decay in the presented work is an adaptation of the commonly used Bakry-Emery method by an introduction of a new functional arising from a natural ``decomposition'' of the underlying space on which the Fokker-Planck operators acts. Moreover, this technique requires minimal spectral properties of the governing linear operators, unlike \cite{AEW18} where a much more involved spectral study, as well as spectral theorems such as the Gearhart-Pr\"uss theorem, was needed. The price one needs to pay for this simplification, however, is that we deal with a non-degenerate diffusion matrix $\D$, though this restriction can be removed with the use of the recent work \cite{AEW23}. We will comment on it in the appropriate place in this work (see Remark \ref{rem:only_place}).

At times in this work we will veer away from the convention that $f$ is of unit mass and still use the $p$-entropy. We notice that if $f\in L^1_+\pa{\R}$\amit{, i.e. $f$ is in the cone 
	of non-negative functions in $L^1\pa{\R^d}$,} we have the following formula:
	\begin{equation}\label{eq:p_entropy_on_L_1}
		e_p\pa{f|f_\infty} = \frac{1}{p(p-1)}\pa{\norm{f}_{L^p\pa{\R^d,f_\infty^{1-p}}}^p-p\pa{\norm{f}_{L^1\pa{\R^d}}-1}-1} \geq 0.
	\end{equation}
Whenever we write $e_p\pa{f|f_\infty}$ in that case, we will assume implicitly that $f\in L^1_+\pa{\R^d}$. A simple application of H\"older inequality (see \eqref{eq:exponential_controls_L_1} in Appendix \ref{app:b}) shows that $\norm{f}_{L^1\pa{\R^d}}\leq \norm{f}_{L^p\pa{\R^d,f^{1-p}_{\infty}}}$ which implies that $e_p\pa{f|f_\infty}<\infty$ if and only if $\norm{f}_{L^p\pa{\R^d,f^{1-p}_{\infty}}}<\infty$.

Our main theorem for this paper is:
\begin{theorem}\label{thm:main}
Let $\D$ be a positive definite matrix and $\C$ be a positive stable matrix with
 \begin{equation}\nonumber
\mu:=\min\br{\Re\pa{\lambda}\;|\;\lambda\text{ is an eigenvalue of }\;\C}
\end{equation}
and denote by $n$ the maximal defect of the eigenvalues\footnote{To say that $n$ is the maximal defect is equivalent to the statement that the maximal size of Jordan blocks associated to the eigenvalue $\lambda$ is $n+1$.} associated to $\mu$. Let $f_0\in L^1_+\pa{\R^d}$ have a unit mass and finite $p-$entropy for some $p\in (1,2]$. Then, the solution to \eqref{eq:fokkerplanck} with initial data $f_0$ satisfies, for any $\epsilon>0$ and all $t\geq 0$,
\amit{\begin{equation}\label{eq:main}
	\begin{split}
	e_{p}\pa{f(t)|f_\infty} \leq\mathcal{C}_{p,\epsilon}&\max\Bigg(e_p\pa{f_0|f_\infty},\\
	&\rpa{p(p-1)e_p\pa{f_0|f_\infty}+1}^{\frac{2}{p}\pa{1+\epsilon(2-p)}}\Bigg) \pa{1+t^{2n}}e^{-2\mu t},
\end{split}
\end{equation}}
where $\mathcal{C}_{p,\epsilon}$ is an explicit constant that does not depend on $f_0$ but which may become unbounded as $\epsilon$ goes to zero. If $f_0$ has finite $2-$entropy then the above can be improved to
\begin{equation}\label{eq:main_p=2}
	\begin{split}
		e_{2}\pa{f(t)|f_\infty} \leq &\mathcal{C}_{2} e_2\pa{f_0|f_\infty}\pa{1+t^{2n}}e^{-2\mu t}, \quad t\geq 0,
	\end{split}
\end{equation}
with $\mathcal{C}_{2}$ being a fixed constant. 
\end{theorem}
\begin{remark}\label{rem:lesser_conditions}
	We would like to point out a few observations:
	\begin{itemize}
	\item The family of functions $\br{\psi_p}_{1<p\leq 2}$ can be extended to $p=1$ by a simple limit procedure. In that case we obtain the generating function of the famous \textit{Boltzmann entropy}:
		$$\psi_1(y):=y\log y -y +1.$$
	This entropy will play an important role in our investigation. 
	\item One can also easily show that when $1\leq p_1\leq p_2 \leq 2$ we have that
	\begin{equation}\label{eq:monotonicity_of_entropies}
		e_{p_1}\pa{f|f_\infty} \leq \frac{p_2}{p_1}e_{p_2}\pa{f|f_\infty}
	\end{equation}	
	showing an almost monotonicity property for our family of entropies. We will mention this property again later and will call it ``monotonicity'' for simplicity.
	\item It is clear to see that our \amit{estimate} in \eqref{eq:main} is not optimal, especially when $f_0$ is very close to $f_\infty$ or when $t$ is close to $t=0$, yet it does give the correct sharp long time behaviour of the solution. We would like to mention that a way to explicitly compute the norm of the propagator for the Fokker-Planck equation can be found in \cite{ASS}.
	\end{itemize}
\end{remark}
A usual strategy to show theorems of the nature of Theorem \ref{thm:main} is to try and find a functional inequality between the entropy and its \textit{production}, defined as minus the dissipation of the entropy under the flow of the equation. Such a geometric inequality, when exists, together with the definition of the entropy production results in a differential inequality for the entropy which leads to explicit rates of convergence to equilibrium. This method is usually known as \textit{the entropy method}\footnote{It is worth to mention that in the case of the standard Fokker-Planck equation (i.e. $\D=\C={\bf{I}}$), the functional inequality that is recovered using the entropy method for the Boltzmann entropy is the famous log-Sobolev inequality.}. In most physically relevant cases one expects a functional inequality of the form
$$I\pa{f|f_\infty} \geq c e\pa{f|f_\infty}, $$
 \amit{where} $I$ \amit{is} the entropy production \amit{of the entropy $e$}, which implies an exponential rate of convergence. 
\\
More often than not when using the entropy method one finds that the entropy production has a ``better'' structure than the entropy (though it requires higher regularity). For example, in the setting of \eqref{eq:fokkerplanck}, the entropy production of the $\psi-$entropy, $e_\psi$, is the so-called $\psi-$\textit{Fisher Information}
\begin{equation}\label{eq:normalfisher}
	I_\psi(f|f_\infty)=\int_{\R^d}\psi^{\prime\prime}\pa{\frac{f(x)}{f_\infty(x)}}\nabla\pa{\frac{f(x)}{f_\infty(x)}}^T \D \nabla\pa{\frac{f(x)}{f_\infty(x)}}f_\infty(x) dx,
\end{equation}
which possesses a geometrically appealing quadratic structure in the gradient of $f$, independently of the generating function  $\psi$.\\
Relying on this observation, the \textit{Bakry-Emery method}, an extremely prolific and successful method, was devised: Consider the entropy production as the functional whose evolution we'd like to investigate and apply the entropy method to it (see \cite{BaEmD85,BaEmH84}). If one can find a geometric inequality between the entropy production and \textit{its} production then a simple time integration would yield the required geometric inequality for the original entropy method (i.e. for the entropy and its production). Moreover, as the inequality attained using this method implies the inequality we're looking for it is more, or at least as, geometrically fundamental. \\
The Bakry-Emery method lies at the heart of the modern investigation of convergence to equilibrium in many entropically governed and physically relevant equations, including the Fokker-Planck equation. See, for instance, \cite{AE,AMTU01,BaEmD85}. \\
The Bakry-Emery method was the starting point of the long time analysis of \eqref{eq:fokkerplanck} by Arnold and Erb in \cite{AE}. A closer look at the Fisher information, given by  \eqref{eq:normalfisher}, reveals a problem with the method when the diffusion matrix is degenerate. Indeed, in that case one can't expect a linear bound on $I_\psi(f|f_\infty)$ w.r.t.\ $e_\psi(f|f_\infty)$
to be valid in general, as $I_\psi(f|f_\infty)$ could be zero in many situations where the entropy is not\footnote{Moreover, an inequality of the form $$ I_\psi\pa{f|f_\infty} \geq \Phi\pa{ e_{\psi}\pa{f|f_\infty}}$$
is also impossible for positive functions $\Phi$ in such a case.}. \\
To deal with these problems, Arnold and Erb have utilised the idea of \textit{hypocoercivity}, namely a modification of the desired ``distance'' functional in such a way that the inherent geometry of the problem will be pushed to the surface. In this particular setting 
the  \textit{modified} Fisher information the authors \amit{considered} was
\begin{equation}\label{eq:defoffisher}
	I_\psi^{\bf{P}}(f|f_\infty)=\int_{\R^d}\psi^{\prime\prime}\pa{\frac{f(x)}{f_\infty(x)}}\nabla\pa{\frac{f(x)}{f_\infty(x)}}^T {\bf{P}} \nabla\pa{\frac{f(x)}{f_\infty(x)}}f_\infty(x) dx,
\end{equation}
where $\PP$ is a positive definite matrix that is chosen \amit{in a way} which gives rise to Theorem \ref{thm:anton_erb}. We refer to \S \ref{sec:generalfisher} for further details on the choice of $\PP$.\\
The loss of $2\epsilon$ in the convergence rate of the defective case in Theorem \ref{thm:anton_erb}(ii), \textit{even when $\D$ is not degenerate}, can be attributed to the technique of finding $\PP$. It is worth to note that by allowing $\PP$ to also depend on the time variable $t$, Monmarch\'e has recovered \amit{in \cite{MoH19}} the sharp large time behaviour stated in Theorem \ref{thm:main}, though this result was attained via a differential inequality and not a functional inequality.

In \cite{AEW18} we focused our attention on attaining the sharp long time behaviour \amit{of $e_p$} in both the degenerate and defective case. As was previously mentioned, our investigation, while successful, has relied heavily on detailed spectral investigation of the Fokker-Planck operator $L$ in the Hilbertian setting of the space $L^2\pa{\R^d,f_\infty^{-1}}$ and other ``heavy'' tools from spectral theory. Utilising this investigation and tools we have managed to \amit{obtain} the desired convergence rate for $e_2$ and using the monotonicity of the family of entropies $\br{e_p}_{1<p\leq 2}$ (w.r.t.\ $p$) and a non-symmetric \textit{hypercontractivity} result\footnote{This refers to an embedding in "higher spaces", indexed by $p$. We have shown in \cite{AEW18} that if $e_p\pa{f_0|f_\infty}$ is finite for $1<p<2$ then the solution for equation \eqref{eq:fokkerplanck} will eventually (with explicit estimation on time) have finite $e_2$ entropy.} we achieved the same convergence rate for all $\br{e_p}_{1<p\leq 2}$. \\
This study, however, circumvented the entropy production completely as well as the Bakry-Emery method. The goal we set ourselves, whose result is this work, is to see if one can modify the Bakry-Emery method in a way that keeps its main idea intact, but allows one to use it in greater generality - which includes equation \eqref{eq:fokkerplanck}. 

To do so we start by noticing that in our setting the entropy production of the $\psi-$entropy is given by
$$I_\psi^{\bf{D}}(f|f_\infty)=\int_{\R^d}\underbrace{\psi^{\prime\prime}\pa{\frac{f(x)}{f_\infty(x)}}}_{\tiny{\begin{tabular}{c}$\psi$-related term -- \\ $f/f_\infty$ must (in most cases) \\ be positive \end{tabular}}}\underbrace{\nabla\pa{\frac{f(x)}{f_\infty(x)}}^T {\bf{D}} \nabla\pa{\frac{f(x)}{f_\infty(x)}}f_\infty(x)}_{\tiny{\begin{tabular}{c}$L^2$-like term --\\ sign of $f/f_\infty$ is irrelevant\end{tabular}}} dx.$$
The integrand of the Fisher information naturally ``breaks'' into two parts. This motivates us to consider different arguments in each of these two parts. The following definition, which we find to be a \textit{corner stone} of this work, expresses the above observation:
\begin{definition}\label{def:generalisedfisher}
	Let $\psi$ be a generating function (of relative entropy) and let $f$ and $g$ be given measurable functions with $f$ being positive and $g \in H^1_{\mathrm{loc}}\pa{\R^d}$. For a given positive semi-definite matrix ${\bf{P}}\in \R^{d\times d}$ we define the \textit{Generalised $\psi$-Fisher information}, $\mathcal{I}^{\bf{P}}_{\psi}$, with respect to a function $f_\infty$ as
	\begin{equation}\label{eq:generalisedfisher}
		\mathcal{I}^{\bf{P}}_{\psi}\pa{f,g|f_\infty} = \int_{\R^d}\psi^{\prime\prime}\pa{\frac{f(x)}{f_\infty(x)}}\nabla\pa{\frac{g(x)}{f_\infty(x)}}^T {\bm{P}}\nabla\pa{\frac{g(x)}{f_\infty(x)}} f_\infty(x)dx.
	\end{equation}
When $\psi=\psi_p$ with $1<p\leq 2$ we will use the notation $\mathcal{I}^{\bf{P}}_{\psi_p}=\mathcal{I}^{\bf{P}}_{p}$. As in the case of our family of entropies, $\br{e_p}_{1\leq p \leq 2}$, it is convenient to demand that $f$ will have the same mass as $f_\infty$.
\end{definition} 
It is worthwhile to emphasise at this point a key feature of the generalised Fisher information, one which we will rely on heavily in our investigation in this paper -- while the function $f$ \textit{must} be positive, the function $g$ \textit{need not} have any definite sign.

\subsection{An Outline of the Proof of our Main Theorem} 
The proof of our main theorem, Theorem \ref{thm:main}, is based on the following steps:
\begin{enumerate}[(i)]
	\item\label{item:description_of_proof_I} We \amit{will} show, following the ideas of the Bakry-Emery method, that the generalised Fisher information of a pair of solutions to the Fokker-Planck equation shares the same decay properties that the standard Fisher information has. \amit{More precisely, we will show that
		\begin{equation}\nonumber
			\frac{d}{dt}\mathcal{I}^{\bf{P}}_\psi\pa{f(t),g(t)|f_\infty}  \leq -2\lambda \mathcal{I}^{\bf{P}}_{\psi}\pa{f(t),g(t)|f_\infty},
		\end{equation}
	where $\PP$ and $\lambda>0$ are such that $\C^T {\bf{P}}+{\bf{P}}\C \geq 2\lambda {\bf P}$. This will automatically imply an $2\pa{\mu-\epsilon}$ exponential decay rate for the generalised Fisher information, similar to \eqref{eq:entropydecayAEdef}, with an appropriate choice of $\PP_{\epsilon}$ (Theorem \ref{thm:Fisher_decay}).
	}
	\item\label{item:description_of_proof_I.5} We will refine and improve \amit{the above} estimate for $\I^{\bf{P}}_{2}(f,g|f_\infty)$ when $g$ is ``spectrally far enough'' from the kernel of $L$ by extracting an additional positive term from the remainder term in this case. \amit{More precisely, we will show that
	\begin{equation}\nonumber
		\frac{d}{dt}I^{\bf{P}}_2 \pa{g(t)|f_\infty} \leq -2\pa{\lambda+{d}_{\mathrm{min}}}I^{\bf{P}}_{2}\pa{g(t)|f_\infty},
	\end{equation}
	when $\PP$ and $\lambda>0$ are such that $\C^T {\bf{P}}+{\bf{P}}\C \geq 2\lambda {\bf P}$, $d_{\mathrm{min}}$ is the smallest eigenvalue of the positive definite matrix $\D$, and $$\inner{g(t),\partial_{x_i}f_\infty}_{L^2\pa{\R^d,f_\infty^{-1}}} =0,\qquad i=1,\dots,d$$
	 (Theorem \ref{thm:Fisher_improved_decay})}.
	\item\label{item:description_of_proof_II} We will show that the generalised Fisher information of order $1<p<2$ can be estimated by an interpolation of $\I^{\bf{P}}_1$ and $\I^{\bf{P}}_2$ \amit{motivating us to focus on the investigation of $\I^{\bf{P}}_1$ and $\I^{\bf{P}}_2$ solely. More precisely, we will show that 
		\begin{equation}\nonumber
			\I^{\PP}_p\pa{f,g|f_\infty} \leq c_p\mathcal{I}^{\bf{P}}_{1}\pa{f,g|f_\infty}^{2-p}I^{\PP}_2\pa{g|f_\infty}^{p-1},
		\end{equation}
		for an explicit $c_p>0$ (Theorem \ref{thm:interpolation_p}). 
	} 
	\item \label{item:description_of_proof_III} \amit{Inspired by the above point, we will show that solutions to our Fokker-Planck equations that start from functions who have finite $p-$entropies will automatically create finite generalised Fisher informations of order $1$ and $2$ after some time (``regularisation'').} \amit{More precisely, we will show that for} any pair of initial data to \eqref{eq:fokkerplanck}, $f_0 \geq 0$ with unit mass and finite $p_1-$entropy and $g_0$ such that $\abs{g_0}$ has a finite $p_2-$entropy\amit{, with $p_1,p_2\in (1,2]$,} one can find an explicit $t_0$ after which the solutions $f(t)$ and $g(t)$ to the equation satisfy 
	\begin{itemize}
		\item $f(t),g(t)\in L^2\pa{\R^d,f^{-1}_{\infty}}$,
		\item $\mathcal{I}^{\bf{P}}_{2}(f(t),g(t)|f_\infty)<\infty$,
		\item $\mathcal{I}^{\bf{P}}_{1}(f(t),g(t)|f_\infty)<\infty$.
	\end{itemize}
	 We refer to the first property as  \textit{hypercontractivity}, the second as \textit{upper-contractivity}, and the third as \textit{lower-contractivity} (Theorem \ref{thm:contractivity}).
	\item \label{item:description_of_proof_IV} Combining the above with the $L^2\pa{\R^d,f_\infty^{-1}}$ decomposition  
	$$f=\underbrace{f_1}_{\tiny{\begin{tabular}{c}
			spectrally close \\ to the kernel \\ \amit{(finite dimenstional space)}
		\end{tabular}}}+\underbrace{f_2}_{\tiny{\begin{tabular}{c}
		spectrally far \\ from the kernel
	\end{tabular}}}$$ 
	and the fact that
	$$I^{\PP}_p\pa{f|f_\infty}=\I_p^{\PP}\pa{f,f|f_\infty}=\I_p^{\PP}\pa{f,f_1+f_2|f_\infty}$$
	$$ \leq 2\pa{\underbrace{\I_p^{\PP}\pa{f,f_1|f_\infty}}_{\amit{\tiny{ \begin{tabular}{c} explicit decay. $f_1$ belongs to \\ a finite dimensional space \\ implies a simple ODE connection \end{tabular}}}}+\underbrace{\I_p^{\PP}\pa{f,f_2|f_\infty}}_{\tiny{\text{improved/faster decay}}}}$$
	we will obtain a \textit{Fisher information variant} of our main theorem \amit{for $I_p^{\PP}$. More precisely, we will show that 
	\begin{equation}\nonumber
	\begin{gathered}
		I_p^{\PP}\pa{f(t)|f_\infty} \leq c_{p,\PP,\epsilon} \rpa{p(p-1)e_p\pa{f_0|f_\infty}+1}^{\frac{2}{p}\pa{1+\epsilon(2-p)}}
		\pa{1+t^{2n}} e^{-2\mu t}
	\end{gathered}
	\end{equation}
	after some time, when $\PP$ is such that $\C^T {\bf{P}}+{\bf{P}}\C > 2\pa{\mu-d_{\mathrm{min}}\pa{p-1}} {\bf P}$ and $f_0$ has finite $p-$entropy. This can be improved to
	\begin{equation}\nonumber
	\begin{gathered}
		I_2^{\PP}\pa{f(t)|f_\infty} \leq c_{\PP}e_2\pa{f_0|f_\infty}
		\pa{1+t^{2n}} e^{-2\mu t},
	\end{gathered}
	\end{equation}
	when $f_0$ has finite $2-$entropy
	(Theorem \ref{thm:main_for_fisher}). }
	\item The above \amit{together with the} convex Sobolev inequality
	\amit{$$e_p\pa{f|f_\infty} \leq C_{LS}I_p^{{\bf{I}}}\pa{f|f_\infty}$$}
	 will result in the proof of Theorem \ref{thm:main}. 
\end{enumerate}
 \amit{Looking closely at point \eqref{item:description_of_proof_IV} it is worth to note that }we will, in fact, prove a more fundamental result from which our main theorem will follow.

\subsection{The Structure of the Work}\label{subsec:strucrure}
 In \S\ref{sec:fokkerplanck} we will recall (without proofs) known facts about Fokker-Planck equations of the form \eqref{eq:fokkerplanck}. \S\ref{sec:generalfisher} will be devoted to the generalised Fisher information, defined in \eqref{eq:generalisedfisher}, and the study of its long time behaviour by means of the entropy method. In \S\ref{sec:hyper_hypo} we shall study the regularisation properties of the Fokker-Planck equation and the relationship between different $p-$generalised Fisher informations, which we will then use in \S\ref{sec:convergence} to show a variant of our main theorem for the $p-$Fisher information. This result will allow us to prove our main theorem in the short section \S\ref{sec:convergence_entropy}. In \S\ref{sec:finalremarks} we will conclude our study with a few remarks. A discussion about the decay properties of the generalised Fisher information in additional Fokker-Planck systems will appear in Appendix \ref{app:a} and several proofs that we have decided to postpone to keep the flow of the paper will appear in Appendix \ref{app:b}.

\section{Properties of the Fokker-Planck Equation}\label{sec:fokkerplanck}
In this section we will provide, mostly without proofs, known facts about Fokker-Planck equations of the form \eqref{eq:fokkerplanck} and their solutions. We shall also recall some spectral properties of the Fokker-Planck operator, $L$, in the appropriate setting. The proofs of these results can be found in \cite{AAS,AE}.
\begin{theorem}\label{thm:propertiesoffokkerplanck}
Assume that $\C$ and $\D$ satisfy conditions \eqref{item:cond_semipositive}-\eqref{item:cond_no_invariant_subspace_to_kernel} and consider the (possibly degenerate) Fokker-Planck equation \eqref{eq:fokkerplanck} with initial data $f_0\in L^1_+\pa{\R^d}$. Then:
\begin{enumerate}[(i)]
\item There exists a unique classical solution to the equation, $f\in C^\infty \pa{\R_+\times \R^d}$. Moreover, if $f_0\not\equiv 0$ then the solution is strictly positive for $t>0$.
\item The above solution satisfies $\int_{\R^d}f(t,x)dx=\int_{\R^d}f_0(x)dx$ (mass conservation).
\item If in addition $f_0\in L^{p}\pa{\R^d}$ for some $p\in[1,\infty]$, then $f\in C\pa{[0,\infty),L^p\pa{\R^d}}$. 
\end{enumerate}
\end{theorem}
\begin{theorem}\label{thm:equilibrium}
Assume that $\C$ and $\D$ satisfy conditions \eqref{item:cond_semipositive}-\eqref{item:cond_no_invariant_subspace_to_kernel}. Then there exists a unique stationary state $f_\infty\in L^1\pa{\R^d}$ to \eqref{eq:fokkerplanck} which satisfies $\int_{\R^d} f_\infty(x) dx = 1$. $f_\infty$ is of the form:
\begin{equation}\label{eq:equilibrium}
f_\infty(x)= c_{\K} e^{-\frac{1}{2}x^T \K^{-1} x},
\end{equation}
where $\K\in\R^{d\times d}$ is the unique positive definite solution to the continuous Lyapunov equation
\begin{equation}\nonumber
2\D=\C \K+\K \C^T,
\end{equation} 
and $c_{\K}$ is an appropriate normalisation constant.
\end{theorem}
One can rewrite the Fokker-Planck operator $L$  by using the above stationary state. Indeed, 
\begin{equation}\label{eq:FPrecast-gen}
  Lf(x)=\text{div}\pa{f_\infty(x)\C\K \nabla \pa{\frac{f(x)}{f_\infty(x)}}}\,.
\end{equation}
An extremely useful property of \eqref{eq:fokkerplanck} is that it can be greatly simplified with a linear change of variables. The following can be found in \cite{AAS,ASS}:
\begin{theorem}\label{thm:simpliedDC}
Consider the Fokker-Planck equation \eqref{eq:fokkerplanck} with associated matrices $\C$ and $\D$ that satisfy conditions \eqref{item:cond_semipositive}--\eqref{item:cond_no_invariant_subspace_to_kernel}. Then there exists a \textit{linear} change of variable of the form
$$x={\bf{T}}y,\quad \quad {\bf{T}}\text{ is an invertible $d\times d$ matrix}$$
which transforms \eqref{eq:fokkerplanck} to a Fokker-Planck equation of the same form in the variable $y$, with diffusion and drift matrices $\widetilde{\D}$ and $\widetilde{\C}$ such that
\begin{equation}\label{eq:diagD}
\widetilde{\D}=\operatorname{diag}\br{d_1,d_2,\dots,d_m,0\dots,0}
\end{equation} 
with $d_j>0$, $j=1,\ldots,m$, $m:=\operatorname{rank}(\D)$, being the positive eigenvalues of $\D$, and 
$$\widetilde{\C}_s=\frac{\widetilde{\C}+\widetilde{\C}^T}{2}=\widetilde{\D}.$$
 In these new variables the equilibrium $f_\infty$ is just the standard Gaussian (i.e. $\K=\II$).
\end{theorem} 

As the linear transformation detailed in Theorem \ref{thm:simpliedDC} does not change the rate of decay of any $\psi-$entropy (or its production), we will assume from this point onwards that we are in the coordinate system where $\D$ is of form \eqref{eq:diagD} and equals $\C_s$.\\
Lastly, in preparation to the exploration of the convergence in $L^2\pa{\R^d,f_\infty^{-1}}$, we recall the following spectral theorem from \cite{AE}:
\begin{theorem}\label{thm:spectraldecomposition}
Assume that the drift and diffusion matrices $\C$ and $\D$ satisfy conditions \eqref{item:cond_semipositive}-\eqref{item:cond_no_invariant_subspace_to_kernel}. Denote by 
\begin{equation}\label{eq:defVm}
V_m:=\span\br{\partial_{x_1}^{\alpha_1}\dots \partial_{x_d}^{\alpha_d}f_\infty(x)\ \Big|\ \alpha_1,\dots,\alpha_d\in \N\cup\br{0}, \sum_{i=1}^d \alpha_i=m}, \quad m\in\N\cup \{0\}.
\end{equation}
Then 
\begin{enumerate}[(i)]
	\item $\br{V_m}_{m\in \N\cup \br{0}}$ are mutually orthogonal finite dimensional spaces in $L^2\pa{\R^d, f_\infty ^{-1}}$.
	\item \begin{equation}\nonumber
		L^2\pa{\R^d,f_\infty^{-1}}= \bigoplus_{m\in\N\cup\br{0}}V_m.
	\end{equation}
	\item $V_m$ are invariant under $L$ and its adjoint (and thus under the flow of \eqref{eq:fokkerplanck}).
\end{enumerate}
Moreover, the spectrum of $L$ satisfies
\begin{equation}\nonumber
\sigma\pa{L}=\bigcup_{m\in\N\cup\br{0}}\sigma\pa{L\vert_{V_m}},
\end{equation}
and
\begin{equation}\nonumber
		\sigma\pa{L\vert_{V_m}}=\br{-\sum_{i=1}^d \alpha_i \lambda_i \ \Big|\ \alpha_1 ,\dots,\alpha_d\in\N\cup\br{0}, \sum_{i=1}^d \alpha_i=m},
\end{equation}
where $\br{\lambda_j}_{j=1,\dots,d}$ are the eigenvalues (with possible multiplicity) of the matrix $\C$. The eigenfunctions (or eigenfunctions and generalised eigenfunctions in the case $\C$ is defective) of $\br{L\vert_{V_m}}_{m\in \N\cup \br{0}}$ form a basis for $L^2\pa{\R^d,f_\infty^{-1}}$.
\end{theorem}

\begin{remark}\label{rem:identification_of_V_m}
	Theorem \ref{thm:spectraldecomposition} expresses precisely the idea that in $L^2\pa{\R^d,f_\infty^{-1}}$ the Fokker-Planck operator decomposes the space into finite dimensional orthogonal spaces which are more and more ``spectrally distanced'' from $V_0 = \ker \pa{L}$ -- the span of the equilibrium of the equation. More precisely, the spectral gap of \eqref{eq:fokkerplanck} satisfies $\mu = \min \br{\Re \lambda \mid \lambda \in \sigma(L|_{V_1})}$. Similarly, we have $\min\{\Re \lambda \mid \lambda \in \sigma(L|_{V_m})\} = m\mu$.
\end{remark}

\section{The Generalised Fisher Information}\label{sec:generalfisher}
This section explores the properties of the generalised Fisher information under the evolution of the Fokker-Planck equation \eqref{eq:fokkerplanck}, which will be remarkably reminiscent of the properties of the standard Fisher information.\\
We remind the reader that the generalised Fisher information for a given generating function $\psi$ and a positive semi-definite matrix $\PP$ was defined in Definition \ref{def:generalisedfisher} and is given by
 \begin{equation}\nonumber
 	\mathcal{I}^{\bf{P}}_{\psi}\pa{f,g|f_\infty} = \int_{\R^d}\psi^{\prime\prime}\pa{\frac{f(x)}{f_\infty(x)}}\nabla\pa{\frac{g(x)}{f_\infty(x)}}^T {\bm{P}}\nabla\pa{\frac{g(x)}{f_\infty(x)}} f_\infty(x)dx,
 \end{equation}
for a positive function $f$ such that $\int_{\R^d}f(x)dx=\int_{\R^d}f_\infty(x)dx$ and $g\in H^1_{\mathrm{loc}}\pa{\R^d}$.\\
A few points one can immediately make are: 
 \begin{itemize}
 \item As $\psi^{\prime\prime}$ is generally defined \textit{only} on $\R_+$ (with a few exceptions such as $\psi_2$), the function $f$ of the pair $\pa{f,g}$ must be positive. The function $g$, on the other hand, does not need to have a definite sign. 
 \item The roles of $f$ and $g$ are significantly different and can't be interchanged.
 \item For any positive function $f$ and any generating function $\psi$
 $$\mathcal{I}^{\bf{P}}_\psi\pa{f,f|f_\infty}=I^{\bf{P}}_\psi(f|f_\infty).$$
 Moreover, for $\psi=\psi_2$ we have that 
 $$\mathcal{I}^{\bf{P}}_{2}\pa{f,g|f_\infty} = \int_{\R^d}\nabla\pa{\frac{g(x)}{f_\infty(x)}}^T {\bm{P}}\nabla\pa{\frac{g(x)}{f_\infty(x)}} f_\infty(x)dx=I^{\PP}_2\pa{g|f_\infty}.$$
 \end{itemize}
In the remainder of this section we will explore the convergence to zero of the generalised Fisher information under the Fokker-Planck flow of both $f$ and $g$. We will focus on two cases:
\begin{itemize}
	\item\textit{Arbitrary $\psi$}: In this case we will recover the same decay result that was found in \cite{AE}.
	\item \textit{$\psi=\psi_2$}: \sloppy In this case we will rely on the spectral decomposition of $L^2\pa{\R^d,f_{\infty}^{-1}}$ by $L$ and show that if the initial condition is spectrally farther than $V_1$ from the equilibrium space $V_0$, i.e. if $f_0 \in V_1^{\perp}$, then the convergence of the generalised Fisher information is faster than the previous estimate.
\end{itemize}
We would like to remind the reader that we are assuming that $\D$ and $\C$ are such that $\D$ is diagonal and $\D=\C_s$ (as remarked after Theorem \ref{thm:simpliedDC}).
\begin{theorem}[Arbitrary $\psi$]\label{thm:Fisher_decay}
Consider equation \eqref{eq:fokkerplanck} with a positive semi-definite diagonal matrix $\D$ and a positive stable matrix $\C$ such that $\D=\C_s$, and such that conditions \eqref{item:cond_semipositive}--\eqref{item:cond_no_invariant_subspace_to_kernel} are satisfied. Let $f_0$ be a unit mass element of $ L^1_+\pa{\R^d}$ and $g_0\in L^1\pa{\R^d}$ be given such that there exists $t_0\geq 0$ such that $\mathcal{I}_\psi^{\PP}(f(t_0), g(t_0)|f_\infty) < \infty$. Let $f,g$ be the solutions to the (possibly degenerate) Fokker-Planck equation with initial data $f_0,g_0$. If $\PP$ is a positive definite matrix that satisfies 
\begin{equation}\label{eq:CPcondition}
\C^T {\bf{P}}+{\bf{P}}\C \geq 2\lambda {\bf P},
\end{equation}
for some $\lambda>0$, then for any $t>t_0$
\begin{equation}\label{eq:Fisher_decay_in_functional_form}
	\frac{d}{dt}\mathcal{I}^{\bf{P}}_\psi\pa{f(t),g(t)|f_\infty}  \leq -2\lambda \mathcal{I}^{\bf{P}}_{\psi}\pa{f(t),g(t)|f_\infty},
\end{equation}
and as such for any $t\geq t_0$ 
\begin{equation}\label{eq:Fisher-decay}
\mathcal{I}^{\bf{P}}_\psi \pa{f(t),g(t)|f_\infty} \leq \mathcal{I}^{\bf{P}}_{\psi}\pa{f(t_0),g(t_0)|f_\infty}e^{-2\lambda (t-t_0)}.
\end{equation}
\end{theorem}
\begin{remark}\label{rem:generalised_well_defined_as_f_positive}
	The fact that $f_0\in L^1_+\pa{\R^d}$ and has a unit mass implies, according to Theorem \ref{thm:propertiesoffokkerplanck}, that $f(t)$ is positive for all $t>0$. Moreover, $f,g\in C^\infty\pa{\R_+\times \R^d}$, so no regularity issues arise. Thus, the generalised Fisher information in the above theorem is well defined (taking values in $[0, \infty]$).
\end{remark}
\begin{remark}\label{rem:on_P}
	For a given positive stable $\C$, the supremum of the attainable $\lambda$ in \eqref{eq:CPcondition} is $\mu$, the spectral gap of $\C$. If all $\lambda\in R_\mu$ are non-defective, $\lambda = \mu$ can actually be reached if $\PP>0$ is chosen appropriately (for its construction see the proof of Lemma 4.3 in \cite{AE}).
\end{remark}

\begin{remark}\label{rem:extension_for_potentials}
	It is worth to mention that the results of Theorem \ref{thm:Fisher_decay} remain valid in far more general cases which include possibilities for the matrix $\D$ to be $x$-dependent and for the drift term to be nonlinear in $x$. We refer the interested reader to such result in Appendix \ref{app:a}.
\end{remark}
\begin{theorem}[$\psi=\psi_2$]\label{thm:Fisher_improved_decay}
Consider equation \eqref{eq:fokkerplanck} with a positive definite diagonal matrix $\D$ and a positive stable matrix $\C$ such that $\D=\C_s$. Let  $g_0\in L^2\pa{\R^d,f_\infty^{-1}}$ be given, and let $g$ be the solution to the Fokker-Planck equation with initial data $g_0$. Assume in addition that $g_0\perp V_1$, with $V_1$ defined in \eqref{eq:defVm}, i.e.
$$\int_{\R^d}g_0(x) \partial_{x_i}f_\infty(x) f_\infty(x)^{-1}dx=0,\qquad \forall i=1,\dots,d.$$
Then for any positive definite matrix $\PP$ such that
\begin{equation}
\C^T {\bf{P}}+{\bf{P}}\C \geq 2\lambda {\bf P}
\end{equation}
for some $\lambda>0$, we have that for any $t>0$, $I_2\pa{g(t)|f_\infty}<\infty$ and 
\begin{equation}\label{eq:Fisher_improved_decay}
\frac{d}{dt}I^{\bf{P}}_2 \pa{g(t)|f_\infty} \leq -2\pa{\lambda+{d}_{\mathrm{min}}}I^{\bf{P}}_{2}\pa{g(t)|f_\infty},
\end{equation}
where $\mathrm{d}_{\mathrm{min}}>0$ is the minimal eigenvalue of $\D$. Consequently we have that for any $t_0>0$
\begin{equation}\label{eq:Fisher_improved_decay_I}
	I^{\bf{P}}_2 \pa{g(t)|f_\infty} \leq I^{\bf{P}}_{2}\pa{g(t_0)|f_\infty}e^{-2\pa{\lambda+\mathrm{d}_{\mathrm{min}}} (t-t_0)}.
\end{equation}
for $t\geq t_0$.
\end{theorem}

\begin{remark}\label{rem:only_place}
	We would like to point out that Theorem \ref{thm:Fisher_improved_decay} is the only place we will truly need the positivity condition on $\D$ -- the need for this in any theorem that follows stems from a use of Theorem \ref{thm:Fisher_improved_decay} in its proof. This dependence is made explicit in the result as the term $\mathrm{d}_{\mathrm{min}}>0$ appears and provides an improved decay. In recent work \cite{AEW23}, the authors have managed to show that $I^{\bf{P}}_2$ satisfies an improved decay property, similar in nature to the above, even when $\D$ is degenerate, as long as conditions \eqref{item:cond_semipositive}-\eqref{item:cond_no_invariant_subspace_to_kernel} hold. This, however, required an explicit estimate of the operator norm of $e^{Lt}$ on the Banach space $V_0^{\perp}$ which was explored in \cite{ASS}. The proof we provide in the case where $\D$ is not degenerate is much more elementary. 
\end{remark}

\subsection{Proof of Main Results}

\begin{proof}[Proof of Theorem \ref{thm:Fisher_decay}] The proof is an adaptation of the the proof of Proposition 4.5 in \cite{AE}. Denoting by $u(x):=\nabla \pa{\frac{f(x)}{f_\infty(x)}}$ and $v:=\nabla \pa{\frac{g(x)}{f_\infty(x)}}$ we find that
\begin{equation}\label{eq:diffoffisher}
	\begin{gathered}
		\frac{d}{dt}\mathcal{I}^{\bf{P}}_{\psi}\pa{f(t),g(t)|f_\infty}=\underbrace{\int_{\R^d} \psi^{\prime\prime\prime}\pa{\frac{f(t,x)}{f_\infty(x)}}\partial_t f(t,x) v(t,x)^T {\bm{P}}v(t,x)dx}_{\mathfrak{A}:=}\\
		+\underbrace{2\int_{\R^d} \psi^{\prime\prime}\pa{\frac{f(t,x)}{f_\infty(x)}}v(t,x)^T {\bf{P}} \partial_t v(t,x)f_\infty(x)dx}_{\mathfrak{B}:=}.
	\end{gathered}
\end{equation}
To close the appropriate inequality \eqref{eq:Fisher_decay_in_functional_form} we will evaluate these two terms, starting with $\mathfrak{B}$.\\
Recalling from \eqref{eq:FPrecast-gen} that the Fokker-Planck in our case can be rewritten as
$$\partial_t f(t,x)=\text{div}\pa{f_\infty(x)\C \nabla \pa{\frac{f(t,x)}{f_\infty(x)}}}$$ 
we see that
$$\partial_t v(t,x)=\nabla \pa{\partial_t \frac{g(t,x)}{f_\infty(x)}}=\nabla\pa{\frac{\text{div}\pa{f_\infty(x) \C  v(t,x)}}{f_\infty(x)}}$$
$$=\nabla\pa{\text{div}\pa{\C v(t,x) }+\frac{\nabla f_\infty(x)^T \C v(t,x)}{f_\infty(x)} }.$$
Since $\text{div}\pa{\B \nabla h}=\text{div}\pa{\B_s \nabla h}$ for any constant matrix $\B$ and $h\in C^2$, we find that 
\begin{equation}\label{eq:partialtV}
\begin{gathered}
\partial_t v(t,x)=\nabla \text{div}\pa{\D v(t,x)}-\nabla \pa{x^T \C  v(t,x)}\\
=\nabla \text{div}\pa{\D  v(t,x)} - \pa{\C J_x \pa{v(t,x)}}^T x - \C v(t,x),
\end{gathered}
\end{equation}
where we have used the fact that $\nabla f_\infty(x)=-x f_\infty(x)$ and the notation 
$$J_x \pa{v(x)} := \pa{\begin{tabular}{c}
$ \pa{\nabla v_1}^T$ \\ $\vdots$ \\ $ \pa{\nabla v_d}^T$
\end{tabular}}=\pa{\begin{tabular}{ccc}
$\partial^2_{x_1}\frac{g(x)}{f_\infty(x)}$ & $\dots$ & $\partial^2_{x_d x_1}\frac{g(x)}{f_\infty(x)}$ \\ $\vdots$ & $\vdots$ & $\vdots$ \\ $\partial^2_{x_1x_d}\frac{g(x)}{f_\infty(x)}$ & $\dots$ & $\partial^2_{x_d^2}\frac{g(x)}{f_\infty(x)}$
\end{tabular}}=\mathrm{Hess}\pa{\frac{g(x)}{f_\infty(x)}}.$$
Next, we notice that
$$\int_{\R^d} \psi^{\prime\prime}\pa{\frac{f(x)}{f_\infty(x)}}v(x)^T {\bf{P}} \nabla\big(\text{div}\pa{\D v(x)}\big)f_\infty(x)dx$$
$$=\sum_{i=1}^d \int_{\R^d} \psi^{\prime\prime}\pa{\frac{f(x)}{f_\infty(x)}}v_i(x)\sum_{l=1}^d{\bf{P}}_{il} \pa{\sum_{j=1}^d d_j \partial_{x_l}\partial_{x_j}v_j(x)} f_\infty(x)dx$$
$$=-\sum_{i,j,l=1}^d \int_{\R^d} \psi^{\prime\prime\prime}\pa{\frac{f(x)}{f_\infty(x)}}u_j(x) v_i(x){\bf{P}}_{il} d_j \partial_{x_l}v_j(x) f_\infty(x)dx$$
$$-\sum_{i,j,l=1}^d \int_{\R^d} \psi^{\prime\prime}\pa{\frac{f(x)}{f_\infty(x)}}\partial_{x_j}v_i(x){\bf{P}}_{il} d_j \partial_{x_l}v_j(x) f_\infty(x)dx$$
$$+\sum_{i,j,l=1}^d \int_{\R^d} \psi^{\prime\prime}\pa{\frac{f(x)}{f_\infty(x)}} v_i(x){\bf{P}}_{il} d_j \partial_{x_l}v_j(x)x_j f_\infty(x)dx.$$
Since $\D$ is diagonal we have that
$$d_j\partial_{x_l}v_j(x)=\pa{\D J_x\pa{v(x)}}_{jl}$$ 
and as such
\begin{equation}\label{eq:fishcerdecayproofI}
\begin{gathered}
\int_{\R^d} \psi^{\prime\prime}\pa{\frac{f(x)}{f_\infty(x)}}v(x)^T {\bf{P}} \nabla\big(\text{div}\pa{\D v(x)}\big) f_\infty(x)dx\\
\overset{{\bf{P}}_{il}={\bf{P}}_{li}}{=}-\int_{\R^d}\psi^{\prime\prime\prime}\pa{\frac{f(x)}{f_\infty(x)}}u(x)^T  \D J_x\pa{v(x)}{\bf{P}}v(x) f_\infty(x)dx \\
-\int_{\R^d}\psi^{\prime\prime}\pa{\frac{f(x)}{f_\infty(x)}}\tr\big({\bf{P}}J_x\pa{v(x)}\D J_x\pa{v(x)}\big) f_\infty(x)dx \\
+\int_{\R^d}\psi^{\prime\prime}\pa{\frac{f(x)}{f_\infty(x)}} v(x)^T {\bf{P}} J_x\pa{v(x)}^T \D x f_\infty(x)dx .
\end{gathered}
\end{equation}
Using the fact that for any $v\in\R^d$ and any matrix $\B\in\R^{d\times d}$ 
$$v^T \B v = \frac{1}{2}v^T\pa{\B+\B^T}v$$
we conclude from \eqref{eq:partialtV}, \eqref{eq:fishcerdecayproofI}, and the symmetry of $J_x\pa{v(x)}$ that
\begin{equation}\label{eq:fishcerdecayproofII}
\begin{gathered}
\mathfrak{B}=2\int_{\R^d} \psi^{\prime\prime}\pa{\frac{f(t,x)}{f_\infty(x)}}v(t,x)^T {\bf{P}} \partial_t v(t,x)f_\infty(x)dx \\
=-\int_{\R^d} \psi^{\prime\prime}\pa{\frac{f(t,x)}{f_\infty(x)}}v(t,x)^T \pa{{\bf{P}}\C+\C^T{\bf{P}}}v(t,x)f_\infty(x) dx \\
-2\int_{\R^d}\psi^{\prime\prime\prime}\pa{\frac{f(t,x)}{f_\infty(x)}}u(t,x)^T \D J_x\pa{v(t,x)}{\bf{P}}v(t,x) f_\infty(x)dx \\
-2\int_{\R^d}\psi^{\prime\prime}\pa{\frac{f(t,x)}{f_\infty(x)}}\tr\big({\bf{P}}J_x\pa{v(t,x)}\D J_x\pa{v(t,x)}\big) f_\infty(x)dx \\
+2\int_{\R^d}\psi^{\prime\prime}\pa{\frac{f(t,x)}{f_\infty(x)}} v(t,x)^T {\bf{P}} J_x\pa{v(t,x)} \pa{\D-\C^T}x  f_\infty(x)dx. 
\end{gathered}
\end{equation}
Denoting by 
$$\A =\D-\C^T=\C_s-\C^T=\C_{as},$$ 
the anti-symmetric part of $\C$, and using the definition of $J_x\pa{v(x)}$ we find that the last term in the above expression is

$$\int_{\R^d}\psi^{\prime\prime}\pa{\frac{f(x)}{f_\infty(x)}} v(x)^T {\bf{P}} J_x\pa{v(x)}\A x  f_\infty(x)dx$$
$$=\sum_{i,j,l,m=1}^d \int_{\R^d} \psi^{\prime\prime}\pa{\frac{f(x)}{f_\infty(x)}} v_i(x){\bf{P}}_{il} \partial_{x_m}v_l(x)\A_{mj}x_j f_\infty(x)dx$$
$$=-\sum_{i,j,l,m=1}^d \int_{\R^d} \psi^{\prime\prime\prime}\pa{\frac{f(x)}{f_\infty(x)}}u_m(x) v_i(x){\bf{P}}_{il} v_l(x)\A_{mj}x_j f_\infty(x)dx$$
$$-\sum_{i,j,l,m=1}^d \int_{\R^d} \psi^{\prime\prime}\pa{\frac{f(x)}{f_\infty(x)}} \partial_{x_m}v_i(x){\bf{P}}_{il} v_l(x)\A_{mj}x_j f_\infty(x)dx$$
$$-\sum_{i,j,l,m=1}^d \int_{\R^d} \psi^{\prime\prime}\pa{\frac{f(x)}{f_\infty(x)}} v_i(x){\bf{P}}_{il} v_l(x)\A_{mj}\delta_{jm} f_\infty(x)dx$$
$$+\sum_{i,j,l,m=1}^d \int_{\R^d} \psi^{\prime\prime}\pa{\frac{f(x)}{f_\infty(x)}} v_i(x){\bf{P}}_{il} v_l(x)\A_{mj}x_j x_m f_\infty(x)dx$$
$$=-\int_{\R^d} \psi^{\prime\prime\prime}\pa{\frac{f(x)}{f_\infty(x)}}u(x)^T \A x\, v(x)^T {\bf{P}}v(x) f_\infty(x)dx$$
$$-\int_{\R^d}\psi^{\prime\prime}\pa{\frac{f(x)}{f_\infty(x)}} v(x)^T {\bf{P}} J_x\pa{v(x)}\A x  f_\infty(x)dx$$
where we have used the symmetry of $\PP$ and $J_x\pa{v(x)}$, and the fact that for any anti-symmetric matrix $\A$
$$\sum_{i,j}\xi_i \A_{ij}\xi_j=0,\quad\text{and}\quad  \A_{ll}=0,\;\;\forall l=1,\dots, d.$$
Thus
$$2\int_{\R^d}\psi^{\prime\prime}\pa{\frac{f(x)}{f_\infty(x)}} v(x)^T {\bf{P}} J_x\pa{v(x)}^T\pa{\D-\C^T}x  f_\infty(x)dx$$
$$=-\int_{\R^d} \psi^{\prime\prime\prime}\pa{\frac{f(x)}{f_\infty(x)}}u(x)^T \pa{\D-\C^T} x \, v(x)^T {\bf{P}}v(x) f_\infty(x)dx.$$
Substituting the above in \eqref{eq:fishcerdecayproofII} yields
\begin{equation}\label{eq:fishcerdecayproofIII}
\begin{gathered}
\mathfrak{B}=2\int_{\R^d} \psi^{\prime\prime}\pa{\frac{f(t,x)}{f_\infty(x)}}v(t,x)^T {\bf{P}} \partial_t v(t,x)f_\infty(x)dx \\
=-\int_{\R^d} \psi^{\prime\prime}\pa{\frac{f(t,x)}{f_\infty(x)}}v(t,x)^T \pa{{\bf{P}}\C+\C^T{\bf{P}}}v(t,x)f_\infty(x) dx \\
-2\int_{\R^d}\psi^{\prime\prime\prime}\pa{\frac{f(t,x)}{f_\infty(x)}}u(t,x)^T \D J_x\pa{v(t,x)}{\bf{P}}v(t,x) f_\infty(x)dx \\
-2\int_{\R^d}\psi^{\prime\prime}\pa{\frac{f(t,x)}{f_\infty(x)}}\tr\big({\bf{P}}J_x\pa{v(t,x)}\D J_x\pa{v(t,x)}\big) f_\infty(x)dx \\
-\int_{\R^d} \psi^{\prime\prime\prime}\pa{\frac{f(t,x)}{f_\infty(x)}}u(t,x)^T \pa{\D-\C^T}x\, v(t,x)^T {\bf{P}}v(t,x) f_\infty(x)dx. 
\end{gathered}
\end{equation}

We now turn our attention to the first expression in \eqref{eq:diffoffisher}, $\mathfrak{A}$.
$$\mathfrak{A}=\int_{\R^d} \psi^{\prime\prime\prime}\pa{\frac{f(t,x)}{f_\infty(x)}}\partial_t f(t,x) v(t,x)^T {\bm{P}}v(t,x)dx
$$
$$
=\int_{\R^d} \psi^{\prime\prime\prime}\pa{\frac{f(t,x)}{f_\infty(x)}}\text{div}\pa{f_\infty(x) \C u(t,x)} v(t,x)^T {\bm{P}}v(t,x)dx$$
$$=\int_{\R^d} \psi^{\prime\prime\prime}\pa{\frac{f(t,x)}{f_\infty(x)}}\text{div}\pa{\C u(t,x)} v(t,x)^T {\bm{P}}v(t,x)f_\infty(x) dx
$$
$$
 -\int_{\R^d} \psi^{\prime\prime\prime}\pa{\frac{f(t,x)}{f_\infty(x)}}u(t,x)^T \C^T x\, v(t,x)^T {\bm{P}}v(t,x)f_\infty(x)dx$$
$$\underset{ \text{div}\pa{\B \nabla h}=\text{div}\pa{\B_s \nabla h}}{=}\int_{\R^d} \psi^{\prime\prime\prime}\pa{\frac{f(t,x)}{f_\infty(x)}}\text{div}\pa{\D u(t,x)} v(t,x)^T {\bm{P}}v(t,x)f_\infty(x) dx$$
$$
-\int_{\R^d} \psi^{\prime\prime\prime}\pa{\frac{f(t,x)}{f_\infty(x)}}u(t,x)^T \C^T x v(t,x)^T {\bm{P}}v(t,x)f_\infty(x)dx.$$
As
$$\int_{\R^d} \psi^{\prime\prime\prime}\pa{\frac{f(x)}{f_\infty(x)}}\text{div}\pa{\D u(x)} v(x)^T {\bm{P}}v(x)f_\infty(x) dx$$
$$=-\int_{\R^d} \psi^{\prime\prime\prime\prime}\pa{\frac{f(x)}{f_\infty(x)}}u(x)^T\D u(x) v(x)^T {\bm{P}}v(x)f_\infty(x) dx$$
$$-2\int_{\R^d} \psi^{\prime\prime\prime}\pa{\frac{f(x)}{f_\infty(x)}} u(x)^T \D J_x\pa{v(x)} {\bf{P}} v(x)f_\infty(x) dx$$
$$+\int_{\R^d} \psi^{\prime\prime\prime}\pa{\frac{f(x)}{f_\infty(x)}}u(x)^T \D x\, v(x)^T {\bm{P}}v(x)f_\infty(x) dx,$$
which follows from a simple integration by parts and the fact that due to the symmetry of ${\bf P}$ and $J_x\pa{v(x)}$
$$\sum_{i=1}^d d_i u_i(x) \partial_{x_i}\pa{v^T(x) {\bf{P}}v(x)} = 2\sum_{i,j,l=1}^d d_i u_i(x) \partial_{x_i}v_j(x) {\bf{P}}_{jl}v_l(x)$$
$$=2u(x)^T \D J_x\pa{v(x)} {\bf{P}} v(x),$$
we conclude that
\begin{equation}\label{eq:fishcerdecayproofIV}
\begin{gathered}
\mathfrak{A}=\int_{\R^d} \psi^{\prime\prime\prime}\pa{\frac{f(t,x)}{f_\infty(x)}}\partial_t f(t,x) v(t,x)^T {\bm{P}}v(t,x)dx\\
=-\int_{\R^d} \psi^{\prime\prime\prime\prime}\pa{\frac{f(t,x)}{f_\infty(x)}}u(t,x)^T \D u(t,x) \, v(t,x)^T {\bm{P}}v(t,x)f_\infty(x) dx\\
-2\int_{\R^d} \psi^{\prime\prime\prime}\pa{\frac{f(t,x)}{f_\infty(x)}} u(t,x)^T \D J_x\pa{v(t,x)} {\bf{P}} v(t,x)f_\infty(x) dx\\
+\int_{\R^d} \psi^{\prime\prime\prime}\pa{\frac{f(t,x)}{f_\infty(x)}}u(t,x)^T \pa{\D-\C^T} x \,v(t,x)^T {\bm{P}}v(t,x)f_\infty(x) dx.
\end{gathered}
\end{equation}
Combining \eqref{eq:fishcerdecayproofIII} and \eqref{eq:fishcerdecayproofIV} shows that
\begin{equation}\label{eq:fishcerdecayproofV}
\begin{gathered}
\frac{d}{dt}\mathcal{I}^{\bf{P}}_{\psi}\pa{f,g|f_\infty}=-\int_{\R^d} \psi^{\prime\prime}\pa{\frac{f(t,x)}{f_\infty(x)}}v(t,x)^T \pa{{\bf{P}}\C+\C^T{\bf{P}}}v(t,x)f_\infty(x) dx \\
-\int_{\R^d} \psi^{\prime\prime\prime\prime}\pa{\frac{f(t,x)}{f_\infty(x)}}u(t,x)^T \D u(t,x) \, v(t,x)^T {\bm{P}}v(t,x)f_\infty(x) dx\\
-4\int_{\R^d}\psi^{\prime\prime\prime}\pa{\frac{f(t,x)}{f_\infty(x)}}u(t,x)^T \D J_x\pa{v(t,x)}{\bf{P}}v(t,x) f_\infty(x)dx \\
-2\int_{\R^d}\psi^{\prime\prime}\pa{\frac{f(t,x)}{f_\infty(x)}}\tr\big({\bf{P}}J_x\pa{v(t,x)}\D J_x\pa{v(t,x)}\big) f_\infty(x)dx.
\end{gathered}
\end{equation}
Denoting by
$${\bf{X}}:=\pa{\begin{tabular}{cc}
$\psi^{\prime\prime}\pa{\frac{f(x)}{f_\infty(x)}}$ & $\psi^{\prime\prime\prime}\pa{\frac{f(x)}{f_\infty(x)}}$ \\
$\psi^{\prime\prime\prime}\pa{\frac{f(x)}{f_\infty(x)}}$ & $\frac{1}{2}\psi^{\prime\prime\prime\prime}\pa{\frac{f(x)}{f_\infty(x)}}$
\end{tabular}}, $$
$${\bf{Y}}:=\pa{\begin{tabular}{cc}
$\tr\big({\bf{P}}J_x\pa{v(x)}\D J_x\pa{v(x)}\big)$ & $u(x)^T \D J_x\pa{v(x)}{\bf{P}}v(x) $ \\
$u(x)^T \D J_x\pa{v(x)}{\bf{P}}v(x) $ & $u(x)^T \D u(x)\,  v(x)^T {\bm{P}}v(x)$
\end{tabular}},$$
and using \eqref{eq:CPcondition} we find that 
\begin{equation}\label{eq:fishcerdecayproofVI}
\begin{gathered}
\frac{d}{dt}\mathcal{I}^{\bf{P}}_{\psi}\pa{f(t),g(t)|f_\infty}\leq -2\lambda\int_{\R^d} \psi^{\prime\prime}\pa{\frac{f(t,x)}{f_\infty(x)}}v(t,x)^T {\bf P} v(t,x)f_\infty(x) dx\\
-2\int_{\R^d}\tr\pa{{\bf{X}}(t){\bf{Y}}(t)}f_\infty(x)dx. 
\end{gathered}
\end{equation}
If we will show that ${\bf{X}}$ and $\bf{Y}$ are positive semi-definite matrices, which will imply that $\tr\pa{{\bf{XY}}}\geq 0$, we would be able to infer our differential inequality \eqref{eq:Fisher_decay_in_functional_form} and conclude the proof. \\
Indeed, since $\psi$ is a generating function for an admissible relative entropy 
$${\bf{X}}_{11}=\psi^{\prime\prime} \geq 0,, \quad {\bf X}_{22}=\frac12 \psi''''\geq 0, \quad\text{and} \det{\bf{X}}=\frac{1}{2}\psi^{\prime\prime}\psi^{\prime\prime\prime\prime}-\pa{\psi^{\prime\prime\prime}}^2 \geq 0$$
	showing that $\bf{X}$ is positive semi-definite. Turning our attention to $\bf{Y},$
we notice that since $\D$ and ${\bf{P}}$ are positive semi-definite ${\bf Y}_{22}\geq 0$. The cyclicity of the trace together with the symmetry of $J_x\pa{v(x)}$ and the fact that $\D$ is positive semi-definite
 imply that
$${\bf{Y}}_{11}=\tr\pa{{\bf{\sqrt{P}}}^TJ_x\pa{v(x)}^T \D J_x\pa{v(x)}{\bf \sqrt{P}}} \geq 0.$$
To show the non-negativity of $\det{\bf{Y}}$ we use the fact that $u^T v =\tr\pa{v u^T}$ and notice that
$$\pa{u(x)^T \D J_x\pa{v(x)}{\bf{P}}v(x)}^2 =\pa{\tr\pa{\D J_x\pa{v(x)}\PP v(x) u(x)^T }}^2	$$
$$=\pa{\tr\pa{v(x) u(x)^T\D J_x\pa{v(x)}{\bf{P}}}}^2=\pa{\tr\pa{{\bf\sqrt{P}}v(x)u(x)^T \sqrt{\D} \cdot \sqrt{\D} J_x(v(x)) {\bf\sqrt{P}}}}^2$$
$$\underset{\tiny{\begin{tabular}{c}
		$\tr\pa{{\bf{A}}^T\B}$ is an \\ inner product	
		\end{tabular}}}{\leq} \tr\pa{\pa{{\bf \sqrt{P}} v(x)u(x)^T \sqrt{\D}} ^T\pa{{\bf \sqrt{P}} v(x)u(x)^T \sqrt{\D}}}$$
$$\hspace{4cm}\times \tr\pa{\pa{\sqrt{\D}J_x(v(x)) {\bf \sqrt{P}} }^T \pa{\sqrt{\D}J_x(v(x)) {\bf \sqrt{P}} }}$$
$$=\tr\big(u(x)v(x)^T {\bf P} v(x)u(x)^T \D\big)\, \tr\big(\D J_x(v(x)) {\bf P} J_x(v(x))\big)$$
$$
=v(x)^T {\bf P} v(x) \, u(x)^T\D u(x)\cdot \tr\big(\D J_x(v(x)) {\bf P} J_x(v(x))\big).
$$
As
$$\det{\bf{Y}}=v(x)^T {\bf P} v(x)\, u(x)^T\D u(x)\cdot \tr\big(\D J_x(v(x)) {\bf P} J_x(v(x))\big)-\pa{u(x)^T \D J_x\pa{v(x)}{\bf{P}}v(x)}^2,$$
we conclude the desired positive semi-definiteness of ${\bf{Y}}$. The proof is thus complete.
\end{proof}
We  now turn our attention to the proof of Theorem \ref{thm:Fisher_improved_decay}.
\begin{proof}[Proof of Theorem \ref{thm:Fisher_improved_decay}]
The instantaneous generation of a finite $2-$Fisher information for the Fokker-Planck equation is a well known result and a proof can be found in Theorem 4.8 in \cite{AE}. It is important to point out that while Arnold and Erb prove the result for a positive unit mass function $f$, the fact that we are dealing with $I_2$ and $e_2$ implies that the sign of the function is not important and a simple rescaling of a non-zero function, in the case where its mass is not $1$, will show that the result remains valid. We continue now to show \eqref{eq:Fisher_improved_decay} and \eqref{eq:Fisher_improved_decay_I}.\\
The key idea of the proof is to explicitly compute the remainder term in \eqref{eq:fishcerdecayproofVI} and to use it to gain the improved inequality \eqref{eq:Fisher_improved_decay}.
As was shown in the proof of Theorem \ref{thm:Fisher_decay}, we have that for any generating function $\psi$ and any $t\geq t_0>0$
\begin{equation}\nonumber
\begin{gathered}
\frac{d}{dt}\mathcal{I}^{\bf{P}}_{\psi}\pa{f(t),g(t)|f_\infty}\leq -2\lambda \mathcal{I}^{\bf{P}}_{\psi}\pa{f(t),g(t)|f_\infty}
-2\int_{\R^d}\tr\pa{{\bf{X}}(t){\bf{Y}}(t)}f_\infty(x)dx,
\end{gathered}
\end{equation}
where 
$${\bf{X}}:=\pa{\begin{tabular}{cc}
$\psi^{\prime\prime}\pa{\frac{f(x)}{f_\infty(x)}}$ & $\psi^{\prime\prime\prime}\pa{\frac{f(x)}{f_\infty(x)}}$ \\
$\psi^{\prime\prime\prime}\pa{\frac{f(x)}{f_\infty(x)}}$ & $\frac{1}{2}\psi^{\prime\prime\prime\prime}\pa{\frac{f(x)}{f_\infty(x)}}$
\end{tabular}},$$
$${\bf{Y}}:=\pa{\begin{tabular}{cc}
$\tr\pa{{\bf{P}}J_x\pa{\nabla\pa{\frac{ g(x)}{f_\infty(x)}}}\D J_x\pa{\nabla\pa{\frac{ g(x)}{f_\infty(x)}}}}$ & $\nabla\pa{\frac{ f(x)}{f_\infty(x)}}^T \D J_x\pa{\nabla\pa{\frac{ g(x)}{f_\infty(x)}}}{\bf{P}}\nabla\pa{\frac{ g(x)}{f_\infty(x)}} $ \\
$\nabla\pa{\frac{ f(x)}{f_\infty(x)}}^T \D J_x\pa{\nabla\pa{\frac{ g(x)}{f_\infty(x)}}}{\bf{P}}\nabla\pa{\frac{ g(x)}{f_\infty(x)}} $ & $\nabla\pa{\frac{ f(x)}{f_\infty(x)}}^T \D \nabla\pa{\frac{ f(x)}{f_\infty(x)}} \nabla\pa{\frac{ g(x)}{f_\infty(x)}}^T {\bm{P}}\nabla\pa{\frac{ g(x)}{f_\infty(x)}}$
\end{tabular}}.$$
Since
$$\mathcal{I}^{\PP}_{2}(f(t),g(t)|f_\infty)=I^{\PP}_{2}\pa{g(t)|f_\infty}$$ 
we conclude that
\begin{equation}\label{eq:imporved_decay_psi_2_I}
	\begin{gathered}
		\frac{d}{dt}I^{\PP}_{2}\pa{g(t)|f_\infty}\leq -2\lambda I^{\PP}_{2}\pa{g(t)|f_\infty}
		-2\int_{\R^d}\tr\pa{{\bf{X}}(t){\bf{Y}}(t)}f_\infty(x)dx.
	\end{gathered}
\end{equation}
In order to attain \eqref{eq:Fisher_improved_decay} we will focus our attention on the term $\tr\pa{{\bf{X}}{\bf{Y}}}$, which is much simpler when $\psi=\psi_2$.
Indeed, in this case we have that
$${\bf{X}}=\pa{\begin{tabular}{cc}
$1$ & $0$ \\
$0$ & $0$
\end{tabular}}$$
and as such
$$\tr\pa{{\bf{X}}{\bf{Y}}}=\tr\pa{{\bf{P}}J_x\pa{\nabla\pa{\frac{ g(x)}{f_\infty(x)}}}\D J_x\pa{\nabla\pa{\frac{ g(x)}{f_\infty(x)}}}}.$$
Using the fact that $J_x\pa{\nabla\pa{\frac{ g(x)}{f_\infty(x)}}}$ and $\PP$ are symmetric together with the cyclicity of the trace we can rewrite the above as
$$\tr\pa{\X\Y}=\tr\pa{{\sqrt{\bf{P}}}J_x\pa{v(x)}\D J_x\pa{v(x)}^T \sqrt{\bf{P}}^T},$$
where, like the proof of Theorem \ref{thm:Fisher_decay}, we have used the notation $v(x)=\nabla\pa{\frac{g(x)}{f_\infty(x)}}$.\\
Since for any diagonal matrix $\D=\operatorname{diag}\pa{d_1,\dots,\amit{d_d}}$, and matrices $\A$ and $\B$ we have that
\begin{equation}\label{eq:simple_tr_property}
\tr\pa{\A\D\B^T}=\sum_{i,j}d_j A_{ij}B_{ij},
\end{equation}
we see that
$$\tr\pa{\X\Y}=\sum_{i,j=1}^d d_j \pa{\sqrt{\PP}J_x \pa{v(x)}}^2_{ij} \geq \mathrm{d}_{\mathrm{min}}\sum_{i,j=1}^d \pa{\sqrt{\PP}J_x \pa{v(x)}}^2_{ij}.$$
From the definition of $J_x\pa{v(x)}$ we have that
$$\pa{\sqrt{\PP}J_x \pa{v(x)}}_{ij}=\sum_{k=1}^d \sqrt{\PP}_{ik}\partial_{x_j}v_k(x)=\partial_{x_j} \pa{\sqrt{\PP}v(x)}_i,$$ 
and denoting by $\zeta(x):=\sqrt{\PP} v(x)$ we find that 
$$\tr\pa{\X\Y} \geq \mathrm{d}_{\mathrm{min}} \sum_{i=1}^d \sum_{j=1}^d \abs{\partial_{x_j} \zeta_i(x)}^2= \mathrm{d}_{\mathrm{min}} \sum_{i=1}^d \abs{\nabla\zeta_i(x)}^2.$$
The weighted Poincar\'e inequality
\begin{equation}\nonumber
\int_{\R^d} f(x)^2f_\infty(x)dx- \pa{\int_{\R^d}f(x)f_\infty(x)dx}^2 \leq \int_{\R^d}\abs{\nabla f(x)}^2 f_\infty(x)dx,
\end{equation}
which holds for $f_\infty$ (see, for instance \cite{Be89}) implies that
\begin{equation}\label{eq:improved_rate_proof_I}
\begin{gathered}
\int_{\R^d}\tr\pa{\XX\YY}f_\infty(x)dx \geq \mathrm{d}_{\mathrm{min}}\sum_{i=1}^d \int_{\R^d} \abs{\nabla \zeta_i(x)}^2 f_\infty(x)dx\\
\geq \mathrm{d}_{\mathrm{min}}\sum_{i=1}^d \pa{\int_{\R^d} \zeta_i(x)^2f_\infty(x)dx - \pa{\int_{\R^d}\zeta_i(x)f_\infty(x)dx}^2}.
\end{gathered}
\end{equation}
This is the only point in the proof where we will use the information that $g_0\perp V_1$. Since $\br{V_m}_{m\in\N}$ are invariant under the Fokker-Planck operator and its adjoint we conclude that $g(t,x)\perp V_1$ for all $t\geq 0$. Thus,  for any $i=1,\dots, d$
$$\int_{\R^d}v_i(t,x)f_\infty(x)dx = - \int_{\R^d}\frac{g(t,x)}{f_\infty(x)}\partial_{x_i}f_\infty(x)dx=-\inner{g(t,x),\partial_{x_i}f_\infty(x)}_{L^2\pa{\R^d,f_\infty^{-1}}}=0.$$
Since $\zeta_i(t,x)=\sum_{k=1}^d \sqrt{\PP}_{ik}v_k(t,x)$ we conclude that 
\begin{equation}\label{eq:improved_rate_proof_II}
\int_{\R^d}\zeta_i(t,x)f_\infty(x)dx=0,
\end{equation}
for all $i=1,\dots,d$ and $t\geq 0$. Going back to \eqref{eq:improved_rate_proof_I}, and using \eqref{eq:improved_rate_proof_II} together with the fact that
$$\sum_{i=1}^d \int_{\R^d}\zeta_i(x)^2 f_\infty(x)
=\int_{\R^d}(\sqrt{\PP}v(x))^T(\sqrt{\PP}v(x))f_\infty(x)dx$$
$$=\int_{\R^d}v(x)^T \PP v(x)f_\infty(x)dx=I_2^{\PP}(g|f_\infty)$$
we find that
$$\int_{\R^d}\tr\pa{\X(t)\Y(t)}f_\infty(x)dx \geq \mathrm{d}_{\mathrm{min}}I_2^{\PP}\pa{g(t)|f_\infty}.$$
Combining the above with \eqref{eq:imporved_decay_psi_2_I} gives
\begin{equation}\nonumber
	\begin{gathered}
		\frac{d}{dt}I^{\bf{P}}_{2}\pa{g(t)|f_\infty}\leq -2\lambda I_2^{\bf{P}}\pa{g(t)|f_\infty}-2\mathrm{d}_{\mathrm{min}}I_2^{\PP}\pa{g(t)|f_\infty},
	\end{gathered}
\end{equation}
which is the desired result.
\end{proof}

\begin{remark}\label{rem:failed_inequality}
	The requirement that $\D$ is non-degenerate is essential in the proof we've provided above. Indeed, a key inequality we've used was that when $\D$ is positive definite we can find $\beta>0$ such that
	$$\tr\pa{\A\D\A^T} \geq \beta \tr\pa{\A\A^T} $$
	for any matrix $\A$. This doesn't hold when $\D$ is degenerate.\\
	In particular, with the choice of 
	\begin{equation}\nonumber
		\D= \begin{pmatrix}
			1& 0\\
			0 & 0
		\end{pmatrix}, \qquad \C= \begin{pmatrix}
			1& 1\\
			-1& 0
		\end{pmatrix},
	\end{equation}
	we can find an explicit function $g_0\in V_2$ such that the solution to the Fokker-Planck equation with initial datum $g_0$, $g(t)$, satisfies 
	$$
	\tr\pa{\X(t_n)\Y(t_n)}=0
	$$
	for a sequence $\br{t_n}_{n\in\N}$ that converges to infinity. This would mean that we won't be able to proceed as we did in our proof of Theorem \ref{thm:Fisher_improved_decay}, which implies that this approach can't, in general, be extended.
\end{remark}

\section{The Interplay Between Fisher Informations and Regularisation Properties of the Fokker-Planck Equation}\label{sec:hyper_hypo}
In the previous section we established that when the matrix $\D$ is positive definite one can get an improved convergence rate for the $2-$Fisher information of a function as long as it is spectrally far enough from the equilibrium state of the equation. Two immediate questions arise if we want to use this result:
\begin{itemize}
	\item Can we infer information on $\mathcal{I}_p^{\PP}\pa{f,g|f_\infty}$ from $I_2^{\PP}\pa{g|f_\infty}$ when $1<p<2$ when $g$ is a projection of $f$ in some sense?
	\item What guarantee do we have that $I_2^{\PP}\pa{f(t)|f_\infty}$ and $e_2\pa{f(t)|f_\infty}$ are finite if the initial datum is not in $L^2\pa{\R^d,f_{\infty}^{-1}}$?
\end{itemize}
In this section we will address the above concerns. We shall show that:
\begin{itemize}
\item For any $1<p<2$ one can bound $\I^{\PP}_p$ by interpolating $\I^{\PP}_1$ and $\I^{\PP}_2=I^{\PP}_2$, where
$$\I^{\PP}_1\pa{f,g|f_\infty}=\int_{\R^d}\frac{\nabla\pa{\frac{g(x)}{f_\infty(x)}}^T {\bm{P}}\nabla\pa{\frac{g(x)}{f_\infty(x)}}}{\frac{f(x)}{f_\infty(x)}} f_\infty(x)dx$$
is the generalised Fisher information that is associated to the Boltzmann entropy. 
	\item For any initial datum $f_0$ with finite $p-$entropy for some $1<p\leq 2$ there exists an explicit time after which the solution to our equation, $f(t)$, has a finite $2-$entropy (hypercontractivity)  and $2-$Fisher information (upper-contractivity). Moreover, if in addition $g(t)$ is another solution to the equation with initial datum with finite $\widetilde{p}-$entropy for some $1<\widetilde{p}\leq 2$ then we will get a similar result for $\I^{\PP}_1\pa{f(t),g(t)|f_\infty}$ (lower-contractivity). In fact we will show the creation of a finite $\I_p^{\PP}$ for any $1\leq p\leq 2$. 
\end{itemize}

\subsection{Interpolation of Generalised Fisher Information}\label{subsec:interpolation}
As was discussed in Remark \ref{rem:lesser_conditions}, the family of entropies, $\br{e_p}_{1\leq p\leq 2}$ is (almost) monotone. Consequently, we know that a finite $p_2-$entropy immediately implies a finite $p_1-$ entropy for $p_1<p_2$. This structure is \textit{lacking} when one considers the standard and generalised Fisher informations. The reason behind this lies \amit{in} the fact that the generating functions of the $p-$Fisher informations 
$$\psi^{\prime\prime}_p(y)=y^{p-2},\qquad 1\le p\le 2$$
have \textit{significantly} different behaviour as $y\searrow0$ and $y\nearrow \infty$ which can't be compared to any function from the generating family. Indeed, for any $1<p<2$ we see that $\psi^{\prime\prime}_p$ is dominated by $\psi^{\prime\prime}_1$ near zero and by $\psi^{\prime\prime}_2$ at infinity. This observation leads us to the main interpolation result of this subsection:
\begin{theorem}\label{thm:interpolation_p}
Let $f$ and $g$ be functions such that $\I^{\PP}_2\pa{f,g|f_\infty}=I^{\PP}_2\pa{g|f_\infty}$ and $\I^{\PP}_1\pa{f,g|f_\infty}$ are finite\footnote{Note that the fact that we're using the generalised $1-$Fisher information implies that $f$ must be positive almost everywhere.} for some positive definite $\PP\in\R^{d\times d}$. Then, for any $1<p<2$ we have that 
\begin{equation}\label{eq:interpolation_p}
\I^{\PP}_p\pa{f,g|f_\infty} \leq \frac{1}{\pa{p-1}^{p-1}\pa{2-p}^{2-p}}\mathcal{I}^{\bf{P}}_{1}\pa{f,g|f_\infty}^{2-p}I^{\PP}_2\pa{g|f_\infty}^{p-1}. 
\end{equation}
\end{theorem}
We remark that the interpolation constant in \eqref{eq:interpolation_p} equals 1 in both limits $p\searrow 1$ and $p\nearrow 2$. This theorem can be extended to the following:
\begin{theorem}\label{thm:interpolation_general}
Let $\psi$ be a generating function for an admissible relative entropy such that $\psi^{\prime\prime}$ has a singularity at zero.  Then, if there exist $c_1,c_2\geq 0$ and $0<\alpha<1$ such that
\begin{equation}\label{eq:growthofpsiprimeprime}
\psi^{\prime\prime}(y) \leq c_1 y^{-\alpha}+c_2, \quad \forall y>0,
\end{equation}
we have that for any positive definite ${\bf P}\in\R^{d\times d}$ 
\begin{equation}\label{eq:interpolationgeneralised}
\mathcal{I}^{\bf{P}}_{\psi}\pa{f,g|f_\infty} \leq 2c_1\mathcal{I}^{\bf{P}}_{1}(f,g|f_\infty)^{\alpha}I^{\bf{P}}_2(g|f_\infty)^{1-\alpha}+2c_2 I^{\bf{P}}_2(g|f_\infty),
\end{equation}
for any functions $f$ and $g$ for which the right hand side is well defined. 
\end{theorem}
Since we will only use Theorem \ref{thm:interpolation_p} in our main body of work we leave the proof of the general inequality to Appendix \ref{app:b} and focus on proving Theorem \ref{thm:interpolation_p}.
\begin{proof}[Proof of Theorem \ref{thm:interpolation_p}]
Given any $1<p<2$, we have that for any $y>0$ and arbitrary $R>0$
$$\psi_p^{\prime\prime}(y)=y^{p-2} \leq  \frac{R^{p-1}}{y}\1_{\br{0<\xi\leq R}}(y) +R^{p-2} \1_{\br{\xi>R}}(y).$$
Thus
$$\mathcal{I}^{\bf{P}}_{\psi}\pa{f,g|f_\infty} \leq R^{p-1}\mathcal{I}^{\bf{P}}_{1}\pa{f,g|f_\infty} +R^{p-2}I^{\bf{P}}_2(g|f_\infty).$$
Optimising over $R$ yields
$$R=\frac{2-p}{p-1}\frac{I^{\bf{P}}_2(g|f_\infty)}{\mathcal{I}^{\bf{P}}_{1}\pa{f,g|f_\infty}}$$
and as such
$$\mathcal{I}^{\bf{P}}_{\psi}\pa{f,g|f_\infty} \leq \frac{1}{p-1}\pa{\frac{p-1}{2-p}}^{2-p}\mathcal{I}^{\bf{P}}_{1}\pa{f,g|f_\infty}^{2-p}I^{\bf{P}}_2(g|f_\infty)^{p-1}. $$
\end{proof}


\subsection{Hyper-, Upper- and Lower-Contractivity}\label{subsec:contractivity}
The notion of hypercontractivity goes back to \cite{Nelson} (see also \cite{BGL}) and in its heart, it refers to ``regularisation'' by the evolution semigroup (improved integrability). In the context of our setting, an appropriate regularisation result will be that for any initial datum with finite $p_1$-entropy, where $p_1\in (1,2)$, we will be able to find a time after which the $p_2-$entropy and generalised $p_2-$Fisher information are finite for $p_2\in \rpa{1,2}$. Due to the monotonicity of the entropies we see that a hypercontractivity result for the entropies will follow immediately if we can show such result for $p_2=2$. This, in fact, was already established in Theorem 4.3 of \cite{AEW18}. The same proof will actually show that not only $e_2$ is finite, but $I_2^{\PP}$ will also become finite after enough time has passed.
As was mentioned in the previous subsection, however, the family of generating functions for the $p-$Fisher informations (generalised or not),$\br{\psi^{\prime\prime}_p}$, \amit{is} not monotone and as such the creation of a finite $I_2^{\PP}$ does \textit{not} guarantee a finite $\I^{\PP}_p$ for $1\leq p <2$. The goal of this subsection is to show directly the generation of a finite $\I_p^{\PP}\pa{f(t),g(t)|f_\infty}$ \amit{, where $p\in [1,2]$,} for solutions to the Fokker-Planck equations, $f(t)$ and $g(t)$, with initial data that have finite entropies. We would like to mention that according to our interpolation result, Theorem \ref{thm:interpolation_p}, the generation of a finite $I_2^{\PP}(g|f_\infty)$ together with a finite $\I_1^{\PP}\pa{f,g|f_\infty}$ would imply a generation of finite $\I^{\PP}_{p}\pa{f,g|f_\infty}$ for any $1\leq p\leq 2$. This implies that we only need to prove the creation of a finite $\I_1^{\PP}\pa{f,g|f_\infty}$, but as the same proof holds for any $p\in [1,2]$ we will not pursue the interpolation route at this point (though it will play an important role in the proof of our main theorem). 
 
The main theorem of this section is the following:\footnote{Recall that the generalised $2-$Fisher information is in fact the $2-$Fisher information of the second component, i.e. $\I_2^{\PP}\pa{f,g|f_\infty}=I_2^{\PP}\pa{g|f_\infty}$.}

\begin{theorem}\label{thm:contractivity}
	Let $f(t),g(t)$ be solutions to the Fokker-Planck equation \eqref{eq:fokkerplanck} with diffusion and drift matrices $\D$ and $\C$ which satisfy conditions \eqref{item:cond_semipositive}--\eqref{item:cond_no_invariant_subspace_to_kernel} and initial data $f_0,g_0$ such that $f_0$ is a non-negative unit mass function. Assume in addition that there exist $p_1,p_2\in (1,2]$ such that $e_{p_1}\pa{f_0|f_\infty}<\infty$ and $e_{p_2}\pa{\abs{g_0}|f_\infty}<\infty$ . Then: 
	\begin{enumerate}[(i)]
		\item \label{item:hyper} (hypercontractivity and upper-contractivity for $p=2$) There exists an explicit time, $t_1(p_2)$, which is independent of $g_0$ such that for all $t\geq t_1(p_2)$ we have that the solution to the equation \eqref{eq:fokkerplanck} satisfies
			\begin{equation}\label{eq:hypercontractivity_L_2_with_entropy_condition}
		\begin{split}
			\norm{g(t)}^2_{L^2\pa{\R^d,f_\infty^{-1}}} \leq  \mathcal{A}(p_2) \rpa{p_2\pa{p_2-1}e_{p_2}\pa{\abs{g_0}|f_\infty}+1}^{\frac{2}{p_2}}
		\end{split}
	\end{equation}
	with
	\begin{equation}\nonumber
		\mathcal{A}(p_2) = \pa{\frac{8}{3}}^d \pa{\frac{2p_2-1}{p_2-1}}^{\frac{d\pa{p_2-1}}{p_2}} \pa{2p_2}^{\frac{2}{p_2-1}}.
	\end{equation}
	Moreover,
		\begin{equation}\label{eq:hypercontractivity_entropy_with_entropy_condition}
		\begin{split}
			e_2\pa{g(t)|f_\infty} \leq  \max\pa{\mathcal{A}(p_2),1} \rpa{p_2\pa{p_2-1}e_{p_2}\pa{\abs{g_0}|f_\infty}+1}^{\frac{2}{p_2}}
		\end{split}.
	\end{equation}
		In addition,
		\begin{equation}\label{eq:hypercontractivity_with_entropy_condition}
			\begin{split}
					I_{2}^{\PP}(g(t)|f_\infty) \leq 
						\mathcal{B}(p_2)\rpa{p_2\pa{p_2-1}e_{p_2}\pa{\abs{g_0}|f_\infty}+1}^{\frac{2}{p_2}} 
			\end{split}
		\end{equation}
	with 
	\begin{equation}\nonumber
	\mathcal{B}(p_2)=	2^{\frac{3d}{2}} \pa{\frac{2p_2-1}{p_2-1}}^{\frac{d\pa{p_2-1}}{p_2}} \pa{2p_2}^{\frac{2}{p_2-1}}\mathrm{p}_{\mathrm{max}},
	\end{equation}
	where $\mathrm{p}_{\mathrm{max}}$ is the largest eigenvalue of $\PP$.
	\item \label{item:hypo} (upper and lower-contractivity for $p\in [1,2)$) For any $p\in [1,2)$ and any $\eta>0$ there exists an explicit time, $t_1(p_1,p_2,\eta)>0$, that may become unbounded as $\eta$ go to zero, such that for all $t\geq t_1(p_1,p_2,\eta)$ we have that 
	\begin{equation}\label{eq:hypocontractivity_with_entropy_condition}
		\begin{split}
				\I^{\PP}_p\pa{f(t),g(t)|f_\infty}\leq \mathcal{C}(p,p_1,p_2,\eta) &\rpa{p_1\pa{p_1-1}e_{p_1}\pa{f_0|f_\infty}+1}^{\frac{\eta\pa{2-p}}{p_1}}\\
			&\times \rpa{p_2\pa{p_2-1}e_{p_2}\pa{\abs{g_0}|f_\infty}+1}^{\frac{2}{p_2}}
		\end{split}
	\end{equation}
with 
\begin{equation}\nonumber
		\begin{split}\mathcal{C}(p,p_1,p_2,\eta) = &
			2^{\pa{\eta+1-\frac{d}{2}}(2-p)+1+2d}3^{\frac{d(2-p)}{2}}
			\pa{2p_2}^{\frac{2}{p_2-1}}\\ 
			&\times\pa{\frac{2p_1-1}{p_1-1}}^{\frac{d\eta\pa{p_1-1}\pa{2-p}}{2p_1}} \pa{\frac{2p_2-1}{p_2-1}}^{\frac{d\pa{p_2-1}}{p_2}}\mathrm{p}_{\mathrm{max}},
	\end{split}
	\end{equation}
and where $\mathrm{p}_{\mathrm{max}}$ is as in \eqref{item:hyper}.
\end{enumerate}
\end{theorem}
\begin{remark}\label{rem:also_works_for_f}
	It is worth to note that while part \eqref{item:hyper} only considers the function $g(t)$, the fact that $f_0$ has a finite $p_1-$entropy and is a non-negative function implies that the same result, with $p_2$ replaced by $p_1$, hold for $f(t)$.
\end{remark}
The proof of this theorem will occupy the remaining part of this subsection and will follow similar ideas to the hypercontractivity result in Theorem 4.3, \cite{AEW18}. The singularity of $\psi^{\prime\prime}_p$  for $p\in [1,2)$, the generating functions of the $p-$Fisher information, will require another ingredient which was not considered before -- a generation of a ``close to equilibrium'' lower bound for the solution to the Fokker-Planck equation after enough time has passed. This lower bound will give us the ability to control the singular part of $\psi^{\prime\prime}_p$ when $p\in [1,2)$. The authors are unaware of any reference where this result is stated.\\
The steps we will undertake to rigorously prove Theorem \ref{thm:contractivity} are:
\begin{itemize}
	\item Since for any $L^1$ function (with or without unit mass)
	$$e_p\pa{\abs{f} \Big | f_\infty}=\frac{1}{p\pa{p-1}}\pa{\norm{f}^p_{L^p\pa{\R^d,f^{1-p}_{\infty}}}-p\pa{\norm{f}_{L^1\pa{\R^d}}-1}-1},$$
	we see that $e_p\pa{\abs{f}|f_\infty}<\infty$ for some $p\in (1,2]$ if and only if $\abs{f}^p$ has a finite exponential moment of form $e^{\frac{\pa{p-1}\abs{x}^2}{2}}$. This means that we can interchange the entropic conditions of Theorem \ref{thm:contractivity} with the corresponding exponential moments. 
	\item The explicit expression for the solution of the Fokker-Planck equation under assumptions \eqref{item:cond_semipositive}--\eqref{item:cond_no_invariant_subspace_to_kernel} (which can be found in \S 1 in \cite{Ho67} or \S 6.5 in \cite{RiFP89}) reads as
	\begin{equation}\label{eq:exactsolution}
		f(t,x)=\frac{1}{\pa{2\pi}^{\frac{d}{2}}\sqrt{\det \W(t)}}\int_{\R^d}e^{-\frac{1}{2}\pa{x-e^{-\C t}y}^T \W(t)^{-1}\pa{x-e^{-\C t}y}} f_0(y)dy,
	\end{equation}
	where 
	\begin{equation*}
		\W(t):=2\int_{0}^t e^{-\C s}\D e^{-\C^T s}ds.
	\end{equation*} 
	Moreover,
	\begin{equation}\label{eq:exactsolution_grad}
		\begin{gathered}
			\nabla f(t,x)=-\frac{1}{\pa{2\pi}^{\frac{d}{2}}\sqrt{\det \W(t)}}\int_{\R^d}e^{-\frac{1}{2}\pa{x-e^{-\C t}y}^T \W(t)^{-1}\pa{x-e^{-\C t}y}}\\
			\times \W(t)^{-1}\pa{x-e^{-\C t}y} f_0(y)dy.
		\end{gathered}
	\end{equation}
	We will be able to estimate $e_2$ and $I^{\PP}_2$ after enough time has passed so that $\W(t)\approx \II$ and $\norm{e^{-\C t}}$ is very small.
	\item Showing the generation of an ``almost'' equilibrium lower bound to the solutions of our equation will allow us to proceed the same way as in the case $p=2$ and to show the creation of finite $\I_p^{\PP}$ after enough time has passed. 
\end{itemize}
As the proof that will now follow is quite technical, we recommend \amit{that} the experienced reader \amit{will} skip the remaining part of the subsection, after glancing at the theorem referring to the creation of a lower bound, Theorem \ref{thm:lower_bound}.

We start by making explicit the connection between the $p-$entropies and the exponential moment of a function, at least the implication we will use in this subsection, as well as the technical statement that assures us that $\W(t)$ is approximately $\II$ and $\norm{e^{-\C t}}$ is very small. 

\begin{lemma}\label{lem:hyper_conditions}
	For any $1<p\leq 2$ we have that for any $f\in L^1\pa{\R^d}$ 
	\begin{equation}\label{eq:exp_estimation_by_norm_general_function}
		\int_{\R^d}e^{ \frac{(p-1)\abs{x}^2}{2(2p-1)}}\abs{f(x)}dx \leq \pa{\frac{2p-1}{p-1}}^{\frac{d(p-1)}{2p}}\pa{2p}^{\frac{1}{p-1}}\rpa{p(p-1)e_p\pa{\abs{f}|f_\infty}+1}^{\frac{1}{p}}.
	\end{equation}
	Moreover, if $\int_{\R^d}\abs{f(x)}dx=1$ we can improve the above and get
	\begin{equation}\label{eq:exp_estimation_by_entropy_with_unit_mass}
		\int_{\R^d}e^{\frac{(p-1)\abs{x}^2}{2(2p-1)}}\abs{f(x)}dx \leq  \pa{\frac{2p-1}{p-1}}^{\frac{d(p-1)}{2p}}\rpa{p(p-1)e_{p}\pa{\abs{f}|f_\infty}+1}^\frac{1}{p}.
	\end{equation}
\end{lemma}
The proof is a relatively straightforward computation and as such is left to Appendix \ref{app:b}.

\begin{lemma}\label{lem:proved_propertires_for_hyper_hypo}
	Consider the Fokker-Planck equation \eqref{eq:fokkerplanck} with diffusion and drift matrices $\D$ and $\C$ respectively. Assume in addition that conditions \eqref{item:cond_semipositive}--\eqref{item:cond_no_invariant_subspace_to_kernel} are satisfied and that $\D=\C_s$. Then there exists a constant $\widetilde{c}>0$ that only depends on $\C$ such that 
	$$\norm{e^{-\C t}} \leq \widetilde{c}\pa{1+t^n}e^{-\mu t},$$
	where $\mu$ is as in \eqref{eq:def_of_mu}, $n$ being the maximal defect number associated to $\mu$, and the norm of a matrix is its spectral norm. Moreover, for any $\epsilon>0$ and $t\geq \widetilde{t}(\epsilon)$ where
	\begin{equation}\label{eq:example_for_t_2}
		\widetilde{t}\pa{\epsilon}:=\frac{2}{\mu} \log \pa{\frac{\widetilde{c}\pa{1+\pa{\frac{2n}{\mu e}}^n}}{\epsilon}},
	\end{equation}
	we have that
	\begin{equation}\label{eq:hyperproofIII}
		\norm{e^{-\C t}} \leq \epsilon.
	\end{equation}
	In addition, there exists a constant $\widehat{c}>0$ that only depends on $\C$ such that for any $\epsilon>0$ we have \amit{if} $t\geq \widehat{t}(\epsilon)$ where
	\begin{equation}\label{eq:example_for_t_1}
		\widehat{t}(\epsilon):=\frac{1}{\mu} \log\pa{\frac{\widehat{c}(1+\epsilon)\pa{1+\pa{\frac{2n}{\mu e}}^{2n}}}{\epsilon}},
	\end{equation}
	with $\mu$ and $n$ as above, then
	\begin{equation}\label{eq:hyperproof0}
		\max \pa{\norm{\W(t)-\mathbf{I}},\norm{\W^{-1}(t)-\mathbf{I}}}\leq \epsilon.
	\end{equation} 
	Consequently 
		\begin{equation}\label{eq:hyperproofI}
		e^{-\frac{1}{2}\pa{1+\epsilon}\abs{x-e^{-\C t}y}^2}  \leq e^{-\frac{1}{2}\pa{x-e^{-\C t}y}^T \W(t)^{-1}\pa{x-e^{\C t}y}} \leq e^{-\frac{1}{2}\pa{1-\epsilon}\abs{x-e^{-\C t}y}^2} 
	\end{equation}
	and
	\begin{equation}\label{eq:hyperproofII}
		(1-\epsilon)^d \leq \det{\W(t)} \leq (1+\epsilon)^d.
	\end{equation}
\end{lemma}
The proof of this lemma, which we will not provide here, can be found in Lemmas 4.7, Lemma 4.8 and their proofs in \cite{AEW18}, where the explicit time is found by using the following inequality:
$$\pa{1+t^{2n}}e^{-2\mu t} \leq \pa{1+\pa{\frac{2n}{\mu e}}^{2n}}e^{-\mu t},\qquad \forall t\geq 0.$$

With these tools at hand we turn our attention to proving part \eqref{item:hyper} of Theorem \ref{thm:contractivity} which requires nothing more than Lemma \eqref{lem:proved_propertires_for_hyper_hypo} and the formulae for the solution for the Fokker-Planck equation and its gradient. It is convenient to first prove a stronger statement:

\begin{lemma}\label{lem:hypercontractivity}
	Let $g(t)$ be the solution to the Fokker-Planck equation \eqref{eq:fokkerplanck} with diffusion and drift matrices $\D$ and $\C$ which satisfy Conditions \eqref{item:cond_semipositive}--\eqref{item:cond_no_invariant_subspace_to_kernel} and with initial datum $g_0$. Assume in addition that there exists $\epsilon>0$ such that 
	\begin{equation}\nonumber\label{eq:hyper_condition}
		\int_{\R^d}e^{\epsilon \abs{x}^2}\abs{g_0(x)}dx<\infty.
	\end{equation}
	Then there exists an explicit time, $\widetilde{t}_1(\epsilon)$, which is independent of $g_0$, such that for all $t\geq \widetilde{t}_1(\epsilon)$ we have that the solution to the equation satisfies
	\begin{equation}\label{eq:hypercontractivity_entropy}
		\norm{g(t)}^2_{L^2\pa{\R^d,f_\infty^{-1}}} \leq \pa{\frac{8}{3}}^d\pa{\int_{\R^d}e^{\epsilon \abs{x}^2}\abs{g_0(x)}dx}^2,
	\end{equation}
	and
	\begin{equation}\label{eq:hypercontractivity}
		I_{2}^{\mathbf{I}}(g(t)|f_\infty) \leq 2^{\frac{3d}{2}}\pa{\int_{\R^d}e^{\epsilon \abs{x}^2}\abs{g_0(x)}dx}^2.
	\end{equation}
\end{lemma}
\begin{proof}
	This proof follows the same ideas as those presented in Theorem 4.3 in \cite{AEW18}. For the sake of completion we provide it in Appendix \ref{app:b}.
\end{proof}

\begin{proof}[Proof of part \eqref{item:hyper} of Theorem \ref{thm:contractivity}]
Inequalities \eqref{eq:hypercontractivity_L_2_with_entropy_condition} and \eqref{eq:hypercontractivity_with_entropy_condition} for the case $\PP=\II$ follow immediately from \eqref{eq:hypercontractivity_entropy} and \eqref{eq:hypercontractivity} together with \eqref{eq:exp_estimation_by_norm_general_function} and the choice $\epsilon=\frac{p_2-1}{2\pa{p_2-1}}$.
For a general $\PP$, \eqref{eq:hypercontractivity_with_entropy_condition} follows from the fact that 
$$I^{\PP}_2\pa{f|f_\infty} \leq \mathrm{p}_{\mathrm{max}}I^{\II}_2\pa{f|f_\infty},$$
where $\mathrm{p}_{\mathrm{max}}$ is the largest eigenvalue of $\PP$.\\
 Inequality \eqref{eq:hypercontractivity_entropy_with_entropy_condition} follows from \eqref{eq:hypercontractivity_L_2_with_entropy_condition} together with the fact that for any non-negative $L^1$ function $f$
\begin{equation}\nonumber 
	\begin{gathered}
		e_2\pa{f|f_\infty}=\frac{1}{2}\pa{\norm{f}_{L^2\pa{\R^d,f_\infty^{-1}}}-2\norm{f}_{L^1\pa{\R^d}}+1}
		\leq\max\pa{\norm{f}_{L^2\pa{\R^d,f_\infty^{-1}}}, 1}.
	\end{gathered}
\end{equation}
The proof is thus complete.
\end{proof}

We now turn our attention to the second part of Theorem  \ref{thm:contractivity}. We start, as promised, with the lower bound creation. 
\begin{theorem}\label{thm:lower_bound}
	Let $f(t)$ be a solution to the Fokker-Planck equation \eqref{eq:fokkerplanck} with diffusion and drift matrices $\D$ and $\C$ which satisfy conditions \eqref{item:cond_semipositive}--\eqref{item:cond_no_invariant_subspace_to_kernel} and with a non-negative unit mass initial datum $f_0$. Then for any $\eta>0$ there exists an explicit time $\widehat{t}_1(\eta)>0$ that depends on $\C$ and may become unbounded as $\eta$ goes to zero, such that for any $t\geq \widehat{t}_1\pa{\eta}$  one has that
	\begin{equation}\label{eq:lower_bound}
		f(t,x) \geq C_{\eta,f_0} e^{-\pa{\frac{1}{2}+\eta}\abs{x}^2}
	\end{equation} 
	where $C_{\eta,f_0}>0$ is an explicit constant that only depends on $\eta$ and the initial datum $f_0$. \\	
	If in addition there exists $\epsilon>0$ such that 
	\begin{equation}\label{eq:same_condition_eps}
		\int_{\R^d}e^{\epsilon \abs{x}^2}f_0(x)dx < \infty
	\end{equation}
	then one can choose
	\begin{equation}\label{eq:estimation_on_C_f_0}
		C_{\eta,f_0}=\frac{\pa{2\int_{\R^d}e^{\epsilon \abs{x}^2}f_0(x)dx}^{-\frac{\eta\pa{1+\eta}\pa{1+2\eta}}{8\epsilon}}}{2\pa{\pi(2+\eta)}^{\frac{d}{2}}} .
	\end{equation}
\end{theorem}  

\begin{proof}[Proof of Theorem \ref{thm:lower_bound}]
	We start by recalling the explicit expression for the solution to the system, given in \eqref{eq:exactsolution}. For a given $\epsilon_1>0$, to be chosen later, we consider  $\widehat{t}\pa{\epsilon_1}$, given by \eqref{eq:example_for_t_1}, for which  \eqref{eq:hyperproof0}, \eqref{eq:hyperproofI} and \eqref{eq:hyperproofII} are satisfied when $t\geq \widehat{t}\pa{\epsilon_1}$. As $f_0$ is non-negative we find that for any $R>0$
	$$f(t,x) \geq \frac{1}{\pa{2\pi(1+\epsilon_1)}^{\frac{d}{2}}}\int_{\R^d}e^{-\pa{\frac{1}{2}+\epsilon_1}\abs{x-e^{-\C t}y}^2}f_0(y)dy$$
	$$\geq \frac{1}{\pa{2\pi(1+\epsilon_1)}^{\frac{d}{2}}} e^{-\pa{\frac{1}{2}+\epsilon_1}\abs{x}^2} \int_{\abs{y}\leq R}e^{(1+2\epsilon_1)x^T e^{-\C t}y} e^{-\pa{\frac{1}{2}+\epsilon_1}\abs{e^{-\C t}y}^2}f_0(y)dy.$$
	Restricting $t$ further to be greater than $ \max\pa{\widehat{t}(\epsilon_1),\widetilde{t}(\epsilon_1)}$, where $\widetilde{t}(\epsilon_1)$ is given by \eqref{eq:example_for_t_2}, we see that due to \eqref{eq:hyperproofIII} 
	$$f(t,x) \geq \frac{1}{\pa{2\pi(1+\epsilon_1)}^{\frac{d}{2}}} e^{-\pa{\frac{1}{2}+\epsilon_1}\abs{x}^2}e^{-(1+2\epsilon_1)\epsilon_1 R \abs{x}}e^{-\pa{\frac{1}{2}+\epsilon_1}\epsilon_1^2 R^2}\int_{\abs{y}\leq R}f_0(y)dy.$$
	Choosing $R$ such that 
	\begin{equation}\label{eq:condition_on_R}
		\int_{\abs{y}\leq R}f_0(y)dy \geq \frac{1}{2}
	\end{equation}
	 and using the fact that for any $a>0$ 
	$$a\abs{x} \leq  \abs{x}^2+\frac{a^2}{4}$$ 
	we conclude with the choice of $a:=\pa{1+2\epsilon_1}R$ that
	\begin{equation}\nonumber 
		f(t,x) \geq \frac{1}{2\pa{2\pi(1+\epsilon_1)}^{\frac{d}{2}}} e^{-\pa{\frac{1}{2}+2\epsilon_1}\abs{x}^2}e^{-\frac{(1+2\epsilon_1)^2\epsilon_1 R^2}{4} }e^{-\pa{\frac{1}{2}+\epsilon_1}\epsilon_1^2 R^2}.
	\end{equation}
	Choosing $\epsilon_1:=\frac{\eta}{2}$ and $\widehat{t}_1\pa{\eta}:=\max\pa{\widehat{t}\pa{\frac{\eta}{2}},\widetilde{t}\pa{\frac{\eta}{2}}}$ shows the first statement of the theorem with 
	\begin{equation}\label{eq:C_f_0_non_explicit}
		C_{\eta,f_0}(R)=\frac{1}{2\pa{\pi(2+\eta)}^{\frac{d}{2}}}e^{-\frac{\eta(1+\eta)\pa{1+2\eta}R^2}{8} }.
	\end{equation}

	To show the second statement of the theorem we will use condition \eqref{eq:same_condition_eps} to find $R$ that would satisfy condition \eqref{eq:condition_on_R}. \\
	Since 
	$$\int_{\abs{y}>R}f_0(y)dy \leq e^{-\epsilon R^2}\int_{\abs{y}>R}e^{\epsilon \abs{y}^2}f_0(y)dy \leq e^{-\epsilon R^2}\int_{\R^d}e^{\epsilon \abs{y}^2}f_0(y)dy$$
	we see that $\int_{\abs{y}\leq R}f_0(y)dy   \geq \frac{1}{2}$ when\footnote{Note that since $$\int_{\R^d}e^{\epsilon\abs{y}^2}f_0(y)dy \geq \int_{\R^d}f_0(y)dy=1$$ we have that $R^2$ is well defined.}
	$$R^2 := \frac{1}{\epsilon}\log\pa{2 \int_{\R^d}e^{\epsilon\abs{y}^2}f_0(y)dy}.$$
	Plugging this expression back in \eqref{eq:C_f_0_non_explicit} gives \eqref{eq:estimation_on_C_f_0} and completes the proof.
\end{proof}
\begin{corollary}\label{cor:lower_bound_psi_p}
	Under the same assumptions as Theorem \ref{thm:lower_bound} we see that if there exists $\epsilon>0$ such that
	\begin{equation}\nonumber
		\int_{\R^d}e^{\epsilon \abs{x}^2}f_0(x)dx < \infty
	\end{equation}
	then for any $\eta>0$ there exists an explicit time $\widehat{\widehat{t}}_1(\epsilon,\eta)>0$ that may become unbounded as $\eta$ goes to zero, such that for any $t\geq \widehat{\widehat{t}}_1(\epsilon,\eta)$  one has that
	\begin{equation}\label{eq:lower_bound_psi_p}
		\psi_p^{\prime\prime}\pa{\frac{f(t,x)}{f_\infty(x)}}\leq 2^{\pa{\eta+1-\frac{d}{2}}(2-p)}3^{\frac{d(2-p)}{2}}\pa{\int_{\R^d}e^{\epsilon \abs{y}^2}f_0(y)dy}^{\eta(2-p)} e^{\pa{2-p}\eta\abs{x}^2},
	\end{equation}
	for $1\leq p \leq 2$.
\end{corollary}
\begin{proof}
	Using Theorem \ref{thm:lower_bound}, we can find $\widehat{t}_1\pa{\eta_1}$, for $\eta_1>0$ to be chosen shortly, such that for any $t\geq \widehat{t}_1\pa{\eta_1}$ we have that \eqref{eq:lower_bound} and \eqref{eq:estimation_on_C_f_0} hold with $\eta$ replaced by $\eta_1$. Since $\psi^{\prime\prime}_p\pa{y}=y^{p-2}$ and $1\leq p\leq 2$ we have that for any $t\geq \widehat{t}_1\pa{\eta_1}$
	\begin{equation}\nonumber
		\begin{split}
			\psi_p^{\prime\prime}\pa{\frac{f(t,x)}{f_\infty(x)}}&\leq  2^{\pa{\frac{\eta_1\pa{1+\eta_1}\pa{1+2\eta_1}}{8\epsilon}+1-\frac{d}{2}}(2-p)}\pa{2+\eta_1}^{\frac{d(2-p)}{2}}\\
			& \times \pa{\int_{\R^d}e^{\epsilon \abs{y}^2}f_0(y)dy}^{\frac{\eta_1\pa{1+\eta_1}\pa{1+2\eta_1}(2-p)}{8\epsilon}} e^{\pa{2-p}\eta_1\abs{x}^2}.
		\end{split}
	\end{equation}
	For $\eta_1\pa{\epsilon,\eta}=\min\pa{1,\eta,\eta \epsilon}$ we find that\footnote{Here we used the fact that for any unit mass non-negative function $f_0$
		$$\int_{\R^d}e^{\alpha \abs{y}^2}f_0(y)dy \geq \int_{\R^d}f_0(y)dy=1,$$
		for any $\alpha>0$.} the above implies \eqref{eq:lower_bound_psi_p} for $\widehat{\widehat{t}}_1(\epsilon,\eta):=\widehat{t}_1\pa{\eta_1\pa{\eta,\epsilon}}$.
\end{proof}

Much like with the first part of Theorem \ref{thm:contractivity} we turn our attention now to a stronger version of the second part of the theorem:

\begin{lemma}\label{lem:hypocontractivity}
	Let $f(t),g(t)$ be solutions to the Fokker-Planck equation \eqref{eq:fokkerplanck} with diffusion and drift matrices $\D$ and $\C$ which satisfy conditions \eqref{item:cond_semipositive}--\eqref{item:cond_no_invariant_subspace_to_kernel} and with initial data $f_0,g_0$ such that $f_0$ is non-negative. Assume in addition that there exist $\epsilon_1>0$ and $\epsilon_2>0$ such that 
		\begin{equation}\label{eq:condition_on_f_0_hypo}
		\int_{\R^d}e^{\epsilon_1 \abs{x}^2}f_0(x)dx < \infty.
	\end{equation}
	and
		\begin{equation}\label{eq:condition_on_g_0_hypo}
		\int_{\R^d}e^{\epsilon_2 \abs{x}^2}\abs{g_0(x)}dx < \infty,
	\end{equation}
	Then for any $\eta>0$ there exists an explicit time, $t_2(\epsilon_1,\epsilon_2,\eta)>0$, that may become unbounded as $\epsilon_1$, $\epsilon_2$ or $\eta$ go to zero, such that for all $t\geq t_2(\epsilon_1,\epsilon_2,\eta)$ we have that
	\begin{equation}\label{eq:hypocontractivity}
		\begin{split}
				\I^{{\bf{I}}}_p\pa{f(t),g(t)|f_\infty}\leq &2^{\pa{\eta+1-\frac{d}{2}}(2-p)+1+2d}3^{\frac{d(2-p)}{2}} \\
				&\times\pa{\int_{\R^d}e^{\epsilon_1 \abs{x}^2}f_0(x)dx}^{\eta(2-p)}\pa{\int_{\R^d}e^{\epsilon_2\abs{x}^2} \abs{g_0(x)}dx}^2
		\end{split}
	\end{equation}
	for any given $1\leq p < 2$. 
\end{lemma}
\begin{remark}\label{rem:time_shift}
	It is clear to see that if we take $t_2$ given in Lemma \ref{lem:hypocontractivity}  and add a fixed $T_0>0$ to, then \eqref{eq:hypocontractivity} remains true by replacing $f_0$ and $g_0$ with $f\pa{T_0,\cdot}$ (which becomes even strictly positive) and $g\pa{T_0,\cdot}$ respectively. 
\end{remark}

\begin{proof}[Proof of Lemma \ref{lem:hypocontractivity}]
Using Corollary \ref{cor:lower_bound_psi_p} we see that for all $t \geq\widehat{\widehat{t}}_1\pa{\epsilon_1,\eta_1}$, where $\widehat{\widehat{t}}_1\pa{\epsilon_1,\eta_1}$ is given in the corollary for $\eta_1=\eta_1\pa{\eta}$ to be chosen shortly, 
\begin{equation}\nonumber
\begin{gathered}
\I^{{\bf{I}}}_p\pa{f(t),g(t)|f_\infty} =\int_{\R^d} \psi^{\prime\prime}_p\pa{\frac{f(t,x)}{f_\infty(x)}}\abs{\nabla\pa{\frac{g(t,x)}{f_\infty(x)}}}^2f_\infty(x)dx \\
\leq 2^{\pa{\eta_1+1-\frac{d}{2}}(2-p)}3^{\frac{d(2-p)}{2}}\pa{\int_{\R^d}e^{\epsilon_1 \abs{x}^2}f_0(x)dx}^{\eta_1(2-p)}\int_{\R^d}\abs{\nabla\pa{\frac{g(t,x)}{f_\infty(x)}}}^2 \pa{e^{\pa{2-p}\eta_1 \abs{x}^2}f_\infty(x)}dx\\
=2^{\pa{\eta_1+1-\frac{d}{2}}(2-p)}3^{\frac{d(2-p)}{2}}\pa{\int_{\R^d}e^{\epsilon_1 \abs{x}^2}f_0(x)dx}^{\eta_1(2-p)}\int_{\R^d}\abs{\nabla g\amit{\pa{t,x}} + x g(\amit{t,x})}^2 \pa{e^{\pa{2-p}\eta_1 \abs{x}^2}f^{-1}_\infty(x)}dx.
\end{gathered}
\end{equation} 
From this point onwards most of the estimates follow much like the proof of Lemma \ref{lem:hypercontractivity} which can be found in Appendix \ref{app:b}. As such, we will skip a few simple calculations. \\
From \eqref{eq:exactsolution} and \eqref{eq:exactsolution_grad} one can show that for a fixed $\eta_2=\eta_2\pa{\epsilon_2}$, to be chosen shortly, if $t_3\pa{\eta_2}:=\max\pa{\widetilde{t}\pa{\eta_2},\widehat{t}\pa{\eta_2}}$ is chosen such that \eqref{eq:hyperproofIII}, \eqref{eq:hyperproofI}, and \eqref{eq:hyperproofII} hold for $\epsilon=\eta_2$  and for any $t\geq t_3\pa{\eta_2}$, then
\begin{equation}\nonumber
	\begin{gathered}
		\abs{\nabla g(t,x) + xg(t,x)}  \\ \leq\frac{1}{\pa{2\pi\pa{1-\eta_2}}^{\frac{d}{2}}}\int_{\R^d}e^{-\frac{1}{2}\pa{1-\eta_2}\abs{x-e^{-\C t}y}^2}
		\pa{\eta_2\abs{x-e^{-\C t}y}+\abs{e^{-\C t}y}} \abs{g_0(y)}dy.
	\end{gathered}
\end{equation}
Using the following Minkowski integral inequality for $q\geq 1$
\begin{equation}\nonumber 
	\begin{split}
		\left(\int_{\R^d}\abs{\int_{\R^d}F(y,x)d\mu_1( y)}^q  d\mu_2(x) \right)^{\frac{1}{q}} 
		\leq \int_{\R^d} \pa{ \int_{\R^d}\abs{F(y,x)}^q d\mu_{2}( x) }^{\frac{1}{q}}d\mu_1(y),
	\end{split}
\end{equation}
and the elementary inequality 
\begin{equation}\nonumber
	\sqrt{\eta}\abs{x-e^{-\C t}y} \leq e^{\frac{\eta}{2}\abs{x-e^{-\C t}y}^2}
\end{equation}
one finds that
$$\int_{\R^d}\abs{\nabla g(\amit{t,x}) + x g(\amit{t,x})}^2 \underbrace{\pa{e^{\pa{2-p}\eta_1 \abs{x}^2}f^{-1}_\infty(x)}dx}_{:=d\mu_2(x)} \leq \frac{2}{\pa{2\pi\pa{1-\eta_2}^2}^{\frac{d}{2}}}$$

$$\times \pa{\int_{\R^d}\pa{\int_{\R^d}\underbrace{e^{-\pa{1-2\eta_2}\abs{x-e^{-\C t}y}^2}}_{:=F(y,x)^2}e^{\pa{\pa{2-p}\eta_1+\frac{1}{2}}\abs{x}^2}
		dx}^{\frac{1}{2}} \underbrace{\sqrt{\eta_2+\abs{e^{-\C t }y}^2}\abs{g_0(y)}dy}_{:=d\mu_1(y)}}^2$$
	$$\leq \frac{2\eta_2}{\pa{1-\eta_2}^d\pa{1-4\eta_2-2(2-p)\eta_1}^{\frac{d}{2}}}\pa{\int_{\R^d}\sqrt{1+\eta_2\abs{y}^2}e^{\frac{\pa{1-2\eta_2}\pa{1+2(2-p)\eta_1}}{1-4\eta_2 -2(2-p)\eta_1 }\frac{\eta_2^2}{2} \abs{y}^2} \abs{g_0(y)}dy}^2 $$
	and as such
\begin{equation}\label{eq:fisher_hypo_proof_I}
\begin{gathered}
\I^{{\bf{I}}}_p\pa{f(t),g(t)|f_\infty}\leq  \frac{2^{\pa{\eta_1+1-\frac{d}{2}}(2-p)+1}3^{\frac{d(2-p)}{2}}\eta_2}{\pa{1-\eta_2}^d\pa{1-4\eta_2-2(2-p)\eta_1}^{\frac{d}{2}}} \\
\times \pa{\int_{\R^d}e^{\epsilon_1 \abs{x}^2}f_0(x)dx}^{\eta_1(2-p)}\pa{\int_{\R^d}e^{\pa{\frac{\pa{1-2\eta_2}\pa{1+2(2-p)\eta_1}}{1-4\eta_2 -2(2-p)\eta_1 }\frac{\eta_2^2}{2} +\frac{\eta_2}{2} }\abs{x}^2} \abs{g_0(x)}dx}^2.
\end{gathered}
\end{equation} 
To complete the proof we need to choose $\eta_1\pa{\epsilon_2}$ and $\eta_2\pa{\epsilon_2}$ such that 
$$\frac{\pa{1-2\eta_2}\pa{1+2(2-p)\eta_1}}{1-4\eta_2 -2(2-p)\eta_1 }\frac{\eta_2^2}{2} +\frac{\eta_2}{2} \leq \epsilon_2.$$
As $0 \leq (2-p) \leq 1$ for any $1\leq p \leq 2$, we see that if $\eta_1,\eta_2 \leq \frac{1}{8}$ then
$$1-4\eta_2 -2(2-p)\eta_1 \geq \frac{1}{4}$$
and 
$$\frac{\pa{1-2\eta_2}\pa{1+2(2-p)\eta_1}}{1-4\eta_2 -2(2-p)\eta_1 }\frac{\eta_2^2}{2} +\frac{\eta_2}{2} \leq \pa{\frac{15}{64}+\frac{1}{2}}\eta_2=\frac{47\eta_2}{64}.$$
Thus, choosing $\eta_2:=\min\pa{\frac{1}{8},\epsilon_2}$ and $\eta_1:=\min\pa{\frac{1}{8},\eta}$ we obtain that\footnote{Here, again, we have used the fact that $\int_{\R^d}e^{\epsilon_1 \abs{y}^2}f_0(y)dy \geq 1$, as well as the \amit{estimate}
$$\frac{1}{\pa{1-\eta_2}^2\pa{1-4\eta_2-2(2-p)\eta_1}} \leq \frac{256}{49} < 2^4.$$}
\begin{equation}\nonumber
	\begin{gathered}
		\I^{{\bf{I}}}_p\pa{f(t),g(t)|f_\infty}\leq 2^{\pa{\eta+1-\frac{d}{2}}(2-p)+1+2d}3^{\frac{d(2-p)}{2}} \\
		\times \pa{\int_{\R^d}e^{\epsilon_1 \abs{x}^2}f_0(x)dx}^{\eta(2-p)}\pa{\int_{\R^d}e^{\epsilon_2\abs{x}^2} \abs{g_0(x)}dx}^2\\
	\end{gathered}
\end{equation}
for $t\geq t_2\pa{\epsilon_1,\epsilon_2,\eta}:=\max\pa{\widehat{\widehat{t}}_1\pa{\epsilon_1,\eta_1\pa{\eta}},t_3\pa{\eta_2\pa{\epsilon_2}}}$,
which is the desired result.
\end{proof}

We can now conclude this subsection, and with it this section, with the proof of part \eqref{item:hypo} of Theorem \ref{thm:contractivity}:

\begin{proof}[Proof of part \eqref{item:hypo} of Theorem \ref{thm:contractivity}]
	The case $\PP=\II$ is an immediate consequence of lemmas \ref{lem:hypocontractivity} and \ref{lem:hyper_conditions} with the choices
	$$\epsilon_1:=\frac{p_1-1}{2(2p_1-1)},\qquad \epsilon_2:=\frac{p_2-1}{2(2p_2-1)}.$$
	Since $f$ has a unit mass, we have used \eqref{eq:exp_estimation_by_entropy_with_unit_mass} for it instead of the general \eqref{eq:exp_estimation_by_norm_general_function} which we have used for $g$.\\
	For a general $\PP$, \eqref{eq:hypocontractivity_with_entropy_condition} follows from the fact that for any $p\in \rpa{1,\amit{2}}$
		$$\I^{\PP}_p\pa{f,g|f_\infty} \leq \mathrm{p}_{\mathrm{max}}\I^{\II}_p\pa{f,g|f_\infty},$$
		where $\mathrm{p}_{\mathrm{max}}$ is the largest eigenvalue of $\PP$.\\
\end{proof}

\begin{remark}\label{rem:small_improvements}
	Some of the constants in our main theorem for this subsection, Theorem \ref{thm:contractivity}, could be improved when $p_1$ and/or $p_2$ equal to $2$. We will not show this in this work to simplify our overall presentation.
\end{remark}

\section{The Convergence to Equilibrium in Fisher Information}\label{sec:convergence}
In this section we will build upon everything that we have done so far and prove a variant of our main theorem, Theorem \ref{thm:main}, for the Fisher information. In the next section we will use this variant to prove Theorem \ref{thm:main}. 
\begin{theorem}\label{thm:main_for_fisher}
Let $\D$ be a positive definite matrix and $\C$ be a positive stable matrix. Define
\begin{equation}\nonumber
	\mu=\min\br{\Re\pa{\lambda}\;|\;\lambda\text{ is an eigenvalue of }\;\C},
\end{equation}
and let $n$ is the maximal defect of the eigenvalues associated to $\mu$. Let $f(t)$ be a solution to \eqref{eq:fokkerplanck} with unit mass initial data $f_0\in L^1_+\pa{\R^d}$ and finite $p-$entropy for some $1<p\leq 2$. In addition, let $\PP$ be a positive definite matrix that satisfies\footnote{See Remark \ref{rem:on_P} for more details on the existence of such matrix $\PP$.} 
$$\C^T {\bf{P}}+{\bf{P}}\C \geq 2\pa{\mu-\nu} {\bf P} $$
for some $0\leq \nu<\mathrm{d}_{\mathrm{min}}(p-1)$, where $\mathrm{d}_{\mathrm{min}}$ is the smallest eigenvalues of $\D$.\\
Then, for any $\epsilon>0$ there exists an explicit $\tau_0(p,\epsilon)\geq 0$ that is independent of $f_0$ and that may become unbounded as $\epsilon$ goes to zero, and an explicit $c_{p,\epsilon}$ that is independent of $f_0$ such that for all $t\geq \tau_0\pa{p,\epsilon}$
\begin{equation}\label{eq:main_for_fisher}
	\begin{gathered}
	I_p^{\PP}\pa{f(t)|f_\infty} \leq c_{p,\epsilon} \mathrm{p}_{\mathrm{max}}\rpa{p(p-1)e_p\pa{f_0|f_\infty}+1}^{\frac{2}{p}\pa{1+\epsilon(2-p)}}\\
	\times\pa{1+t^{2n}} e^{-2\mu t}.
	\end{gathered}
\end{equation}

 If $f_0$ has finite $2-$entropy then the above can be improved in the following way: There exists an explicit $\tau_0> 0$ that is independent of $f_0$ such that for any $t\geq \tau_0$ 
\begin{equation}\label{eq:main_for_fisher_improved}
	\begin{gathered}
		I_2^{\PP}\pa{f(t)|f_\infty} \leq c\, \mathrm{p}_{\mathrm{max}}e_2\pa{f_0|f_\infty}
		\pa{1+t^{2n}} e^{-2\mu t},
	\end{gathered}
\end{equation}
where $c$ is fixed constant. It is important to note that $\tau_0$ can't be zero in general as $I_2^{\PP}\pa{f_0|f_\infty}$ may be infinite.
\end{theorem}


The proof of Theorem \ref{thm:main_for_fisher} is quite involved, but we have set up the roadmap to it in points \eqref{item:description_of_proof_I}-\eqref{item:description_of_proof_IV} in \S\ref{subsec:main_results}. The main steps we will employ to prove it are:\\
\textbf{Step 1:} Due to our hypercontractivity result (part \eqref{item:hyper} of Theorem \ref{thm:contractivity}) we know that after enough time has passed our solution $f(t)$ is in $L^2\pa{\R^d,f_\infty^{-1}}$.
 This allows us to use the spectral properties of the Fokker-Planck operator $L$, and in particular the improved decay away from the first eigenspace $V_1$ (as expressed in Theorem \ref{thm:Fisher_improved_decay}). This motivates us to \textit{decompose} our solution to its projection on $V_1$, $f_1(t)$, and its orthogonal part, $f_2(t)$. It is worth to note that when $p=2$ we have that $f(t)\in L^2\pa{\R^d,f_\infty^{-1}}$ for all $t\geq 0$.\\
\textbf{Step 2:} As $V_1$ is finite dimensional we will be able to investigate the long time behaviour of $\I^{\PP}_p\pa{f(t),f_1(t)|f_\infty}$ for any $p\in [1,2]$ explicitly. \\
\textbf{Step 3:} Using our upper-, and lower-contractivity results, parts \eqref{item:hyper} and \eqref{item:hypo} of Theorem \ref{thm:contractivity}, together with the interpolation inequality of Theorem \ref{thm:interpolation_p} and the long time behaviours of the (generalised) Fisher information described in Theorems \ref{thm:Fisher_decay} and \ref{thm:Fisher_improved_decay} we will be able to show that $\I^{\PP}_p\pa{f(t),f_2(t)|f_\infty}$ decays significantly faster than $\I^{\PP}_p\pa{f(t),f_1(t)|f_\infty}$ when one chooses $\PP$ in a suitable way, reminiscent of the choice made in \cite{AE}.\\
\textbf{Step 4:} Combining the above with the inequality 
\begin{equation}\label{eq:Fisher_bound_by_generalised_fisher}
		I^{\PP}_p\pa{f|f_\infty}=\I_p^{\PP}\pa{f,f_1+f_2|f_\infty} \leq 2\pa{\I_p^{\PP}\pa{f,f_1|f_\infty}+\I_p^{\PP}\pa{f,f_2|f_\infty}}
\end{equation}
will give the desired result. \\
We shall now go through these steps rigorously.
\subsection{Step 1: Decomposition.}\label{subsec:decomp} We remind the reader that the mutually orthogonal spaces $\br{V_m}_{m\in\N\cup\br{0}}\subset L^2\pa{\R^d,f_\infty^{-1}}$ are defined as 
\begin{equation}\nonumber
	V_m:=\span\br{\partial_{x_1}^{\alpha_1}\dots \partial_{x_d}^{\alpha_d}f_\infty(x)\ \Big|\ \alpha_1,\dots,\alpha_d\in \N\cup\br{0}, \sum_{i=1}^d \alpha_i=m},
\end{equation}
which in our case where $f_\infty(x)$ is the standard Gaussian results in
\begin{equation}\nonumber
	V_m=\span\br{x_1^{\alpha_1}\dots x_d^{\alpha_d}f_\infty(x)\ \Big|\ \alpha_1,\dots,\alpha_d\in \N\cup\br{0}, \sum_{i=1}^d \alpha_i=m}.
\end{equation}
\amit{Since} for any $i,j=1,\dots,d$ we have that 
$$\int_{\R^d}\pa{x_i f_\infty(x)}\pa{x_j f_\infty(x)} f_\infty(x)^{-1}dx=\frac{1}{\pa{2\pi}^{\frac{d}{2}}}\int_{\R^d}x_i x_je^{-\frac{\abs{x}^2}{2}}dx=\delta_{i,j},$$
we conclude that 
\begin{equation}\label{eq:def_of_xi}
	\xi_i(x):=x_i f_\infty(x),\qquad i=1,\dots, d
\end{equation}
is an orthonormal basis for $V_1$. Thus, any $f\in L^2\pa{\R^d,f_\infty^{-1}}$ can be written as
\begin{equation}\label{eq:def_of_f_1_ad_f_2}
	f=P_{V_1}f+\pa{I-P_{V_1}f}=:f_1+f_2
\end{equation}
where $P_{V_1}$ is the orthogonal projection on  $V_1$ which is given by
$$P_{V_1}f(x):=\sum_{i=1}^d \inner{f,\xi_i}_{L^2\pa{\R^d,f_\infty^{-1}}}\xi_i(x).$$
\subsection{Step 2: The Behaviour of the Generalised Fisher Information on the ``Spectrally Close'' Part of the Decomposition.}\label{subsec:behaviour_on_V_1} In this step we will prove the following:
\begin{theorem}\label{thm:fisher_decay_on_V_1}
	Consider the Fokker-Planck equation \eqref{eq:fokkerplanck} whose diffusion and drift matrices satisfy conditions \eqref{item:cond_semipositive}--\eqref{item:cond_no_invariant_subspace_to_kernel} and let $f(t)$ be the solution to the equation with a unit mass initial datum $f_0\in L^1_{+}\pa{\R^d}$. Assume in addition that $e_p\pa{f_0|f_\infty}<\infty$ for some $1<p\leq 2$. Then for any positive definite matrix $\PP$, any $1\leq \widetilde{p} \leq 2$, and any $\epsilon>0$ there exists an explicit time $\tau_1(p,\widetilde{p},\epsilon)\geq 0$, that may become unbounded as $\epsilon$ goes to zero, such that for any $t\geq \tau_1(p,\widetilde{p},\epsilon)$ 
	\begin{equation}\label{eq:fisher_decay_on_V_1}
		\begin{gathered}
	\I_{\widetilde{p}}^{\PP}\pa{f(t),f_1(t)|f_\infty} \leq 
	c_{p,\widetilde{p},\epsilon}	\mathrm{p}_{\mathrm{max}} \\
	\times \rpa{p(p-1)e_p\pa{f_0|f_\infty}+1}^{\frac{2}{p} \pa{1+\epsilon(2-\widetilde{p})}}\pa{1+(t-\tau_1)^{2n}} e^{-2\mu (t-\tau_1)}, 
		\end{gathered}
\end{equation}
where $c_{p,\widetilde{p},,\epsilon}$ is an explicit constant, depending only on $p$, $\widetilde{p}$, $\epsilon$, and $\C$. 
In particular, choosing $\widetilde{p}=p$ in \eqref{eq:fisher_decay_on_V_1} when $1<p\leq 2$ we conclude that for all $t\geq \tau_1\pa{p,\epsilon}$ we have that
	\begin{equation}\label{eq:fisher_decay_on_V_1_with_I_p}
	\begin{gathered}
		\I_{p}^{\PP}\pa{f(t),f_1(t)|f_\infty} \leq 
		c_{p,\epsilon} 	\mathrm{p}_{\mathrm{max}}\\
\times		\rpa{p(p-1)e_p\pa{f_0|f_\infty}+1}^{\frac{2}{p}\pa{1+\epsilon(2-p)}}\pa{1+(t-\tau_1)^{2n}} e^{-2\mu (t-\tau_1)}.
	\end{gathered}
\end{equation}
\end{theorem}

The proof of this theorem relies heavily on the next lemma:
\begin{lemma}\label{lem:LonV1}
Consider the space $V_1$ defined in Theorem \ref{thm:spectraldecomposition} and its orthonormal basis $\br{\xi_i}_{i=1,\dots,d}$, given in \eqref{eq:def_of_xi}. Then, the matrix representation of the Fokker-Planck operator $L$ with respect to this basis is $-\C$.
Moreover, there exists a constant $C>0$, depending only on $\C$, such that if $f(t)\in L^2\pa{\R^d,f_\infty^{-1}}$ for all $t\geq \overline{t_0}$ for some $\overline{t_0}\geq 0$, then
\begin{equation}\label{eq:LonV1decay}
\abs{\nabla\pa{\frac{f_1(t,x)}{f_\infty(x)}}} \leq C \norm{f_1\pa{\overline{t_0}}}_{L^2\pa{\R^d,f_\infty^{-1}}} \pa{1+(t-\overline{t_0})^n} e^{-\mu \pa{t-\overline{t_0}}},\quad \forall t\geq \overline{t_0}.
\end{equation}
\end{lemma}
It is worth to note that the \amit{right} hand side of \eqref{eq:LonV1decay} is independent of $x$. Variants of this lemma can be found in Theorem 3.11 of \cite{AEW18} and Proposition 4.9 of \cite{ASS}. For the sake of completeness we provide the proof here.
\begin{proof}
Using \eqref{eq:FPrecast-gen} we see that for any $i\in\br{1,\dots,d}$
\begin{equation}\nonumber 
\begin{gathered}
L \xi_i(x) = \text{div}\pa{f_\infty(x) \C \nabla \pa{x_i}}=\text{div}\pa{f_\infty(x)\C e_i}\\
=\nabla f_\infty(x)^T \C e_i
=-\pa{\xi_1(x),\dots,\xi_d(x)}\C e_i= -\sum_{j=1}^{d} \C_{ji}\xi_{j}(x),
\end{gathered}
\end{equation}
and consequently
\begin{equation}\label{eq:L_is_minus_C_on_V_1}
	L\pa{\sum_{i=1}^d a_i \xi_i(x)}= -\sum_{i,j=1}a_i \C_{ji}\xi_j(x) = \sum_{i=1}^d\pa{\sum_{j=1}^d\pa{-\C}_{ij}a_j}\xi_i(x)
\end{equation}
which shows the first part of the lemma. \\
We now turn our attention to the second part. We start by denoting by $a_i(t):=\inner{f(t),\xi_i}_{L^2\pa{\R^d,f_\infty^{-1}}}$ and $f_1(t,x)=\sum_{i=1}^d a_i(t)\xi_i(x)$.
Due to  \eqref{eq:L_is_minus_C_on_V_1} we have that
$$\dot{a_i}(t)=\inner{Lf_1(t),\xi_i}_{L^2\pa{\R^d,f_\infty^{-1}}}=-\sum_{j,k=1}^d a_j(t)\C_{k j}\inner{\xi_k,\xi_i}_{L^2\pa{\R^d,f_\infty^{-1}}}=-\sum_{j=1}^d \C_{ij}a_j(t)\quad \forall t\geq \overline{t_0},$$
which can be rewritten as
\begin{equation}\label{eq:appropriate_ODE}
	\dot{{\bf{a}}}(t)=-\C {\bf{a}}(t),\quad \forall t\geq\overline{t_0},
\end{equation}
where ${\bf{a}}(t):=\pa{a_1(t),\dots,a_d(t)}^T$.\\ 
Since $\abs{{\bf{a}}(t)}=\norm{f_1(t)}_{L^2\pa{\R^d,f^{-1}_{\infty}}}$, \eqref{eq:appropriate_ODE} implies that there exists an explicit constant $C>0$ that depends only on $\C$ such that
\begin{equation}\label{eq:f1timeevo}
\abs{{\bf{a}}(t)} \leq C \norm{f_1\pa{\overline{t_0}}}_{L^2\pa{\R^d,f_\infty^{-1}}} \pa{1+\pa{t-\overline{t_0}}^n} e^{-\mu \pa{t-\overline{t_0}}}.
\end{equation}
As
$$\nabla\pa{\frac{f_1(t,x)}{f_\infty(x)}}=\sum_{j=1}^d a_j(t) \nabla\pa{\frac{\xi_j(x)}{f_\infty(x)}}=\sum_{j=1}^d a_j(t) \nabla\pa{x_j}={\bf{a}}(t),$$
the proof is complete.
\end{proof}

\begin{proof}[Proof of Theorem \ref{thm:fisher_decay_on_V_1}] 
	Due to part \eqref{item:hyper} of Theorem \ref{thm:contractivity} and Remark \ref{rem:also_works_for_f} we know that for any $p\in (1,2]$ we can find $t_1\pa{p}$ such that if $t\geq t_1\pa{p}$ then $f(t)\in L^2\pa{\R^d,f_\infty^{-1}}$ and 
	$$\norm{f(t)}_{L^2\pa{\R^d,f_\infty^{-1}}} \leq \mathcal{A}(p) \rpa{p\pa{p-1}e_{p}\pa{f_0|f_\infty}+1}^{\frac{2}{p}}.$$
	Next we will show that the projection $P_{V_1}$ is well defined for any function $h\in L_+^1\pa{\R^d}$ such that $e_{p}\pa{h|f_\infty}<\infty$ for some $p\in (1,2]$. Consequently, since the fact that  $e_p\pa{f_0|f_\infty}<\infty$ implies that the solution to the Fokker-Planck equation with initial datum $f_0$, $f(t)$, satisfies $e_p\pa{f(t)|f_\infty}<\infty$ (see Theorem \ref{thm:anton_erb}), we will conclude that $f_1(t)\in L^2\pa{\R^d,f_\infty^{-1}}$ for any $t\geq 0$. Indeed, due to Lemma \ref{lem:hyper_conditions} a function in $h\in L^1_+\pa{\R^d}$ with finite $p-$entropy where $p\in (1,2]$ must have an exponential moment. This implies that 
		$$\abs{\inner{h,\xi_i}_{L^2\pa{\R^d,f_\infty^{-1}}}}= \abs{\int_{\R^d}x_i h(x)dx }< \infty,$$
	which is the desired result.\\
	We have thus shown that $f(t)\in L^2\pa{\R^d,f_\infty^{-1}}$ for any $t\geq t_1\pa{p}$ and $f_1(t)\in L^2\pa{\R^d,f_\infty^{-1}}$ for any $t\geq 0$\footnote{We will, in fact, only use $t\geq t_1\pa{p}$ but wanted to indicate here that the inclusion of $f_1(t)$ in $L^2\pa{\R^d,f_\infty^{-1}}$ is instantaneous.} and consequently $f_2(t)\in L^2\pa{\R^d,f_\infty^{-1}}$ for any $t\geq t_1\pa{p}$.
	Moreover, since $f_1(t)\perp f_2(t)$ we see that for any $t\geq t_1(p)$
	\begin{equation}\label{eq:L_2_for_f_1_proof}
		\norm{f_1(t)}_{L^2\pa{\R^d,f_\infty^{-1}}} \leq \norm{f(t)}_{L^2\pa{\R^d,f_\infty^{-1}}} \leq \mathcal{A}(p) \rpa{p\pa{p-1}e_{p}\pa{f_0|f_\infty}+1}^{\frac{2}{p}}.
	\end{equation}
	Denoting by
	$$\eta\pa{\epsilon,\widetilde{p}}: = \min\pa{2\epsilon,\frac{1}{4(2-\widetilde{p})}}$$
	and utilising  Lemma \ref{lem:hyper_conditions}, Corollary \ref{cor:lower_bound_psi_p} with the choice $\epsilon_p=\frac{p-1}{2\pa{2p-1}}$, Lemma \ref{lem:LonV1} with $\overline{t_0}=t_1(p)$, and \eqref{eq:L_2_for_f_1_proof} we see that for any $t\geq \tau_1\pa{p,\widetilde{p},\epsilon}:=\max\pa{\widehat{\widehat{t}}\pa{\epsilon_p,\eta\pa{\epsilon,\widetilde{p}}},t_1(p)}$
	 \begin{equation}\nonumber
		\begin{gathered}
			\I^{\PP}_{\widetilde{p}}\pa{f(t),f_1(t)}f_\infty) \leq \mathrm{p}_{\mathrm{max}}\I^{{\bf{I}}}_{\widetilde{p}}\pa{f(t),f_1(t)|f_\infty}\\
			=\mathrm{p}_{\mathrm{max}}\int_{\R^d}\psi_{\widetilde{p}}^{\prime\prime}\pa{\frac{f(t,x)}{f_\infty(x)}}\abs{\nabla\pa{\frac{f_1(t,x)}{f_\infty(x)}}}^2 f_\infty(x)dx\\
			\leq \mathrm{p}_{\mathrm{max}}\widetilde{c}_{p,\widetilde{p},\epsilon}\norm{f_1(\tau_1)}_{L^2\pa{\R^d,f^{-1}_\infty}}^2 \rpa{p\pa{p-1}e_p\pa{f_0|f_\infty}+1}^{\frac{\eta\pa{2-\widetilde{p}}}{p}}\\
			\times \pa{\int_{\R^d}e^{\pa{2-\widetilde{p}}\eta\abs{x}^2}e^{-\frac{\abs{x}^2}{2}}dx}\pa{1+(t-\tau_1)^{2n}} e^{-2\mu (t-\tau_1)}\\
			\leq \mathrm{p}_{\mathrm{max}}c_{p,\widetilde{p},\epsilon}\rpa{p\pa{p-1}e_p\pa{f_0|f_\infty}+1}^{\frac{2}{p}\pa{1+\frac{\eta\pa{2-\widetilde{p}}}{2}}}\pa{1+(t-\tau_1)^{2n}} e^{-2\mu (t-\tau_1)},
		\end{gathered}
	\end{equation}	
	with some geometric constants $\widetilde{c}_{p,\widetilde{p},\epsilon}$ and $c_{p,\widetilde{p},\epsilon}$. Since $\eta \leq 2\epsilon$ the proof of \eqref{eq:fisher_decay_on_V_1} is now complete.
\end{proof}

\subsection{Step 3: The Behaviour of the Generalised Fisher Information on the ``Spectrally Far'' Part of the Decomposition.}\label{subsec:behaviour_on_V_1_perp} In this step we will prove the following:
\begin{theorem}\label{thm:fisher_decay_on_V_1_perp}
	Consider the Fokker-Planck equation \eqref{eq:fokkerplanck} with a positive definite diffusion matrix, $\D$, and a positive stable matrix, $\C$, such that $\D=\C_s$.  Let $f(t)$ be a solution to this equation with a unit mass initial datum $f_0\in L^1_{+}\pa{\R^d}$ that has a finite $p-$entropy for some $1<p\leq 2$. Let $\PP$ be a positive definite matrix such that
	\begin{equation}\nonumber
		\C^T {\bf{P}}+{\bf{P}}\C \geq 2\lambda {\bf P},
	\end{equation}
	for some $\lambda>0$. Then for any $1\leq \widetilde{p} \leq 2$ and any $\epsilon>0$ there exists an explicit time $\tau_2(p,\epsilon)\geq 0$, that may become unbounded as $\epsilon$ goes to zero, such that for any $t\geq \tau_2(p,\epsilon)$ we have that $\I_{\widetilde{p}}^{\PP}\pa{f(t),f_2(t)|f_\infty}<\infty$ and
\begin{equation}\label{eq:Fisher_improved_decay_with_e_p_1}
\begin{gathered}
\I^{\bf{P}}_{\widetilde{p}} \pa{f(t),f_2(t)|f_\infty} \leq c_{p,\widetilde{p},\epsilon}\mathrm{p}_{\mathrm{max}}\\
\times \rpa{p(p-1)e_p\pa{f_0|f_\infty}+1}^{\frac{2}{p}\pa{1+\epsilon(2-\widetilde{p})}}e^{-2\pa{\lambda+\mathrm{d}_{\mathrm{min}}(\widetilde{p}-1)}(t-\tau_2)}
 \end{gathered}
\end{equation}
where $c_{p,\widetilde{p},\epsilon}>0$ is a constant that depends only on $p$, $\widetilde{p}$, $\epsilon$, and $\C$. $\mathrm{p}_{\mathrm{max}}$ and $\mathrm{d}_{\mathrm{min}}$ are the largest and smallest eigenvalues of $\PP$ and $\D$ respectively. \\
In particular, choosing $\widetilde{p}=p$ when $1<p\leq 2$, we find that 
\begin{equation}\label{eq:Fisher_improved_decay_with_e_p}
	\begin{gathered}
		\I^{\bf{P}}_{p} \pa{f(t),f_2(t)|f_\infty} \leq  \widetilde{c}_{p,\epsilon}  \mathrm{p}_{\mathrm{max}}\\
		\times \rpa{p(p-1)e_p\pa{f_0|f_\infty}+1}^{\frac{2}{p}\pa{1+\epsilon(2-p)}}e^{-2\pa{\lambda+\mathrm{d}_{\mathrm{min}}(p-1)}(t-\tau_2)}.
	\end{gathered}
\end{equation}
\end{theorem}

Before we start with the proof of the above, we state the following useful improvement of Lemma \ref{lem:hypercontractivity}:

\begin{lemma}\label{lem:hypercontractivity_improved}
	Let $g(t)$ be the solution to the Fokker-Planck equation \eqref{eq:fokkerplanck} with diffusion and drift matrices $\D$ and $\C$ which satisfy Conditions \eqref{item:cond_semipositive}--\eqref{item:cond_no_invariant_subspace_to_kernel} and with initial datum $g_0\in L^2\pa{\R^d,f_\infty^{-1}}$. Then, there exists an explicit time, $\tau_3$, which is independent of $g_0$, such that for all $t\geq \tau_3$ we have that the solution to the equation satisfies
	\begin{equation}\label{eq:hypercontractivity_improved}
		I_{2}^{\mathbf{I}}(g(t)|f_\infty) \leq 2^{\frac{3d}{2}+1}3^{\frac{d}{2}}e_2\pa{g_0|f_\infty}.
	\end{equation}
\end{lemma}


\begin{proof}
	We start by noticing that since 
	$$\nabla\pa{\frac{\amit{g\pa{t,x}}-f_\infty(x)}{f_\infty(x)}}=\nabla\pa{\frac{\amit{g\pa{t,x}}}{f_\infty(x)}}$$
	we have that for any $t$ for which the $2-$Fisher information is defined
	$$I^{\PP}_2\pa{g(t)|f_\infty}=I^{\PP}_2\pa{g(t)-f_\infty|f_\infty}.$$
	Using \eqref{eq:hypercontractivity} with $\epsilon=\frac{1}{6}$ together with the inequality
		\begin{equation}\nonumber
		\begin{gathered}
			\int_{\R^d}e^{\frac{\abs{x}^2}{6}}\abs{f(x)}dx \leq 
			3^{\frac{d}{4}}\norm{f}_{L^2\pa{\R^d,f_\infty^{-1}}},
		\end{gathered}
	\end{equation}
	(see \eqref{eq:estimate_of_eps_exponential_moment} in Appendix \ref{app:b})
	 we find that for $\tau_3:=\widetilde{t}_1\pa{\frac{1}{6}}>0$ we have that for any $t\geq \tau_3$
	$$	I_{2}^{\mathbf{I}}(g(t)|f_\infty) \leq 2^{\frac{3d}{2}}3^{\frac{d}{2}}\norm{g_0-f_\infty}_{L^2\pa{\R^d,f^{-1}_\infty}}^2=2^{\frac{3d}{2}+1}3^{\frac{d}{2}}e_2\pa{g_0|f_\infty},$$
	which is the desired result.
\end{proof}

We are now ready to prove Theorem \ref{thm:fisher_decay_on_V_1_perp}.

\begin{proof}[Proof of Theorem \ref{thm:fisher_decay_on_V_1_perp}]
	Using part \eqref{item:hyper} of Theorem \ref{thm:contractivity} we can find $t_1(p)>0$ after which we are assured to have that $f(t)\in L^2\pa{\R^d,f_\infty^{-1}}$ and \eqref{eq:hypercontractivity_entropy_with_entropy_condition} holds for $f(t)$ with $p_2=p$. Since $f_1(t)\in L^2\pa{\R^d,f_\infty^{-1}}$ for any $t\geq 0$, as was shown in the proof of Lemma \ref{lem:LonV1}, we conclude that for all $t\geq t_1(p)$ we have that $f_2(t)\in L^2\pa{\R^d,f_\infty^{-1}}$. Since $f_2(t)\in V_1^{\perp}$ for all such \amit{$t$}, Theorem \ref{thm:Fisher_improved_decay} implies that for any $\tau \geq t_1(p)$ we have that if $t\geq \tau$ then
	\begin{equation}\nonumber
		\begin{gathered}
			I^{\bf{P}}_{2}\pa{f_2(t)|f_\infty} 
			\leq I^{\PP}_{2}\pa{f_2\pa{\tau}|f_\infty}e^{-2\pa{\lambda+\mathrm{d}_{\mathrm{min}}} (t-\tau)}.
		\end{gathered}
	\end{equation}
Moreover, since $f_1\in V_1$ and $f_2-f_\infty\in V_1^{\perp}$ we find that
\begin{equation}\label{eq:e_2_for_f_and_f_2}
\begin{gathered}
	e_2\pa{f_2(t)|f_\infty}=\frac{1}{2}\norm{f_2(t)-f_\infty}_{L^2\pa{\R^d,f_\infty^{-1}}}^2  \\
	\leq\frac{1}{2}\norm{f_1(t)+f_2(t)-f_\infty}_{L^2\pa{\R^d,f_\infty^{-1}}}^2 =e_2\pa{f(t)|f_\infty}.
\end{gathered}
\end{equation}
Combining the above with a time shift of Lemma \ref{lem:hypercontractivity_improved} we see that for any $t \geq \widetilde{\tau}\pa{p}:=\tau_3+t_1(p)$ we have that
\begin{equation}\label{eq:improved_decay_perp_v_2_part_I}
	\begin{gathered}
		I^{\bf{P}}_{2}\pa{f_2(t)|f_\infty} \leq I^{\bf{P}}_{2}\pa{f_2\pa{\widetilde{\tau}}|f_\infty}e^{-2\pa{\lambda+\mathrm{d}_{\mathrm{min}}} (t-\widetilde{\tau})}\\	
		\leq 2^{\frac{3d}{2}+1}3^{\frac{d}{2}}\mathrm{p}_{\mathrm{max}}e_2\pa{f(t_1)|f_\infty}e^{-2\pa{\lambda+\mathrm{d}_{\mathrm{min}}} (t-\widetilde{\tau})} \\
		\leq \mathrm{p}_{\mathrm{max}}\widetilde{c}_{1,p}\rpa{p\pa{p-1}e_p\pa{f_0|f_\infty}+1}^{\frac{2}{p}}e^{-2\pa{\lambda+\mathrm{d}_{\mathrm{min}}} (t-\widetilde{\tau})}
	\end{gathered}
\end{equation}
where we have used \eqref{eq:hypercontractivity_entropy_with_entropy_condition} and the fact that $f_0$ is non-negative. Defining
$$\tau_2(p,\epsilon):=\max\pa{\widetilde{\tau}\pa{p}, t_1\pa{p,2,2\epsilon}},$$
where $t_1\pa{p_1,p_2,\epsilon}$ is given in part \eqref{item:hypo} of Theorem \ref{thm:contractivity}, we see that according to Theorem \ref{thm:Fisher_decay} with $\psi=\psi_1$ and a time shifted variant of part \ref{item:hypo} of Theorem \ref{thm:contractivity}  with $p=1$, $f(t)=f(t)$, $p_1=p$, $g(t)=f_2\pa{t}$ and $p_2=2$, similar to Remark \ref{rem:time_shift}, we have that\footnote{Note that since $t\geq\widetilde{\tau}(p)\geq  t_1(p)$ we have that $f_2(t)\in L^2\pa{\R^d,f_\infty^{-1}}$ and, as can be seen in the proof of Lemma \eqref{lem:hyper_conditions} in Appendix \ref{app:b}, this implies that $f\in L^1\pa{\R^d}$ which is needed for Theorem \ref{thm:Fisher_decay}.}
\begin{equation}\nonumber
	\begin{gathered}
		\mathcal{I}^{\bf{P}}_1 \pa{f(t),f_2(t)|f_\infty} 
		\leq \mathrm{p}_{\mathrm{max}}\mathcal{I}^{\II}_{1}\pa{f(\tau_2),f_2(\tau_2)|f_\infty}e^{-2\lambda (t-\tau_2)} \\
		\leq \widehat{c}_{2,p,\epsilon}\mathrm{p}_{\mathrm{max}}\rpa{p(p-1)e_p\pa{f(\tau_2)|f_\infty}+1}^{\frac{2\epsilon}{p}}\pa{2e_2\pa{\abs{f_2(\tau_2)}|f_\infty}+1} e^{-2\lambda (t-\tau_2)}\\
	\leq \widehat{c}_{2,p,\epsilon}\mathrm{p}_{\mathrm{max}}\rpa{p(p-1)e_p\pa{f_0|f_\infty}+1}^{\frac{2\epsilon}{p}}\pa{2e_2\pa{f_2(\tau_2)|f_\infty}+1} e^{-2\lambda (t-\tau_2)}
		\\
		\leq \widehat{c}_{2,p,\epsilon}\mathrm{p}_{\mathrm{max}}\rpa{p(p-1)e_p\pa{f_0|f_\infty}+1}^{\frac{2\epsilon}{p}}\pa{2e_2\pa{f(\tau_2)|f_\infty}+1} e^{-2\lambda (t-\tau_2)}.
	\end{gathered}
\end{equation}
Here we have used the fact that the $p-$entropy of $f(t)$ is decreasing, inequality \eqref{eq:e_2_for_f_and_f_2}, and the fact that for any $g\in L^1\pa{\R^d,f_\infty^{-1}}$ we have that 
$$e_2\pa{\abs{g}|f_\infty} = \frac{1}{2}\pa{\norm{g}^2_{L^2\pa{\R^d,f_\infty^{-1}}}-2\norm{g}_{L^1\pa{\R^d}}+1}$$
$$\leq \frac{1}{2}\pa{\norm{g}^2_{L^2\pa{\R^d,f_\infty^{-1}}}-2\int_{\R^d}g(x)dx+1}=e_2\pa{g|f_\infty}.$$
Using \eqref{eq:hypercontractivity_entropy_with_entropy_condition} again with $f$ instead of $g$ and $p_2=p$ we conclude that for any $t\geq \tau_2\pa{p,\epsilon}$
\begin{equation}\label{eq:improved_decay_perp_v_2_part_II}
	\begin{gathered}
		\mathcal{I}^{\bf{P}}_1 \pa{f(t),f_2(t)|f_\infty} 
		\leq \widetilde{c}_{2,p,\epsilon}\mathrm{p}_{\mathrm{max}}\rpa{p(p-1)e_p\pa{f_0|f_\infty}+1}^{\frac{2}{p}\pa{1+\epsilon}} e^{-2\lambda (t-\tau_2)}.
	\end{gathered}
\end{equation}
Interpolating  \eqref{eq:improved_decay_perp_v_2_part_I}, with  $\widetilde{\tau}$ replaced with the larger $\tau_2$, and \eqref{eq:improved_decay_perp_v_2_part_II} via \eqref{eq:interpolation_p} gives the desired result. The proof of \eqref{eq:Fisher_improved_decay_with_e_p_1} is now complete.
\end{proof}

We are now ready to prove the main theorem of this section. 
\subsection{Step 4: The Proof of Theorem \ref{thm:main_for_fisher}}\label{subsec:main_for_fisher_proof} 
\begin{proof}[Proof of Theorem \ref{thm:main_for_fisher}]
	Given $\epsilon>0$, let $\tau_0\pa{p,\epsilon}:=\max\pa{\tau_1\pa{p,\epsilon},\tau_2\pa{p,\epsilon}}$ where $\tau_1\pa{p,\epsilon}$ and $\tau_2\pa{p,\epsilon}$  are given in Theorem  \ref{thm:fisher_decay_on_V_1} and Theorem \ref{thm:fisher_decay_on_V_1_perp} respectively. Consider the decomposition $f(t,x)=f_1(t,x)+f_2(t,x)$ described in Step 1. Since for any positive definite matrix $\PP$ and any two vectors $\bm{u}$ and $\bm{v}$ we have that
	$$\pa{\bm{u}+\bm{v}}^T \PP \pa{\bm{u}+\bm{v}} \leq 2\pa{\bm{u}^T \PP \bm{u}+\bm{v}^T\PP\bm{v}}$$
	we find that for any $p\in [1,2]$
	\begin{equation}\label{eq:fisher_in_terms_of_f_1_and_f_2}
		\begin{gathered}
			I_p^{\PP}\pa{f(t)|f_\infty}=\I_p^{\PP}\pa{f(t),f_1(t)+f_2(t)|f_\infty} \\
			\leq 2\pa{\I_p^{\PP}\pa{f(t),f_1(t)|f_\infty}+\I_p^{\PP}\pa{f(t),f_2(t)|f_\infty}}.
		\end{gathered}
	\end{equation}
	Inequalities \eqref{eq:fisher_decay_on_V_1_with_I_p}  and  \eqref{eq:Fisher_improved_decay_with_e_p} imply that for any $t\geq \tau_0\pa{p,\epsilon}$
	$$I_p^{\PP}\pa{f(t)|f_\infty} \leq \widetilde{c}_{1,p,\epsilon} \mathrm{p}_{\mathrm{max}}\rpa{p(p-1)e_p\pa{f_0|f_\infty}+1}^{\frac{2}{p}\pa{1+\epsilon(2-p)}}\pa{1+(t-\tau_0)^{2n}} e^{-2\mu (t-\tau_0)}$$
	$$+\widetilde{c}_{2,p,\epsilon} \mathrm{p}_{\mathrm{max}}\rpa{p(p-1)e_p\pa{f_0|f_\infty}+1}^{\frac{2}{p}\pa{1+\epsilon(2-p)}}e^{-2\pa{\mu-\nu+\mathrm{d}_{\mathrm{min}}(p-1)}(t-\tau_0)}$$
	$$\leq\max\pa{\widetilde{c}_{1,p,\epsilon},\widetilde{c}_{2,p,\epsilon}}\mathrm{p}_{\mathrm{max}}\rpa{p(p-1)e_p\pa{f_0|f_\infty}+1}^{\frac{2}{p}\pa{1+\epsilon(2-p)}}\pa{1+(t-\tau_0)^{2n}} e^{-2\mu (t-\tau_0)}$$
	$$\times \pa{1+\pa{1+(t-\tau_0)^{2n}}^{-1}e^{-2\pa{\mathrm{d}_{\mathrm{min}}(p-1)-\nu}(t-\tau_0)}}$$
	$$\leq 2\max\pa{\widetilde{c}_{1,p,\epsilon},\widetilde{c}_{2,p,\epsilon}}\mathrm{p}_{\mathrm{max}}\rpa{p(p-1)e_p\pa{f_0|f_\infty}+1}^{\frac{2}{p}\pa{1+\epsilon(2-p)}}\pa{1+(t-\tau_0)^{2n}} e^{-2\mu (t-\tau_0)}$$
	$$\leq \underbrace{2\max\pa{\widetilde{c}_{1,p,\epsilon},\widetilde{c}_{2,p,\epsilon}}e^{2\mu\tau_0}}_{:=c_{p,\epsilon}}\mathrm{p}_{\mathrm{max}}\rpa{p(p-1)e_p\pa{f_0|f_\infty}+1}^{\frac{2}{p}\pa{1+\epsilon(2-p)}}\pa{1+t^{2n}} e^{-2\mu t},$$
	where the shift from $\tau_1$ and $\tau_2$ to $\tau_0$ was absorbed in the constants $\widetilde{c}_{1,p,\epsilon}$ and $\widetilde{c}_{2,p,\epsilon}$. This concludes the proof of the first statement of the theorem.
	
	When $f_0$ has a finite $2-$entropy, or equivalently $f_0\in L^2\pa{\R^d,f_\infty^{-1}}$, we know that $f(t)\in L^2\pa{\R^d,f_\infty^{-1}}$ for all $t\geq 0$ and consequently so does $f_1(t)$ and $f_2(t)$. Using Theorem \ref{thm:Fisher_improved_decay} and Lemma  \ref{lem:hypercontractivity_improved} we can find $\tau_3>0$ which is independent of $f_0$ such that for any $t\geq \tau_3$ 
	\begin{equation}\label{eq:improved_decay_proof_I}
		\begin{gathered}
				I_{2}^{\PP}(f_2(t)|f_\infty) \leq I^{\bf{P}}_{2}\pa{f_2(\tau_3)|f_\infty}e^{-2\pa{\pa{\mu-\nu}+\mathrm{d}_{\mathrm{min}}} (t-\tau_3)} \\
				 \leq 2^{\frac{3d}{2}+1}3^{\frac{d}{2}}\mathrm{p}_{\mathrm{max}}e_2\pa{f_2(0)|f_\infty}e^{-2\pa{\pa{\mu-\nu}+\mathrm{d}_{\mathrm{min}}} (t-\tau_3)}\\ \leq \mathrm{p}_{\mathrm{max}}2^{\frac{3d}{2}+1}3^{\frac{d}{2}}e^{2\pa{\pa{\mu-\nu}+\mathrm{d}_{\mathrm{min}}}\tau_3} e_2\pa{f_0|f_\infty}e^{-2\pa{\pa{\mu-\nu}+\mathrm{d}_{\mathrm{min}}} t},
		\end{gathered}
	\end{equation}
	where we have used \eqref{eq:e_2_for_f_and_f_2} again. To deal with $f_1$ we notice that from Lemma \ref{lem:LonV1} with $\overline{t_0}=0$ we have that for any $t\geq 0$
	$$		I_2^{\PP}\pa{f_1(t)|f_\infty} \leq \mathrm{p}_{\mathrm{max}}\int_{\R^d}\abs{\nabla\pa{\frac{f_1\pa{t,x}}{f_\infty(x)}}}^2 f_\infty(x)dx$$
$$\leq \widetilde{c}_2 \mathrm{p}_{\mathrm{max}} \norm{f_1(0)}^2_{L^2\pa{\R^d,f_\infty^{-1}}} \pa{1+t^{2n}} e^{-2\mu t} $$
$$\leq \widetilde{c}_2 \mathrm{p}_{\mathrm{max}} \norm{f_0-f_\infty}^2_{L^2\pa{\R^d,f_\infty^{-1}}} \pa{1+t^{2n}} e^{-2\mu t} $$
	with $\widetilde{c}_2$ being a fixed constant. Here we have used the fact that since $f_1(t) \perp f_2(t)-f_\infty$ for any $t$ we have that
	$$\norm{f_1(t)}^2_{L^2\pa{\R^d,f_\infty^{-1}}} \leq \norm{f_1(t)}^2_{L^2\pa{\R^d,f_\infty^{-1}}}+\norm{f_2(t)-f_\infty}^2_{L^2\pa{\R^d,f_\infty^{-1}}}=\norm{f(t)-f_\infty}^2_{L^2\pa{\R^d,f_\infty^{-1}}}.$$
	From the definition of the $2-$entropy we conclude that for any $t\geq 0$
	 	\begin{equation}\label{eq:improved_decay_proof_II}
	 	\begin{gathered}
	 		I_2^{\PP}\pa{f_1(t)|f_\infty} \leq 
	 		c_2 \mathrm{p}_{\mathrm{max}} e_2\pa{f_0|f_\infty}\pa{1+t^{2n}} e^{-2\mu t}.
	 	\end{gathered}
	 \end{equation}
	Combining \eqref{eq:improved_decay_proof_I} and the fact that $\nu < \mathrm{d}_{\mathrm{min}}(p-1)\leq \mathrm{d}_{\mathrm{min}}$, together with \eqref{eq:improved_decay_proof_II} and \eqref{eq:fisher_in_terms_of_f_1_and_f_2} for $p=2$ we see that for any given $t\geq \tau_3$ we have that
	$$I_2^{\PP}\pa{f(t)|f_\infty} \leq c \;\mathrm{p}_{\mathrm{max}}e_2\pa{f_0|f_\infty}\pa{1+t^{2n}} e^{-2\mu t}.$$
The above is exactly \eqref{eq:main_for_fisher_improved} which concludes the proof.
\end{proof}


\section{The Convergence to Equilibrium in Entropy}\label{sec:convergence_entropy}
With the establishment of the convergence to equilibrium in Fisher Information in the previous section, the proof of our main result, Theorem \ref{thm:main}, is now easily within our grasp and will be the only content of this short section.
\begin{proof}[Proof of Theorem \ref{thm:main}]
	We start by finding a positive definite matrix $\PP$ which satisfies \eqref{eq:CPcondition} for $\lambda=\mu-\nu$ with $0\leq \nu < \mathrm{d}_{\mathrm{min}}(p-1)$. The existence of such matrix is guaranteed by the work of Arnold and Erb (see \cite{AE}) and the fact that $p>1$.\\ 
	Using \eqref{eq:main_for_fisher} from Theorem \ref{thm:main_for_fisher} we see that for any given $1<p \leq 2$ and any $\epsilon>0$, there exists an explicit $\tau_0(p,\epsilon)\geq 0$ such that for any $t\geq \tau_0(p,\epsilon)$
	\begin{equation}\nonumber
		\begin{gathered}
			I_p^{\PP}\pa{f(t)|f_\infty} \leq c_{p,\epsilon} \mathrm{p}_{\mathrm{max}}\rpa{p(p-1)e_p\pa{f_0|f_\infty}+1}^{\frac{2}{p}\pa{1+\epsilon(2-p)}}
			\pa{1+t^{2n}} e^{-2\mu t}.
		\end{gathered}
	\end{equation}
Since Fokker-Planck equations of the form \eqref{eq:fokkerplanck} carry a log-Sobolev inequality of the form
$$e_p\pa{f|f_\infty} \leq C_{LS}I_p^{{\bf{I}}}\pa{f|f_\infty},$$
where $C_{LS}>0$ is an explicit constant that is independent of $f$ and $p$ (see, for instance, \cite{AE}), we conclude that for all $t\geq \tau_0(p,\epsilon)$
	\begin{equation}\nonumber
	\begin{gathered}
		e_p\pa{f(t)|f_\infty} \leq \frac{C_{LS}}{\mathrm{p}_{\mathrm{min}}}I_p^{\PP}\pa{f(t)|f_\infty} \\
		\leq C_{LS}c_{p,\epsilon} \frac{\mathrm{p}_{\mathrm{max}}}{\mathrm{p}_{\mathrm{min}}}\rpa{p(p-1)e_p\pa{f_0|f_\infty}+1}^{\frac{2}{p}\pa{1+\epsilon(2-p)}}
		\pa{1+t^{2n}} e^{-2\mu t}.
	\end{gathered}
\end{equation}
Since the entropy is decreasing for all $t$ we conclude that for all $0\leq t\leq \tau_0(p,\epsilon)$
$$e_p\pa{f(t)|f_\infty} \leq e_p\pa{f_0|f_\infty}\leq e^{2\mu \tau_0}e_p\pa{f_0|f_\infty}\pa{1+t^{2n}}e^{-2\mu t}.$$
Combining the results for $t\leq \tau_0(p,\epsilon)$ and $t\geq \tau_0(p,\epsilon)$, and remembering that $\tau_0(p,\epsilon)$ is given explicitly and independently of $f_0$ and may become unbounded as $\epsilon$ goes to zero, we conclude the main proof of the theorem for this case. The case $p=2$ follows in a similar way from \eqref{eq:main_for_fisher_improved}.
\end{proof} 

We conclude our work with a few final remarks. 

\section{Final Remarks}\label{sec:finalremarks}
While the main result of our paper has been established before, we find that the main novelty and importance of the presented work is the recognition of the generalised Fisher information as an important functional to investigate. We find it remarkable
that this functional, which distinguishes and takes advantage of the two different parts of the standard Fisher information, carries a similar long term behaviour of its standard namesake.

It is yet unclear to us if the generalised Fisher information can be used in other equations with an entropic structure, but we feel that our paper promotes the benefits of extending existing investigations of functionals to exploit the geometric properties of the spaces that lie in the heart of the spectral study of the governing linear operator.

\amit{Lastly, we would like to mention that as is mentioned in Remark \ref{rem:only_place}, the recent work \cite{AEW23} has shown an improved version of Theorem \ref{thm:Fisher_improved_decay} which allows for degeneracy in $\D$. Using this result instead of Theorem \ref{thm:Fisher_improved_decay} shows that our main theorem (and its stronger Fisher information variant) holds true as long as $\D$ and $\C$ satisfy conditions \eqref{item:cond_semipositive}-\eqref{item:cond_no_invariant_subspace_to_kernel}.}

\appendix
\section{Generalised Fisher Information for Fokker-Planck Equations with Non-Quadratic Potentials}\label{app:a}

In this appendix we shall briefly consider the decay properties of the generalised Fisher information in more general Fokker-Planck equations. Our analysis is based on the cases considered in \cite{AMTU01}. In these cases the diffusion and drift matrices are allowed to depend on the spatial variable, though the diffusion matrix $\D(x)$ is assumed to be non-degenerate (much like our own work here). The drift term that is considered is one that arises from a potential.

In this appendix we will consider the following type of Fokker-Planck equations:
\begin{equation}\label{FP2}
  \partial_t f(t,x)=\text{div}\pa{\D(x)\pa{\nabla f(t,x)+f(t,x)\nabla\phi(x)}}, \quad\quad t>0, x\in\R^d\,,
\end{equation}
where the confinement potential $\phi(x)$ is such that $e^{-\phi(x)} \in L^1\pa{\R^d}$ and the matrix $\D=\D(x)$ is symmetric and locally uniformly positive definite. Under these assumptions, a unique normalised equilibrium of \eqref{FP2} exists and is given by 
$$f_\infty(x) = c_{\infty}e^{-\phi(x)},$$
for an appropriate constant $c_\infty$. To simplify notations, we may assume without loss of generality that $c_\infty=1$. The Fisher information of an admissible relative entropy with generating function $\psi$ of a unit mass function $f$ with respect to the equilibrium $f_\infty$ is defined as 
$$ I_\psi(f|f_\infty):=\int_{\R^d}\psi^{\prime\prime}\pa{\frac{f(x)}{f_\infty(x)}}u(x)^T\D(x) u(x)f_\infty(x)dx,$$
where $u(x):=\nabla\pa{\frac{f(x)}{f_\infty(x)}}$. Much like our presented work, we can extend this definition to the generalised Fisher information 
$$ \I_\psi(f,g|f_\infty):=\int_{\R^d}\psi^{\prime\prime}\pa{\frac{f(x)}{f_\infty(x)}}v(x)^T\D(x) v(x) f_\infty(x)dx,$$
where $v(x):=\nabla\pa{\frac{g(x)}{f_\infty(x)}}$. \\
The following theorem explores the decay rate of this variant of the generalised Fisher information in the case appearing in Lemma 2.13 of \cite{AMTU01}:
\begin{theorem}\label{thm:Fisher-decay2}
Let $f$ and $g$ be solutions to \eqref{FP2} with initial
data $f_0\in L^2_+\pa{\R^d,f_\infty^{-1}}$ and $g_0\in L^2\pa{\R^d,f_\infty^{-1}}$. Assume in addition that $f_0$ is of unit mass, $I_\psi\pa{f_0,g_0| f_\infty}<\infty$, and that
$$\D(x)=D(x)\II$$
where $D(x)$ is a positive function. If there exists $\lambda_1>0$ such that for any $x\in\R^d$
\begin{equation}\label{A1}
\begin{gathered}
\pa{\frac{1}{2}-\frac{d}{4}}\frac{1}{D(x)}\nabla D(x)\otimes\nabla D(x)+\frac{1}{2}\pa{\Delta D(x)-\nabla D(x)\cdot\nabla \phi(x)}\II \\
+D(x)\hess(\phi)(x)+\frac{\nabla \phi(x)\otimes\nabla D(x) + \nabla D(x)\otimes\nabla \phi(x)}{2}-\hess(D)(x) \geq \lambda_1\II,
\end{gathered}
\end{equation}
then for all $t\geq 0$ we have that
\begin{equation} \label{decay-est}
  \I_\psi\pa{f(t),g(t)|f_\infty}\leq \I_\psi\pa{f_0,g_0|f_\infty}e^{-2\lambda_1t}. 
\end{equation}
\end{theorem}
\begin{remark}
	At this point it is worth to mention that the result on which the above is based was extended in \cite{ArCaJuL08} to more general cases of diffusion matrices $\D(x)$. Due to the similarity of the proof we will shortly provide to that appearing in \cite{AMTU01} we believe that using the same technical computation presented in \cite{ArCaJuL08} we would be able to extend the above, but we have elected to leave that discussion to a different study.
\end{remark}
\begin{proof}
It is simple to see that in this setting equation \eqref{FP2} can be rewritten as
\begin{equation}\label{eq:fprewrite}
	\partial_t f(t,x)=\text{div}\pa{f_\infty(x)D(x)\nabla \pa{\frac{f(t,x)}{f_\infty(x)}}}.
	\end{equation}
Differentiating\footnote{Here we are assuming that we are indeed allowed to differentiate under the sign of the integral, which is a standard procedure in utilising the entropy-method.} our generalised Fisher information yields the expression
\begin{equation}\label{eqapp:0}
\begin{gathered}
	\frac{d}{dt}\I_\psi\pa{f(t),g(t)|f_\infty} = \int_{\R^d}\psi^{\prime\prime\prime}\pa{\frac{f(t,x)}{f_\infty(x)}}\partial_t f(t,x) D(x) \abs{v(t,x)}^2 dx \\
	+2\int_{\R^d}\psi^{\prime\prime}\pa{\frac{f(t,x)}{f_\infty(x)}}\partial_t v(t,x)\cdot  v(t,x)D(x)f_\infty(x)dx.
\end{gathered}	
\end{equation}
Using \eqref{eq:fprewrite} and integration by parts shows that
\begin{equation}\label{eqapp:I}
	\begin{gathered}
		\int_{\R^d}\psi^{\prime\prime\prime}\pa{\frac{f(t,x)}{f_\infty(x)}}\partial_t f(t,x) D(x) \abs{v(t,x)}^2 dx\\
		=-\int_{\R^d}\psi^{\prime\prime\prime\prime}\pa{\frac{f(t,x)}{f_\infty(x)}}D(x)^2\abs{u(t,x)}^2\abs{v(t,x)}^2 f_\infty(x)dx\\
		-\int_{\R^d}\psi^{\prime\prime\prime}\pa{\frac{f(t,x)}{f_\infty(x)}}D(x) u(t,x)\cdot \nabla D(x) \abs{v(t,x)}^2 f_\infty(x)dx\\
		-2\int_{\R^d}\psi^{\prime\prime\prime}\pa{\frac{f(t,x)}{f_\infty(x)}}D(x)^2 v(t,x)^T J_x \pa{v(t,x)} u(t,x) f_\infty(x)dx,
	\end{gathered}
\end{equation}
where we have used the notation of $J_x\pa{v(x)}$ from the proof of Theorem \ref{thm:Fisher_decay}.\\
\amit{Since} 
$$\partial_t\pa{v(t,x)}=\nabla\pa{\frac{\partial_tg(t,x)}{f_\infty(x)}} =\nabla \pa{f_\infty^{-1}(x)\text{div}\pa{f_\infty(x)D(x)v(x)}}$$
we see that 
\begin{equation}\label{eqapp:II}
	\begin{gathered}
	2\int_{\R^d}\psi^{\prime\prime}\pa{\frac{f(t,x)}{f_\infty(x)}}\partial_t v(t,x)\cdot  v(t,x)D(x)f_\infty(x)dx\\
	=2\int_{\R^d}\psi^{\prime\prime}\pa{\frac{f(t,x)}{f_\infty(x)}}\nabla \big[-D(x)v(t,x)\cdot \nabla \phi(x)\\
	\qquad \qquad \qquad \qquad \qquad \qquad +\nabla D(x)\cdot v(t,x)
		+D(x)\text{div}\pa{v(t,x)}\big] v(t,x)D(x)f_\infty(x)dx\\
	=2\int_{\R^d}\psi^{\prime\prime}\pa{\frac{f(t,x)}{f_\infty(x)}}D(x)f_\infty(x)v(t,x)\cdot\nabla \pa{D(x)\text{div}\pa{v(t,x)}}dx\\
	-2\int_{\R^d}\psi^{\prime\prime}\pa{\frac{f(t,x)}{f_\infty(x)}}D(x)f_\infty(x)v(t,x)^T\pa{J_x\pa{D(x)\nabla \phi(x)-\nabla D(x)}}v(t,x)dx\\
	-2\int_{\R^d}\psi^{\prime\prime}\pa{\frac{f(t,x)}{f_\infty(x)}}D(x)f_\infty(x)v(t,x)^T J_x\pa{v(t,x)}\pa{D(x)\nabla \phi(x)-\nabla D(x)}dx.
	\end{gathered}
\end{equation}

Next, we notice that a simple \amit{calculation} shows that
$$\frac{1}{2}\Delta\pa{D(x)\abs{v(x)}^2}=\frac{1}{2}\Delta D(x) \abs{v(x)}^2 +2 \nabla D(x)\cdot  J_x\pa{v(x)}v(x)$$
$$+D(x)\sum_{i,j=1}^d \pa{\partial_i v_j(x)}^2+D(x)\sum_{i,j=1}^d v_j(x)\partial^2_{ii}v_j(x).$$
Recalling that $\partial_i v_j(x)=\partial_j v_i(x)$ we find that 
$$D(x)\sum_{i,j=1}^d v_j(x)\partial^2_{ii}v_j(x)=D(x)\sum_{i,j=1}^d v_j(x)\partial^2_{ij}v_i(x)$$
$$=D(x)\sum_{j=1}^d v_j(x)\partial_j\sum_{i=1}^d\partial_{i}v_i(x)=D(x)v(x)\cdot \nabla\text{div} \pa{v(x)},$$
and as such
$$D(x)v(x)\cdot \nabla\text{div} \pa{v(x)}=\frac{1}{2}\Delta\pa{D(x)\abs{v(x)}^2}-\frac{1}{2}\Delta D(x) \abs{v(x)}^2$$
$$ -2 \nabla D(x) \cdot J_x\pa{v(x)}v(x)-D(x)\sum_{i,j=1}^d \pa{\partial_i v_j(x)}^2.$$
This implies that 
\begin{equation}\label{eqapp:III}
	\begin{gathered}
		v(x)\cdot \nabla \pa{D(x)\text{div}\pa{v(x)}}=\text{div}\pa{v(x) }\pa{v(x) \cdot \nabla D(x)}+D(x) v(x) \cdot \nabla \text{div}\pa{v(x)}\\
		=\text{div}\pa{v(x)} \pa{v(x) \cdot \nabla D(x)}+\frac{1}{2}\Delta\pa{D(x)\abs{v(x)}^2}-\frac{1}{2}\Delta D(x) \abs{v(x)}^2\\
		-2 \nabla D(x) \cdot J_x\pa{v(x)}v(x)-D(x)\sum_{i,j=1}^d \pa{\partial_i v_j(x)}^2.
	\end{gathered}
\end{equation}
On the other hand
$$\text{div}\pa{D(x)\nabla \pa{D(x)\abs{v(x)}^2}f_\infty(x)}=D(x)f_\infty(x) \Delta\pa{D(x)\abs{v(x)}^2}$$
$$+f_\infty(x)\pa{\nabla D(x)-D(x)\nabla \phi(x)}\cdot \nabla\pa{D(x)\abs{v(x)}^2}$$
$$=D(x)f_\infty(x) \Delta\pa{D(x)\abs{v(x)}^2}+f_\infty(x)\pa{\nabla D(x)-D(x)\nabla \phi(x)}\cdot \nabla D(x) \abs{v(x)}^2$$
$$+2D(x)f_\infty(x)\pa{\nabla D(x)-D(x)\nabla \phi(x)}^T J_x\pa{v(x)}v(x)$$
so that
$$D(x)f_\infty(x)\big[\Delta\pa{D(x)\abs{v(x)}^2}+2\pa{\nabla D(x)-D(x)\nabla \phi(x)} J_x\pa{v(x)}v(x)\big]$$
$$=\text{div}\pa{D(x)\nabla \pa{D(x)\abs{v(x)}^2}f_\infty(x)}-f_\infty(x)\pa{\nabla D(x)-D(x)\nabla \phi(x)}\cdot \nabla D(x) \abs{v(x)}^2.$$
Using the above with \eqref{eqapp:III} we see that
\begin{equation}\label{eqapp:IV}
	\begin{gathered}
		2D(x)f_\infty(x)\pa{v(x)\cdot \nabla \pa{D(x)\text{div}\pa{v(x)}}+\pa{\nabla D(x)-D(x)\nabla \phi(x)}^T J_x\pa{v(x)}v(x)}\\
		=D(x)f_\infty(x)\pa{\Delta\pa{D(x)\abs{v(x)}^2}+2\pa{\nabla D(x)-D(x)\nabla \phi(x)}^T J_x\pa{v(x)}v(x)}\\
		+2D(x)f_\infty(x)\text{div}\pa{v(x)} \pa{v(x) \cdot \nabla D(x)}-D(x)f_\infty(x)\Delta D(x) \abs{v(x)}^2\\
		-4D(x)f_\infty(x)\nabla D(x) \cdot\left( J_x\pa{v(x)}v(x)\right)-2D(x)^2 f_\infty(x)\sum_{i,j=1}^d \pa{\partial_i v_j(x)}^2\\
		=\text{div}\pa{D(x)\nabla \pa{D(x)\abs{v(x)}^2}f_\infty(x)}-f_\infty(x)\pa{\nabla D(x)-D(x)\nabla \phi(x)} \nabla D(x) \abs{v(x)}^2\\
		+2D(x)f_\infty(x)\text{div}\pa{v(x)} \pa{v(x) \cdot \nabla D(x)}-D(x)f_\infty(x)\Delta D(x) \abs{v(x)}^2\\
		-4D(x)f_\infty(x)\nabla D(x) \cdot (J_x\pa{v(x)}v(x))-2D(x)^2 f_\infty(x)\sum_{i,j=1}^d \pa{\partial_i v_j(x)}^2.
	\end{gathered}
\end{equation}
Combining  \eqref{eqapp:II} and \eqref{eqapp:IV} with 
$$\int_{\R^d}\psi^{\prime\prime}\pa{\frac{f(x)}{f_\infty(x)}}\text{div}\pa{D(x)\nabla \pa{D(x)\abs{v(x)}^2}f_\infty(x)}dx$$
$$=-\int_{\R^d}\psi^{\prime\prime\prime}\pa{\frac{f(x)}{f_\infty(x)}}u(x)\cdot \pa{D(x)\nabla \pa{D(x)\abs{v(x)}^2}f_\infty(x)}dx$$
$$=-\int_{\R^d}\psi^{\prime\prime\prime}\pa{\frac{f(x)}{f_\infty(x)}}\pa{D(x)^2 u(x)\cdot \nabla \pa{\abs{v(x)}^2}+D(x)\abs{v(x)}^2 u(x)\cdot \nabla D(x)}dx$$
and adding and subtracting the term $\frac{d-2}{4}\pa{v(x)\cdot \nabla D(x)}^2$ gives us 
\begin{equation}\label{eqapp:V}
	\begin{gathered}
		2\int_{\R^d}\psi^{\prime\prime}\pa{\frac{f(t,x)}{f_\infty(x)}}\partial_t v(t,x)\cdot  v(t,x)D(x)f_\infty(x)dx\\
		=-\int_{\R^d}\psi^{\prime\prime\prime}\pa{\frac{f(t,x)}{f_\infty(x)}}\pa{D(x)^2 u(t,x)\cdot \nabla \pa{\abs{v(t,x)}^2}+D(x)\abs{v(t,x)}^2 u(t,x)\cdot \nabla D(x)}dx\\
		-2\int_{\R^d}\psi^{\prime\prime}\pa{\frac{f(t,x)}{f_\infty(x)}}D(x)f_\infty(x)\Bigg[v(t,x)^T\pa{J_x\pa{D(x)\nabla \phi(x)-\nabla D(x)}}v(t,x)\\
		+\frac{1}{2}\Delta D(x) \abs{v(t,x)}^2-\frac{1}{2}\abs{v(t,x)}^2 \nabla D(x)\cdot \nabla \phi(x)-\frac{1}{D(x)}\frac{d-2}{4}\pa{v(t,x)\cdot \nabla D(x)}^2\Bigg] dx\\
		-2\int_{\R^d}\psi^{\prime\prime}\pa{\frac{f(t,x)}{f_\infty(x)}}f_\infty(x)\Bigg[\frac{d-2}{4}\pa{v(t,x)\cdot  \nabla D(x)}^2 -D(x)\pa{v(t,x)\cdot \nabla D(x)}\text{div}\pa{v(t,x)}\\
		+2D(x)\nabla D(x) (J_x\pa{v(t,x)}v(t,x))+D^2(x)\sum_{i,j=1}^d\pa{\partial_i v_j(t,x)}^2+\frac{1}{2}\abs{v(t,x)}^2 \abs{\nabla D (x)}^2\Bigg] dx.
	\end{gathered}
\end{equation}
\amit{Since}
$$\Bigg[\frac{d-2}{4}\pa{v(t,x)\cdot \nabla D(x)}^2 -D(x)\pa{v(t,x)\cdot \nabla D(x)}\text{div}\pa{v(t,x)}$$
$$+2D(x)\nabla D(x) J_x\pa{v(t,x)}v(t,x)+D^2(x)\sum_{i,j=1}^d\pa{\partial_i v_j(t,x)}^2+\frac{1}{2}\abs{v(t,x)}^2 \abs{\nabla D (x)}^2\Bigg]$$
$$=\sum_{i,j=1}^d\pa{D(x)\partial_i v_j(t,x)+\frac{1}{2}\partial_i D(x) v_j(t,x)+\frac{1}{2}v_i(t,x)\partial_j D(x)-\frac{1}{2}\delta_{i,j}\pa{\nabla D(x)\cdot v(t,x)}}^2$$
and
$$\Bigg[v(t,x)^T\pa{J_x\pa{D(x)\nabla \phi(x)-\nabla D(x)}}v(t,x)$$
$$+\frac{1}{2}\Delta D(x) \abs{v(t,x)}^2-\frac{1}{2}\abs{v(t,x)}^2 \nabla D(x)\cdot \nabla \phi(x)-\frac{1}{D(x)}\frac{d-2}{4}\pa{v(t,x)\cdot \nabla D(x)}^2\Bigg] $$
$$=v(t,x)^T \Bigg[ \frac{1}{D(x)}\frac{2-d}{4}\nabla D(x) \otimes \nabla D(x)+\frac{1}{2}\pa{\Delta D(x)-\nabla D(x)\cdot \nabla \phi(x)}$$
$$+D(x)\hess(\phi)(x)+\underbrace{\frac{\nabla \phi(x) \otimes \nabla D(x)+\nabla D(x) \otimes \nabla \phi(x)}{2}}_{\pa{\nabla D(x)\otimes \nabla \phi(x)}_{s}}-\hess(D)(x)\Bigg]v(t,x)$$
we can use condition \eqref{A1} to see that \eqref{eqapp:V} implies that
\begin{equation}\label{eqapp:VI}
	\begin{gathered}
		2\int_{\R^d}\psi^{\prime\prime}\pa{\frac{f(t,x)}{f_\infty(x)}}\partial_t v(t,x)\cdot  v(t,x)D(x)f_\infty(x)dx\\
		\leq -\int_{\R^d}\psi^{\prime\prime\prime}\pa{\frac{f(t,x)}{f_\infty(x)}}\pa{D(x)^2 u(t,x)\cdot \nabla \pa{\abs{v(t,x)}^2}+D(x)\abs{v(t,x)}^2 u(t,x)\cdot \nabla D(x)}dx\\
		-2\lambda_1\underbrace{\int_{\R^d}\psi^{\prime\prime}\pa{\frac{f(t,x)}{f_\infty(x)}}D(x)f_\infty(x)\abs{v(t,x)}^2 dx}_{ \I_{\psi}\pa{f(t),g(t)|f_\infty}}\\
		-2\sum_{i,j=1}^d\int_{\R^d}\psi^{\prime\prime}\pa{\frac{f(t,x)}{f_\infty(x)}}f_\infty(x)\Bigg(D(x)\partial_i v_j(t,x)+\frac{1}{2}\partial_i D(x) v_j(t,x)\\
		+\frac{1}{2}v_i(t,x)\partial_j D(x)-\frac{1}{2}\delta_{i,j}\pa{\nabla D(x)\cdot v(t,x)}\Bigg)^2 dx.
	\end{gathered}
\end{equation}
Plugging \eqref{eqapp:I} and \eqref{eqapp:VI} in \eqref{eqapp:0}, and using the fact that $\nabla \pa{\abs{v(t,x)}^2}=2J_x\pa{v(t,x)} \cdot v(t,x)$, we conclude that 
$$\frac{d}{dt}\I_\psi\pa{f(t),g(t)|f_\infty} \leq -2\lambda_1 \I_{\psi}\pa{f(t),g(t)|f_\infty}$$
$$-\int_{\R^d}\psi^{\prime\prime\prime\prime}\pa{\frac{f(t,x)}{f_\infty(x)}}D^2(x)\abs{u(t,x)}^2\abs{v(t,x)}^2 f_\infty(x)dx$$
$$-2\int_{\R^d}\psi^{\prime\prime\prime}\pa{\frac{f(t,x)}{f_\infty(x)}}D(x) u(t,x)\cdot \nabla D(x) \abs{v(t,x)}^2 f_\infty(x)dx$$
$$-4\int_{\R^d}\psi^{\prime\prime\prime}\pa{\frac{f(t,x)}{f_\infty(x)}}D(x)^2 u(t,x)^T  J_x \pa{v(t,x)} v(t,x) f_\infty(x)dx$$
$$-2\sum_{i,j=1}^d\int_{\R^d}\psi^{\prime\prime}\pa{\frac{f(t,x)}{f_\infty(x)}}f_\infty(x)\Bigg(D(x)\partial_i v_j(t,x)+\frac{1}{2}\partial_i D(x) v_j(t,x)$$
$$+\frac{1}{2}v_i(t,x)\partial_j D(x)-\frac{1}{2}\delta_{i,j}\pa{\nabla D(x)\cdot v(t,x)}\Bigg)^2 dx.$$
Defining the matrices 
$$\X:=\pa{\begin{tabular}{cc}
$2\psi^{\prime\prime}\pa{\frac{f(x)}{f_\infty(x)}}$ & $2\psi^{\prime\prime\prime}\pa{\frac{f(x)}{f_\infty(x)}}$ \\
$2\psi^{\prime\prime\prime}\pa{\frac{f(x)}{f_\infty(x)}}$ & $\psi^{\prime\prime\prime\prime}\pa{\frac{f(x)}{f_\infty(x)}}$
\end{tabular}}$$
and
$$\Y:=\pa{\begin{tabular}{cc}
$\alpha(x)$ & $\beta(x)$\\
$\beta(x)$ & $\gamma(x)$
\end{tabular}
}$$
where
$$\alpha(x):=\sum_{i,j=1}^d \pa{D(x)\partial_i v_j(x)+\frac{1}{2}\partial_i D(x) v_j(x)+\frac{1}{2}v_i(x)\partial_j D(x)-\frac{1}{2}\delta_{i,j}\pa{\nabla D(x)\cdot v(x)}}^2,$$
$$\beta(x):=D(x)^2 u(x)^T J_x\pa{v(x)}v(x)+\frac{1}{2}D(x)\abs{v(x)}^2\pa{\nabla D(x)\cdot u(x)},$$
$$\gamma(x):=D(x)^2 \abs{u(x)}^2\abs{v(x)}^2$$
we find that 
$$\frac{d}{dt}\I_\psi\pa{f(t),g(t)|f_\infty} \leq -2\lambda_1 \I_{\psi}\pa{f(t),g(t)|f_\infty} - \int_{\R^d}\tr\pa{\X(t)\Y(t)}f_\infty(x)dx.$$
\amit{Proving} that $\X$ and $\Y$ are non-negative definite will imply that
$$\frac{d}{dt}\I_\psi\pa{f(t),g(t)|f_\infty} \leq -2\lambda_1 \I_{\psi}\pa{f(t),g(t)|f_\infty},$$
which will show the desired result.\\
The fact that $\X$ is non-negative definite is exactly as in our proof of Theorem \ref{thm:Fisher_decay} and follows from the fact that $\psi$ is a generating function of an admissible relative entropy.\\
The structure of $\Y$ is more complicated than what we've dealt with in our previous proof, as well as more complicated than the settings of \cite{AMTU01}. Inspired by  Lemma 2.3 in \cite{ArCaJuL08} we define
$${\bf Z}={\bf Z}^T:=D(x)J_x\pa{v(x)}+\frac{1}{2}\nabla D(x)\otimes v(x) + \frac{1}{2}v(x)\otimes\nabla D(x)-\frac{\nabla D(x)\cdot v(x)}{2}\II,$$
and
$${\bf W}:=D(x)u(x)\otimes v(x).$$
A straight forward computation shows that
$$\tr{\pa{{\bf Z}^2}}=\tr{\pa{{\bf Z} {\bf{Z}}^T}}=\alpha(x),\quad \tr{\pa{{\bf W} {\bf{W}}^T}}=\gamma(x)$$ 
and
$$\tr{\pa{{\bf{W}}{\bf{Z}}^T}}=\tr{\pa{{\bf{W}}{\bf{Z}}}}=\beta(x).$$
We can conclude that
$$\Y=\pa{\begin{tabular}{cc}
$\tr{\pa{{\bf Z}{\bf Z}^T}}$ & $\tr{\pa{{\bf Z}{\bf W}^T}}$\\
$\tr{\pa{{\bf W}{\bf Z}^T}}$ & $\tr{\pa{{\bf W}{\bf W}^T}}$
\end{tabular}}.$$
Since
$$\pa{\begin{tabular}{cc}
${\bf Z}{\bf Z}^T$ & ${\bf Z}{\bf W}^T$\\
${\bf W}{\bf Z}^T$ & ${\bf W}{\bf W}^T$
\end{tabular}} \geq 0$$
we use the positivity of the partial traces result from \cite{ArCaJuL08} (Lemma 2.4) and conclude that $\Y$ is non-negative. The proof is now complete.
\end{proof}

\section{Additional Proofs}\label{app:b}
\amit{This} Appendix \amit{includes} additional proofs which were deferred in order to assist in the flow of the presented work. Amongst other proofs, we will provide proofs for Theorem \ref{thm:interpolation_general} and Lemma \ref{lem:hypercontractivity}.\\ 

To prove Theorem \ref{thm:interpolation_general} we require the following lemma, which is a generalisation of formula (2.21) from \cite{AMTU01}:
\begin{lemma}\label{lem:propertiesof_g}
	Let $\psi\in C^4(\R_+)$ be a strictly convex function such that $\psi^{\prime\prime}$ has a singularity at zero and $\frac{1}{\psi^{\prime\prime}}$ is concave\footnote{This requirement is equivalent to condition \eqref{eq:psicondition}.}. Then, for any $R>0$ 
	\begin{equation}\label{eqapp:propertiesof_g}
		\psi^{\prime\prime}(y) \leq \begin{cases}
			\frac{R\psi^{\prime\prime}(R)}{y} & 0<y<R \\
			\psi^{\prime\prime}(R) & y\geq R.
		\end{cases}
	\end{equation}
In particular for any fixed $R>0$ and $y\in\R_+$
\begin{equation}\label{eqapp:psi_bound}
	\psi^{\prime\prime}(y) \leq 
		\frac{R\psi^{\prime\prime}(R)}{y} +
		\psi^{\prime\prime}(R).
\end{equation}
\end{lemma}
\begin{proof}
	We start by defining the function 
	$$\phi(y):=\begin{cases}\frac{1}{\psi^{\prime\prime}(y)} & y>0, \\ 0 & y=0.\end{cases}$$ 
	The conditions on $\psi$ ensure that $\phi$ is twice differentiable on $\R_+$ and continuous on $[0,\infty)$. Moreover, as $\phi$ is concave due to our assumptions we see that for any $y_0>0$ 
	\begin{equation}\label{eq:concavity}
		\phi(y) \leq \phi(y_0)+\phi^\prime(y_0)(y-y_0),\qquad \forall y\geq 0.
	\end{equation}
	Since $\phi(y)>0$ for all $y>0$, we conclude that we must have that $\phi^\prime(y) \geq 0$ for all $y>0$ \footnote{Indeed, if there exists $y_0>0$ such that $\phi^\prime(y_0)<0$ then for any $\epsilon>0$
		$$\phi\pa{y_0+\epsilon-\frac{\phi\pa{y_0}}{\phi^{\prime}\pa{y_0}}} \leq \phi\pa{y_0}+\phi^\prime(y_0) \pa{\epsilon- \frac{\phi\pa{y_0}}{\phi^\prime\pa{y_0}}}=\epsilon \phi^\prime\pa{y_0}<0.$$}. Thus, $\phi$ is non-decreasing and we have that for any $y\geq R$ 
	$$\psi^{\prime\prime}(y)=\frac{1}{\phi(y)} \leq \frac{1}{\phi(R)}=\psi^{\prime\prime}(R).$$
	 To show the bound on $\psi^{\prime\prime}(y)$ for $0<y<R$ we notice that due to the concavity of $\phi$
	$$\phi(y)=\phi\pa{\frac{y}{R}R+\frac{R-y}{R}\cdot 0} \geq \frac{y}{R}\phi(R) + \pa{1-\frac{y}{R}}\phi(0)= \frac{y}{R}\phi(R),$$
	 for any $y<R$. The proof is thus complete.
\end{proof}
\begin{proof}[Proof of Theorem \ref{thm:interpolation_general}]
	Using Lemma \ref{lem:propertiesof_g} we find that for any $R>0$
	$$\mathcal{I}^{\bf{P}}_{\psi}\pa{f,g|f_\infty} \leq 
	R\psi^{\prime\prime}(R)\mathcal{I}^{\bf{P}}_{1}(f,g|f_\infty)+\psi^{\prime\prime}(R)I^{\bf{P}}_2(g|f_\infty).$$
	Choosing $R=\frac{I^{\bf{P}}_2(g|f_\infty)}{\mathcal{I}^{\bf{P}}_{1}(f,g|f_\infty)}$ and using \eqref{eq:growthofpsiprimeprime} we find that
	\begin{equation}\nonumber
		\mathcal{I}^{\bf{P}}_{\psi}\pa{f,g|f_\infty}\leq 2\psi^{\prime\prime}\pa{\frac{I^{\bf{P}}_2(g|f_\infty)}{\mathcal{I}^{\bf{P}}_{1}(f,g|f_\infty)}}I_2^{\PP}\pa{g|f_\infty} \leq 2\pa{c_1\pa{\frac{I^{\bf{P}}_2(g|f_\infty)}{\mathcal{I}^{\bf{P}}_{1}(f,g|f_\infty)}}^{-\alpha}+ c_2}I^{\bf{P}}_2(g|f_\infty),
	\end{equation}
	which concludes the proof.
\end{proof}
Next we consider the proof of Lemma \ref{lem:hyper_conditions}:
\begin{proof}[Proof of Lemma \ref{lem:hyper_conditions}]
	From the definition of $\psi_p$ for $1<p\leq 2$ we have that\footnote{Note that the mass of $f_\infty$ is one, but we don't know anything about the mass of $\abs{f}$.}
	\begin{equation}\label{eq:e_p_controls_psi_p_of_norm}
		\begin{gathered}
			e_p\pa{\abs{f}|f_\infty}=\frac{1}{p\pa{p-1}}\pa{\norm{f}_{L^p\pa{\R^d,f_\infty^{1-p}}}^p-p\pa{\norm{f}_{L^1\pa{\R^d}}-1}-1}\\
			\geq \frac{1}{p\pa{p-1}}\pa{\norm{f}_{L^p\pa{\R^d,f_\infty^{1-p}}}^p-p\pa{\norm{f}_{L^p\pa{\R^d,f_\infty^{1-p}}}-1}-1}\\
			=\psi_p\pa{\norm{f}_{L^p\pa{\R^d,f_\infty^{1-p}}}},	
		\end{gathered}
	\end{equation}
	where we have used the fact that
	\begin{equation}\label{eq:exponential_controls_L_1}
		\begin{gathered}
			\int_{\R^d}\abs{f(x)}dx=\int_{\R^d}\pa{\abs{f(x)}e^{\frac{p-1}{2p}\abs{x}^2}}e^{-\frac{p-1}{2p}\abs{x}^2}dx\\
			\leq \pa{2\pi}^{\frac{d(p-1)}{2p}}\pa{\int_{\R^d}\abs{f(x)}^p e^{\frac{(p-1)\abs{x}^2}{2}}dx}^{\frac{1}{p}}=\norm{f}_{L^p\pa{\R^d,f_\infty^{1-p}}}.
		\end{gathered}
	\end{equation}
	It is straight forward to verify that if $y \geq \pa{2p}^{\frac{1}{p-1}}$ then for any $\alpha>0$
	$$\psi_p\pa{y} \leq \alpha \quad \Rightarrow \quad y^p \leq p(p-1)\alpha +1 +p\pa{y-1} \leq  p(p-1)\alpha +1 + \frac{y^p}{2}$$
	$$\Rightarrow y\leq 2^{\frac{1}{p}}\pa{p(p-1)\alpha +1}^{\frac{1}{p}}.$$
	Thus
	$$\psi_p(y) \leq \alpha\quad \Rightarrow \quad y\leq \max \pa{\pa{2p}^{\frac{1}{p-1}},2^{\frac{1}{p}}\pa{p(p-1)\alpha +1}^{\frac{1}{p}}} \leq \pa{2p}^{\frac{1}{p-1}}\pa{p(p-1)\alpha +1}^{\frac{1}{p}}.$$
	Using the above with \eqref{eq:e_p_controls_psi_p_of_norm} and $\alpha:=e_p\pa{\abs{f} |f_\infty}$ yields the estimation
	\begin{equation}\label{eq:norm_bound_by_e_p}
		\norm{f}_{L^p\pa{\R^d,f_{\infty}^{1-p}}} \leq \pa{2p}^{\frac{1}{p-1}}\pa{p(p-1)e_p\pa{\abs{f}|f_\infty}+1}^{\frac{1}{p}}.
	\end{equation}
	Noticing that for any $\epsilon>0$ we have that
	\begin{equation}\label{eq:estimate_of_eps_exponential_moment}
		\begin{gathered}
			\int_{\R^d}e^{\epsilon \abs{x}^2}\abs{f(x)}dx =\int_{\R^d}e^{-\frac{\pa{p-1}\epsilon}{p} \abs{x}^2}e^{\frac{(2p-1)\epsilon}{p}\abs{x}^2}\abs{f(x)}dx  \\
			\leq \pa{\int_{\R^d}e^{-\epsilon\abs{x}^2}dx}^{\frac{p-1}{p}}\pa{\int_{\R^d}e^{(2p-1)\epsilon \abs{x}^2}\abs{f(x)}^pdx}^{\frac{1}{p}}\\
			=\pa{\frac{\pi}{\epsilon}}^{\frac{d(p-1)}{2p}}\pa{\int_{\R^d}e^{(2p-1)\epsilon \abs{x}^2}\abs{f(x)}^pdx}^{\frac{1}{p}}\\
			=\pa{\frac{\pi}{\epsilon}}^{\frac{d(p-1)}{2p}}\pa{2\pi}^{-\frac{d\pa{2p-1}\epsilon }{p}}\norm{f}_{L^p\pa{\R^d,f_\infty^{-2\pa{2p-1}\epsilon}}},
		\end{gathered}
	\end{equation}
	we attain \eqref{eq:exp_estimation_by_norm_general_function} by choosing $\epsilon=\frac{p-1}{2\pa{2p-1}}$ and using \eqref{eq:norm_bound_by_e_p}. \\
	To show \eqref{eq:exp_estimation_by_entropy_with_unit_mass} we notice that in this particular case \eqref{eq:e_p_controls_psi_p_of_norm} can be replaced with
	\begin{equation}\nonumber
		\begin{gathered}
			e_p\pa{\abs{f}|f_\infty}= 
			\frac{1}{p\pa{p-1}}\pa{\norm{f}_{L^p\pa{\R^d,f_\infty^{1-p}}}^p-1}
		\end{gathered}
	\end{equation}
	or equivalently 
	\begin{equation}\nonumber
		\norm{f}_{L^p\pa{\R^d,f_{\infty}^{1-p}}} = \pa{p(p-1)e_p\pa{\abs{f}|f_\infty}+1}^{\frac{1}{p}}.
	\end{equation} 
	Choosing $\epsilon=\frac{p-1}{2\pa{2p-1}}$ in \eqref{eq:estimate_of_eps_exponential_moment} with the above equality yields the desired result.
\end{proof}

Lastly, we focus our attention on the proof of Lemma \ref{lem:hypercontractivity}. 
\begin{proof}[Proof of Lemma \ref{lem:hypercontractivity}]
	Starting with \eqref{eq:exactsolution} we find that 
	$$\abs{g\pa{t,x} } \leq \frac{1}{\pa{2\pi}^{\frac{d}{2}}\sqrt{\det \W(t)}}\int_{\R^d}e^{-\frac{1}{2}\pa{x-e^{-\C t}y}^T \W(t)^{-1}\pa{x-e^{-\C t}y}} \abs{g_0(y)}dy,$$
	where the left hand side can be identified as the solution for the Fokker-Planck equation \eqref{eq:fokkerplanck} with initial datum $\abs{g_0}$. From this point the proof of the $L^2-$bound follows exactly like part $(i)$ of Theorem 4.3 of \cite{AEW18}\footnote{In fact, denoting by $e^{Lt}$ the Fokker-Planck semigroup we have that 
$$\norm{g(t)}_{L^2\pa{\R^d,f_\infty^{-1}}}\leq \norm{e^{Lt}\abs{g_0}}_{L^2\pa{\R^d,f_\infty^{-1}}}\leq \pa{\frac{8}{3}}^d\pa{\int_{\R^d}e^{\epsilon \abs{x}^2}\abs{g_0(x)}dx}^2$$
according to part $(i)$ of Theorem 4.3 of \cite{AEW18}.}. Consequently, we will only focus on the bound on the Fisher information. \\
	Given a fixed $\epsilon_1>0$, to be chosen later, consider $t\geq \widetilde{t}_1\pa{\epsilon_1}:=\max\pa{\widetilde{t}\pa{\epsilon_1},\widehat{t}\pa{\epsilon_1}}$, where $\widetilde{t}\pa{\epsilon_1}$ and $\widehat{t}\pa{\epsilon_1}$ are as those given in Lemma \ref{lem:proved_propertires_for_hyper_hypo}. Since
	$$\nabla \pa{\frac{g(\amit{t,x})}{f_\infty(x)}}=\frac{\nabla g(\amit{t,x})+xg(\amit{t,x})}{f_\infty(x)}$$
	we combine \eqref{eq:exactsolution},\eqref{eq:exactsolution_grad} together with \eqref{eq:hyperproof0}, \eqref{eq:hyperproofI}, and \eqref{eq:hyperproofII} to conclude that for any such $t$
	\begin{equation}\nonumber
		\begin{gathered}
			\abs{\nabla g(t,x) + xg(t,x)} =\frac{1}{\pa{2\pi}^{\frac{d}{2}}\sqrt{\det \W(t)}}\\
		\times 	\Big |\int_{\R^d}e^{-\frac{1}{2}\pa{x-e^{-\C t}y}^T \W(t)^{-1}\pa{x-e^{-\C t}y}}
			\pa{\W(t)^{-1}\pa{x-e^{-\C t}y}-x} g_0(y)dy\Big | \\
			\leq \frac{1}{\pa{2\pi\pa{1-\epsilon_1}}^{\frac{d}{2}}}\Big |\int_{\R^d}e^{-\frac{1}{2}\pa{x-e^{-\C t}y}^T \W(t)^{-1}\pa{x-e^{-\C t}y}}
			\pa{\pa{\W(t)^{-1}-{\bf{I}}}\pa{x-e^{-\C t}y}-e^{-\C t}y} g_0(y)dy\Big |
		\end{gathered}
	\end{equation}
	\begin{equation}\nonumber
		\begin{gathered}
			\leq \frac{1}{\pa{2\pi\pa{1-\epsilon_1}}^{\frac{d}{2}}}\int_{\R^d}e^{-\frac{1}{2}\pa{1-\epsilon_1}\abs{x-e^{-\C t}y}^2}
			\pa{\epsilon_1\abs{x-e^{-\C t}y}+\abs{e^{-\C t}y}} \abs{g_0(y)}dy.
		\end{gathered}
	\end{equation}
	Using Minkowski's integral inequality 
	\begin{equation}\nonumber
		\begin{split}
			\left(\int_{\R^d}\abs{\int_{\R^d}F(y,x)d\mu_1( y)}^q  d\mu_2(x) \right)^{\frac{1}{q}}
			\leq \int_{\R^d} \pa{ \int_{\R^d}\abs{F(y,x)}^q d\mu_{2}( x) }^{\frac{1}{q}}d\mu_1(y),
		\end{split}
	\end{equation}
	we see that 
	\begin{equation}\nonumber
		\begin{gathered}
			I^{{\bf{I}}}_2\pa{g(t)|f_\infty} = \int_{\R^d}\abs{\nabla\pa{\frac{g(t,x)}{f_\infty(x)}}}^2 f_\infty(x)dx \\
			\leq  \frac{1}{\pa{2\pi\pa{1-\epsilon_1}}^d}\int_{\R^d}\pa{\int_{\R^d}\underbrace{e^{-\frac{1}{2}\pa{1-\epsilon_1}\abs{x-e^{-\C t}y}^2}
				\pa{\epsilon_1\abs{x-e^{-\C t}y}+\abs{e^{-\C t}y}}}_{:=F(y,x)^2} \underbrace{\abs{g_0(y)}dy}_{:=d\mu_2(y)}}^2 \underbrace{f_\infty^{-1}(x)dx}_{:=d\mu_1(x)} \\
			\leq \frac{1}{\pa{2\pi\pa{1-\epsilon_1}}^d}\pa{\int_{\R^d}\pa{\int_{\R^d}e^{-\pa{1-\epsilon_1}\abs{x-e^{-\C t}y}^2}
					\pa{\epsilon_1\abs{x-e^{-\C t}y}+\abs{e^{-\C t}y}}^2 f^{-1}_{\infty}(x)dx}^{\frac{1}{2}} \abs{g_0(y)}dy}^2
		\end{gathered}
	\end{equation}
	\begin{equation}\label{eq:hyper_to_referred_by_hypo}
		\begin{gathered}
			\leq \frac{2}{\pa{2\pi\pa{1-\epsilon_1}^2}^{\frac{d}{2}}}\pa{\int_{\R^d}\pa{\int_{\R^d}e^{-\pa{1-\epsilon_1}\abs{x-e^{-\C t}y}^2}e^{\frac{\abs{x}^2}{2}}
					\pa{\epsilon_1^2\abs{x-e^{-\C t}y}^2+\abs{e^{-\C t }y}^2}dx}^{\frac{1}{2}} \abs{g_0(y)}dy}^2\\
			\leq \frac{2}{\pa{2\pi\pa{1-\epsilon_1}^2}^{\frac{d}{2}}}\pa{\int_{\R^d}\sqrt{\epsilon_1+\abs{e^{-\C t }y}^2}\pa{\int_{\R^d}e^{-\pa{1-2\epsilon_1}\abs{x-e^{-\C t}y}^2}e^{\frac{\abs{x}^2}{2}}
					dx}^{\frac{1}{2}} \abs{g_0(y)}dy}^2
		\end{gathered}
	\end{equation}
	where we have used the fact that 
	\begin{equation}\label{eq:useful_exponential_inequality}
		\epsilon_1 \abs{x-e^{-\C t}y}^2 \leq \frac{1}{e}e^{\epsilon_1\abs{x-e^{-\C t}y}^2}\leq e^{\epsilon_1\abs{x-e^{-\C t}y}^2}.
	\end{equation}
	A straightforward calculation gives
	$$\int_{\R^d}e^{-\pa{1-2\epsilon_1}\abs{x-e^{-\C t}y}^2}e^{\frac{\abs{x}^2}{2}}
	dx=\pa{\frac{\pi}{1-2\epsilon_1 - \frac{1}{2}}}^{\frac{d}{2}}e^{\frac{\frac{1}{2}(1-2\epsilon_1)}{1-2\epsilon_1-\frac{1}{2}}\abs{e^{-\C t}y}^2}=\pa{\frac{2\pi}{1-4\epsilon_1 }}^{\frac{d}{2}}e^{\frac{1-2\epsilon_1}{1-4\epsilon_1}\abs{e^{-\C t}y}^2}$$
	and since $t\geq\widetilde{t}\pa{\epsilon_1}$ we can further estimate the above by using \eqref{eq:hyperproofIII} and find that
	$$\int_{\R^d}e^{-\pa{1-2\epsilon_1}\abs{x-e^{-\C t}y}^2}e^{\frac{\abs{x}^2}{2}}
	dx\leq \pa{\frac{2\pi}{1-4\epsilon_1 }}^{\frac{d}{2}}e^{\frac{1-2\epsilon_1}{1-4\epsilon_1}\epsilon_1^2\abs{y}^2}.$$
	Plugging this into our preceding estimate and using \eqref{eq:hyperproofIII} again we find that
	\begin{equation}\nonumber
		\begin{gathered}
			I^{{\bf{I}}}_2\pa{g(t)|f_\infty} \leq \frac{2\epsilon_1}{\pa{\pa{1-4\epsilon_1}\pa{1-\epsilon_1}^2}^{\frac{d}{2}}}\pa{\int_{\R^d}\sqrt{1+\epsilon_1\abs{y}^2}e^{\frac{1-2\epsilon_1}{1-4\epsilon_1}\frac{\epsilon_1^2}{2}\abs{y}^2}\abs{g_0(y)}dy}^2\\
			\leq \frac{2\epsilon_1}{\pa{\pa{1-4\epsilon_1}\pa{1-\epsilon_1}^2}^{\frac{d}{2}}}\pa{\int_{\R^d}\pa{1+\epsilon_1\abs{y}^2}e^{\frac{1-2\epsilon_1}{1-4\epsilon_1}\frac{\epsilon_1^2}{2}\abs{y}^2}\abs{g_0(y)}dy}^2\\
			\leq \frac{2\epsilon_1}{\pa{\pa{1-4\epsilon_1}\pa{1-\epsilon_1}^2}^{\frac{d}{2}}}\pa{\int_{\R^d}e^{\pa{\frac{1-2\epsilon_1}{1-4\epsilon_1}\frac{\epsilon_1^2}{2}+\epsilon_1}\abs{y}^2}\abs{g_0(y)}dy}^2.
		\end{gathered}
	\end{equation}
	Choosing $\epsilon_1\pa{\epsilon}$ such that $\epsilon_1 \leq \frac{1}{8}$ and
	$$\frac{1-2\epsilon_1}{1-4\epsilon_1}\frac{\epsilon_1^2}{2}+\epsilon_1 \leq \epsilon$$
	yields \eqref{eq:hypercontractivity} as 
	\begin{equation}\nonumber
		\frac{2\epsilon_1}{\pa{\pa{1-4\epsilon_1}\pa{1-\epsilon_1}^2}^{\frac{d}{2}}} \leq 2^{\frac{3d}{2}}.
	\end{equation}
	This concludes the proof.
\end{proof}


\bibliographystyle{plain}

\begin{thebibliography}{99}

\bibitem{AAS} 
F. Achleitner, A. Arnold, D. St{\"{u}}rzer, \textit{Large-Time Behavior in Non-Symmetric Fokker-Planck Equations}.
Rivista di Matematica della Universit\`a di Parma \textbf{6} (2015) 1--68. 

\bibitem{AE}
A. Arnold, J. Erb, \textit{Sharp Entropy Decay for Hypocoercive and Non-Symmetric Fokker-Planck Equations with Linear Drift}. Preprint. https://arxiv.org/abs/1409.5425

\bibitem{ArCaJuL08}
A. Arnold, E. Carlen, Q. Ju, \textit{Large-time behavior of non-symmetric Fokker-Planck type equations}, Communications in Stochastic Analysis \textbf{2} (2008) 153--175.

\bibitem{ArCaMa10}
A. Arnold,  J. A. Carrillo, C. Manzini, \textit{Refined long-time asymptotics for some polymeric fluid flow models}, Comm. Math. Sc. \textbf{8} (2010), 763--782.  

\bibitem{ArCaMa17}
A. Arnold, J. A. Carillo, D. Matthes, \textit{Nonlinear diffusion equation: full characterization of entropies}, Preprint (2022).

\bibitem{AEW18}
A. Arnold, A. Einav, T. W\"ohrer, \textit{On the Rate of Decay to Equilibrium in Degenerate and Defective Fokker-Planck Equations}. J. Differential Equations, \textbf{264}, No. 11 (2018), 6843--6872. 

\bibitem{AEW23}
A. Arnold, A. Einav, T. W\"ohrer, \textit{Sharp Decay of the Fisher Information for Degenerate Fokker--Planck Equations}. Preprint (2023)
 
\bibitem{AMTU01}
A. Arnold, P. Markowich, G. Toscani, A. Unterreiter, \textit{On convex Sobolev inequalities and the rate of convergence to equilibrium for Fokker--Planck type equations}, Communications in Partial Differential Equations \textbf{26} (2001), 43--100.

\bibitem{ASS}
A. Arnold, B. Signorello, C. Schmeiser, \textit{Propagator norm and sharp decay estimates for Fokker-Planck equations with linear drift}, Comm. Math. Sc. \textbf{20}, No. 4 (2022) 1047-1080



\bibitem{BaEmD85}
D. Bakry, M. Emery, \textit{Diffusions hypercontractives}, S\'eminaire de probabilt\'es de Strasbourg \textbf{19} (1985), 177--206.

\bibitem{BaEmH84}
D. Bakry, M. Emery, \textit{Hypercontractivit\'e de semi-groupes de diffusion}, C. R. Acad. Sci. Paris S\'er. I Math. \textbf{299} (1984), 775--778.

\bibitem{BaEmI85}
D. Bakry, M. Emery, \textit{In\'egalit\'es de Sobolev pour un semi-group sym\'etrique}, C. R. Acad. Sci. Paris S\'er. I Math. \textbf{301} (1985), 411--413.

\bibitem{BGL}
D. Bakry, I. Gentil, M. Ledoux, \textit{Analysis and geometry of Markov diffusion operators}, Springer Science \& Business Media, (2013).


\bibitem{BaB13}
F. Baudoin, \textit{Bakry-Emery meet Villani}, J. Funct. Anal. \textbf{273}, No. 7  (2017), 2275--2291


\bibitem{Be89}
W. Beckner, \textit{A generalized Poincar\'e inequality for Gaussian measures}, Proc. Amer. Math. Soc. \textbf{105}, No. 2 (1989), 397--400.

\bibitem{BoGe10}
F. Bolley, I. Gentil, \textit{Phi-entropy inequalities for diffusion semigroups}, J. Math. Pures Appl. \textbf{93} (2010) 449--473.

\bibitem{DeH06}
L. Desvillettes, \textit{Hypocoercivity: The example of linear transport}, Contemporary Mathematics \textbf{409}: Recent Trends in Partial Differential Equations (2006), 33--53.

\bibitem{DeVi01}
L. Desvillettes, C. Villani, \textit{On the trend to global equilibrium in spatially inhomogeneous
entropy-dissipating systems: The linear Fokker-Planck equation}, Comm. Pure Appl. Math. \textbf{54} (2001), 1--42.

\bibitem{DoMoScH10}
J. Dolbeault, C. Mouhot, C. Schmeiser, \textit{Hypocoercivity for linear kinetic equations conserving mass}, Trans. AMS. \textbf{367}, No. 6 (2015), 3807--3838.




\bibitem{EN00}
K. J. Engel, R. Nagel, \textit{One-Parameter Semigroups for Linear Evolution Equations}, Springer 2000.

\bibitem{FMP06}
F. Filbet, C. Mouhot, L. Pareschi, \textit{Solving the Boltzmann equation in $N\,\log_2\,N$}, SIAM J. Sci. Comput. \textbf{28} (2006), 1029--1053.

\bibitem{GaMiS13}
S. Gadat, L. Miclo, \textit{Spectral decompositions and $L^2$-operator norms of toy hypocoercive semi-groups}, Kinetic and Related Models, \textbf{2} (2013), 317--372.

\bibitem{GKZ04}
A.N. Gorban, I.V. Karlin, A.Yu. Zinovyev, \textit{Constructive methods of invariant manifolds for kinetic problems}, Phys. Rep.  \textbf{396} (2004), 197--403.


\bibitem{HeNiFP05}
B. Helffer, F. Nier, \textit{Hypoelliptic estimates and spectral theory for Fokker-Planck operators and Witten Laplacians}, Lecture Notes in Mathematics 1862 (2005)

\bibitem{HeF07}
F. H\'erau, \textit{Short and long time behavior of the Fokker-Planck equation in a confining potential and applications}, Journal of Functional Analysis \textbf{244} (2007), 95--118.

\bibitem{HeNi04}
F. H\'erau, F. Nier, \textit{Isotropic Hypoellipticity and Trend to
Equilibrium for the Fokker-Planck Equation
with a High-Degree Potential}, Arch. Rational Mech. Anal. \textbf{171} (2004), 151--218.

\bibitem{Hi70}
C.D. Hill, \textit{A sharp maximum principle for degenerate elliptic-parabolic equations}, Indiana Univ. Math. J. \textbf{20} (1970), 213--229.

\bibitem{Ho67}
L. H\"ormander, \textit{Hypoelliptic second order differential equations}, Acta Math. \textbf{119} (1969), 147--171.

\bibitem{Ho95}
L. H\"ormander, \textit{Symplectic classification of quadratic forms, and general Mehler formulas}, Math. Z. \textbf{219} (1995), 413--449.

\bibitem{HoJoT91}
R.A. Horn, C.R. Johnson, \textit{Topics in Matrix Analysis}, Cambridge University Press (1991).

\bibitem{KoH73}
J.J. Kohn, \textit{pseudo-differential operators and hypoellipticity}, Proc. Symp. Pure Math. \textbf{23} (1973), 61--69.

\bibitem{LNP13}
T. Leli\`evre, F. Nier, G.A. Pavliotis, \textit{Optimal Non-reversible Linear Drift for the Convergence
to Equilibrium of a Diffusion}, J. Stat. Phys. \textbf{152} (2013), 237--274.

\bibitem{MPP02}
G. Metafune, D. Pallara, E. Priola, \textit{Spectrum of Ornstein--Uhlenbeck operators in $L^p$ spaces
with respect to invariant measures}, J. Funct. Anal. \textbf{196} (2002), 40--60.

\bibitem{MoH19}
P. Monmarch\'e, \textit{Generalized $\Gamma$ Calculus and Application to Interacting Particles on a Graph}, Potential Anal. \textbf{3} (2019), 439--466.

\bibitem{MoNeQ06}
C. Mouhot, L. Neumann, \textit{Quantitative perturbative study of convergence to equilibrium for collisional kinetic models in the torus}, Nonlinearity \textbf{19} (2006), 969--998.

\bibitem{Nelson}
E. Nelson, \textit{Dynamical theories of Brownian motion}, Princeton University Press (1967)

\bibitem{Ni14}
F. Nier, \textit{Accurate Estimates for the Exponential Decay of Semigroups with Non-Self-Adjoint Generators}; in: 
O.N. Kirillov, D.E. Pelinovsky (eds.), \textit{Nonlinear Physical Systems: Spectral Analysis, Stability and Bifurcations}, ISTE \& Wiley (2014).

\bibitem{OPPS12}
M. Ottobre, G.A. Pavliotis, K. Pravda-Starov, \textit{Exponential return to equilibrium for hypoelliptic
quadratic systems}, J. Funct. Anal. \textbf{262} (2012), 4000--4039.

\bibitem{OPPS15}
M. Ottobre, G.A. Pavliotis, K. Pravda-Starov, \textit{Some remarks on degenerate hypoelliptic Ornstein--Uhlenbeck
operators}, J. Math. Anal. Appl. \textbf{429} (2015), 676--712.

\bibitem{PaS83}
A. Pazy, \textit{Semigroups of Linear Operators and Applications to Partial Differential Equations}, Springer (1983).

\bibitem{RiFP89}
H. Risken, \textit{The Fokker-Planck equation. Methods of solution and applications.}, Springer-Verlag (1989).


\bibitem{SnZaN70}
J. Snyders, M. Zakai, \textit{On nonnegative solutions of the equation $AD+DA'=-C$}, SIAM J. Appl. Math. \textbf{18} (1970), 704--715.




\bibitem{ViH06}
C. Villani, \textit{Hypocoercivity}, Memoirs of the American Mathematical Society \textbf{202} (2009).




\end{thebibliography}

\end{document}